\documentclass[12pt,reqno]{amsart}
\usepackage{amsmath, amsthm, amsopn, amssymb, microtype}
\usepackage{mathrsfs} 
\usepackage{mathscinet}
\usepackage{bbm}
\usepackage{enumerate, tikz, etoolbox, intcalc, geometry, caption, subcaption, afterpage}
\usepackage[english]{babel}
\usepackage{enumitem}
\usepackage[colorlinks=true,urlcolor=blue,pdfborder={0 0 0}]{hyperref}
\hypersetup{linkcolor=[rgb]{0,0,0.6}}
\hypersetup{citecolor=[rgb]{0,0.6,0}}
\usepackage[nameinlink]{cleveref}
\crefname{theorem}{Theorem}{Theorems}
\crefname{thm}{Theorem}{Theorems}
\crefname{mainthm}{Theorem}{Theorems}
\crefname{conj}{Conjecture}{Theorems}
\crefname{lemma}{Lemma}{Lemmas}
\crefname{lem}{Lemma}{Lemmas}
\crefname{remark}{Remark}{Remarks}
\crefname{prop}{Proposition}{Propositions}
\crefname{defn}{Definition}{Definitions}
\crefname{corollary}{Corollary}{Corollaries}
\crefname{cor}{Corollary}{Corollaries}
\crefname{section}{Section}{Sections}
\crefname{figure}{Figure}{Figures}
\crefname{quest}{Question}{Questions} 

\theoremstyle{plain}
\newtheorem{theorem}{Theorem}[section]
\newtheorem{lemma}[theorem]{Lemma}
\newtheorem{corollary}[theorem]{Corollary}
\newtheorem{proposition}[theorem]{Proposition}
\newtheorem{prop}[theorem]{Proposition}

\theoremstyle{definition}
\newtheorem{definition}[theorem]{Definition}

\newcommand{\Z}{\mathbb{Z}}
\newcommand{\Sub}{\mathrm{Sub}}
\newcommand{\dist}{\mathrm{dist}}
\newcommand{\stab}{\mathit{stab}}

\newcommand{\Ker}{\mathit{Ker}}

\newcommand{\cL}{\mathcal{L}}

\newcommand{\PTiling}{\mathit{PTilings}}

\title{ An embedding theorem for  multidimensional subshifts}
\author{Tom Meyerovitch}
\address{Ben-Gurion University of the Negev.
	Department of Mathematics.
	Beer-Sheva, 8410501, Israel. 
	{\tt mtom@bgu.ac.il}
}

\begin{document}

\maketitle

\begin{abstract}
Krieger's embedding theorem  provides  necessary and sufficient conditions for an arbitrary subshift to embed in a given topologically mixing $\mathbb{Z}$-subshift of finite type. For certain families of $\mathbb{Z}^d$-subshifts  of finite type, Lightwood characterized  the \emph{aperiodic}  subsystems. In the current paper we prove a new embedding theorem 
for a class of subshifts of finite type over any  countable abelian group. 
Our theorem provides necessary and sufficient conditions for an arbitrary subshift $X$ to embed inside a given subshift of finite type $Y$ that satisfies a certain natural condition. For the particular case of $\Z$-subshifts, our new theorem coincides with Krieger's theorem.
Our result gives the first  complete characterization of the subsystems of the multidimensional  full shift $Y= \{0,1\}^{\mathbb{Z}^d}$. 

The natural condition on the target subshift $Y$, introduced explicitly for the first time in the current paper, is called the map extension property.
It was introduced implicitly by Mike Boyle in the early 1980's for $\mathbb{Z}$-subshifts, 
and is closely related to the notion of an absolute retract, introduced by Borsuk in the 1930's.
The class of subshifts over a countable abelian group that admit  the map extension property coincides with the class of contractible subshifts of finite type, a notion introduced in recent work of Leo Poirier and Ville Salo. A $\mathbb{Z}$-subshift has the map extension property if and only if it is a topologically mixing subshift of finite type. 
We show that various natural examples of $\mathbb{Z}^d$ subshifts  of finite type satisfy the map extension property, hence our embedding theorem applies for them. These include any subshift of finite type with a safe symbol and the $k$-colorings of $\mathbb{Z}^d$ with $k \ge 2d+1$. We also establish a new theorem regarding lower entropy factors of multidimensional subshifts, that extends Boyle's lower entropy factor theorem from the one-dimensional case. 
\end{abstract}

\section{Introduction and statement of the main results}

Krieger's embedding theorem   \cite{MR693975} provides simple necessary and sufficient conditions for an arbitrary subshift to embed in a given topologically mixing $\mathbb{Z}$-subshift of finite type.

Here is a formulation of Krieger's embedding theorem \cite{MR693975}:

\begin{theorem}[Krieger's embedding theorem \cite{MR693975}]\label{thm:krieger}
Let $X$ be an arbitrary $\Z$-subshift and let $Y$ be a topologically mixing $\Z$-shift of finite type. Then $X$ embeds in $Y$ if and only if either $X$ is topologically conjugate to $Y$ or $h(X) < h(Y)$ and for every $n \in \mathbb{N}$
the number of points in $X$ that have least period $n$ is less than or equal to the number of points in $Y$ that have least period $n$.
\end{theorem}

Because Krieger's embedding theorem is a fundamental result in  symbolic dynamics, it was quite natural to seek a multidimensional version.
Lightwood  carried out  significant steps in this direction \cite{MR1972240, MR2085910} about a quarter of a century ago.
In particular, Lightwood characterized  the  \emph{aperiodic} $\Z^d$-subshifts  that can embed in a certain class of $\Z^d$-subshifts of finite type.
A subshift $X$ is aperiodic if the stabilzer of every point in $X$ is trivial.
Lightwood proved the following theorem: 

\begin{theorem}[Lightwood's  embedding theorem for square-filling mixing $\Z^2$ -SFTs \cite{MR1972240,MR2085910}]
    Let $X$ be an  aperiodic $\Z^2$ subshift and let $Y$ be a square-filling mixing $\Z^2$-SFT. Then $X$ embeds in $Y$ if and only if $h(X) < h(Y)$.
\end{theorem}

We refrain from defining ``square-filling mixing'', as this notion does not play a role in the current paper. We mention however, that a $\mathbb{Z}$-SFT is square-filling mixing if and only if it is topologically mixing. For $\mathbb{Z}^2$-SFTs ``square-filling mixing'' is strictly stronger than topological mixing, and strictly weaker than being strongly  irreducible (see \Cref{def:SI}). 

The central missing ingredient in Lightwood's work is a method of handling points with non-trivial stabilizers in the domain $X$.
In his paper \cite{MR2085910} Lightwood announced: ``The techniques developed here play a central role in the proof of
an embedding theorem for general $\Z^2$ subshifts into square-filling mixing SFTs which will
be carried out in a subsequent paper''. 
To the best of our knowledge, no such result has been published so far. In this paper,
building on previous work,  including  Krieger's and Lightwood's, we obtain a complete characterization of the subsystems for a natural class of $\Z^d$-SFTs.

In order to defer some definitions, we first state a particular case of our main result.
Given a $\Z^2$ subshift $X$ and $v \in \Z^2$, let 
\[ X_v = \{ x\in X~:~ \sigma_v(x)=x \}.\]
Similarly, for a pair of linearly independent vectors $v,w \in \Z^2$ let $X_{v,w}$ denote the set of points in $x$ whose stablizer is equal to the integer linear combinations of $v$ and $w$. It is easily verified that for any $\Z^2$ subshift $X$ and any pair of linearly independent vectors $v,w \in \Z^2$, the space $X_{v,w}$ is  a finite set.

Recall that a vector $v \in \mathbb{Z}^2$ is \emph{primitive} if it is not a multiple of some other vector in $w \in \mathbb{Z}^2$.
If $v \in \mathbb{Z}^2$ is a primitive vector and $n \in \mathbb{N}$, then the group $\mathbb{Z}^2/\langle n v \rangle$ is isomorphic to $\mathbb{Z} \times \left( \mathbb{Z} / n \mathbb{Z}\right)$, and so  $X_v$ can naturally be viewed as a $\Z \times (\Z/n\Z)$-subshift.

\begin{theorem}\label{thm:embedding_thm_Z2_SFT}
 A $\mathbb{Z}^2$-subshift $X$ embeds in $A^{\mathbb{Z}^2}$ if and only if  either  $X$ is topologically conjugate to $A^{\mathbb{Z}^2}$ or $h(X) < \log |A|$ and both of the following conditions hold:
 \begin{enumerate}
     \item For every primitive vector $v \in \mathbb{Z}^2$  and every $n \in \mathbb{N}$ either  $X_{nv}$ is topologically conjugate to the full-shift $A^{\Z \times (\Z / n \Z)}$ or the topological entropy of  the $\Z \times (\Z/n\Z)$-action on  $X_{nv}$ is strictly less than $\log(A)$.
     \item For every pair of independent vectors $v,w \in \Z^2$ the number of points in $X_{v,w}$
     does not exceed the number of points in $A^{\Z^2}$ whose stabilizer is equal to the integer linear combinations of $v$ and $w$.
     
 \end{enumerate}
\end{theorem}

The statement of \Cref{thm:embedding_thm_Z2_SFT} suggests that subshifts over more general abelian groups must be considered even when one is initially only interested in the characterization of subsystems of $\Z^d$-SFTs.

In this paper, by  %
a \emph{map} between  $\Gamma$-subshifts $X$ and $Y$ we mean a continuous, $\Gamma$-equivariant function from $X$ into $Y$. An \emph{embedding} of $X$ into $Y$ is an injective map from $X$ into $Y$. An \emph{isomorphism}  of $X$ and $Y$ is a bijective map between $X$ and $Y$. We
 use the notation $X \hookrightarrow Y$ to indicate that $X$ embeds continuously and $\Gamma$-equivariantly into $Y$.

The following is the main result of this paper:
\begin{theorem}\label{thm:relative_embedding_thm}
Let $\Gamma$ be a countable abelian group and 
let $X,Y$ be $\Gamma$-subshifts. Suppose that  $Y$ has the map extension property. 
Let $Z \subseteq X$ be a (possibly empty) $\Gamma$-subshift contained in $X$, 
and let $\hat \rho:Z \to Y$
 be an embedding. 
Then $\hat \rho$ extends to an embedding $\rho:X \to Y$ if and only if for every subgroup $\Gamma_0 \le  \Gamma$ one of the following conditions holds:
\begin{enumerate}
\item $\hat \rho\mid_{Z_{\Gamma_0}}:Z_{\Gamma_0} \to Y_{\Gamma_0}$ extends to an isomorphism $\rho_{\Gamma_0}:X_{\Gamma_0} \to Y_{\Gamma_0}$
\item $h(\overline{X}_{\Gamma_0}) < h(\overline{Y}_{\Gamma_0})$ and $\Ker(Y_{\Gamma_0})  \le \Ker(X_{\Gamma_0})$.
\end{enumerate}
\end{theorem}

The  \emph{map extension property} is defined and discussed in \Cref{sec:retracts}. 
A  $\Z$-subshift has the map extension property if and only if it is a mixing SFT (\Cref{prop:mixing_Z_SFT_abs_retract}).
All the other new notations that appear in the statement of \Cref{thm:relative_embedding_thm} are defined in \Cref{sec:defs}. 
Specifically, the notation $X_{\Gamma_0}$ is explained in \Cref{def:X_Gamma_0} and the notation $\overline{X}_{\Gamma_0}$ is explained in  \Cref{def:overline_X_Gamma_0}.

We emphasize the primary case $Z=\emptyset$ in the following corollary:
\begin{corollary}\label{cor:main_embedding_theorem}
Let $\Gamma$ be a countable abelian group and 
let $X,Y$ be $\Gamma$-subshifts. Suppose that  $Y$ has the map extension property. 
Then $X \hookrightarrow Y$ if and only if for every subgroup $\Gamma_0 \le  \Gamma$ one of the following conditions holds:
\begin{enumerate}
\item  $X_{\Gamma_0}$ is topologically conjugate to $Y_{\Gamma_0}$.
\item  $h(\overline{X}_{\Gamma_0}) < h(\overline{Y}_{\Gamma_0})$ and $\Ker(Y_{\Gamma_0}) \le \Ker(X_{\Gamma_0})$. 
\end{enumerate} 
\end{corollary}

In the particular case that $\Gamma=\Z$, the statement of \Cref{cor:main_embedding_theorem} coincides with the statement of Krieger's embedding theorem, as stated in \Cref{thm:krieger}. In the case $\Gamma = \Z^2$, the statement of \Cref{cor:main_embedding_theorem} coincides with the statement of \Cref{thm:embedding_thm_Z2_SFT}.

\textbf{Structure of the paper:} In \Cref{sec:defs} we 
review some basic definitions regrading countable abelian groups and subshifts, also introducing a few new ad-hoc definitions needed for our work. In \Cref{sec:krieger_marker_lemma} we present and proof a specific version of Krieger's fundamental Marker Lemma for actions of countable groups. In \Cref{sec:retracts} the map extension property is introduced and studied, along with the notion of retracts, absolute retracts and contractibility for subshifts. In \Cref{sec:entropy_min_subshifts} we  briefly recall the notion of entropy minimality for subshifts and use it to prove the necessity of the conditions for embedding in our main result. In \Cref{sec:partial_tilings} we introduce truncated Voronoi diagrams of subsets in $\mathbb{R}^d$ and in finitely generated abelian groups, deriving results about the existence of ``sufficiently invariant partial continuous equivariant tilings'' for subshifts. These are based on  techniques developed by Lightwood as part of his embedding theorems, with suitable adaptations that are required by the presence of periodic points. Lightwood's papers are not a prerequisite for this work. In \Cref{sec:marker_patttern} we derive a specific and somewhat technical statement regarding the existence of ``good markers'', using the map extension property. 
In  \Cref{sec:emb_fg} we conclude the proof of \Cref{thm:relative_embedding_thm} for the case where $\Gamma$ is a finitely generated abelian group.
In \Cref{sec:reduction_to_fg} we extend the proof to countable abelian groups (possibly not finitely generated). In \Cref{sec:lower_entropy_factors} we prove a multidimensional version of Boyle's theorem regarding lower entropy factors of subshifts, that applies to subshifts with the map extension property.

\textbf{A remark regarding the relation to Poirier and Salo's work \cite{poirier2024contractible}:} The notion of ``contractibility'' for subshifts, along with symbolic-dynamics analogs of other classical notions in topology such as homotopy have been developed by Poirier and Salo concurrently with the writing of this paper. We have not been aware of each other's work when I posted a preliminary version of this work in December 2023. Immediately after I made the first version of this work public, Ville Salo contacted me and informed me of his work with Leo Poirier. An initial version of \cite{poirier2024contractible} was uploaded to Arxiv around the same time.
It quickly became evident to us that Poirier and Salo's  notion of contractility is closely related to the map extension property that I introduced. My initial definition of the map extension property included an additional technical requirement. Ville Salo quickly realized that my initial definition of the map extension property can be streamlined to its current form. At least when abelian groups are concerned, for an arbitrary subshift the map extension property  coincides with being a contractible subshift of finite type. The overlap between the current work and \cite{poirier2024contractible} is partly an historical artifact due to concurrency. There are, however, significant differences in the point of view and emphasis between the current paper and Poirier and Salo's work \cite{poirier2024contractible}. The current paper is aimed at proving the embedding theorem for subshifts over countable abelian groups (with $\mathbb{Z}^d$-subshifts as the primary motivation). Poirier and Salo's work deals with subshifts over arbitrary countable groups (possibly non-commutative).
In attempt to keep the current papers self-contained and readable there is some overlap between these papers. Some results, although not the main ones, also overlap. Reading  \cite{poirier2024contractible} concurrently with the current paper is encouraged for deeper understanding the notion of contractility and homotopy that will be used in this paper, although it is not a prerequisite for this reading.

\textbf{Acknowledgments:} I am grateful to Leo Poirier and Ville Salo for insightful discussions in particular regarding homotopy and contractability in symbolic dynamics. I thank Barak Weiss for providing background about Korkine–Zolotarev reduced bases of lattices in $\mathbb{R}^d$ and the anonymous reviewer for a thorough review. This research was supported by the Israel Science Foundation grant No. 985/23.

\section{Notation and definitions}\label{sec:defs}

\subsection{Countable abelian groups}
Throughout this paper $\Gamma$ will be a countable or finite abelian group. We will use additive notation for the group operation in $\Gamma$. We will denote the identity  element of $\Gamma$ by $0$. We write $A \Subset \Gamma$ to indicate that  $A$ is a finite subset of $\Gamma$. For $A,B \subseteq \Gamma$ and $\gamma \in \Gamma$ the following standard notation will be used:
\[ \gamma + A = \left\{ \gamma +a ~:~ a\in A\right\},\]
\[ A+B = \left\{ a+b ~:~ a\in A,~ b \in B\right\},\]
\[ A-B =\left\{ a-b ~:~ a\in A,~ b \in B\right\}, \]
\[ -A = \left\{ -a ~:~ a \in A \right\}.\]

For $A,B,C \subseteq \Gamma$, we use the notation   $C=A \oplus B$ to indicate that for every $c \in C$ there exists a unique $a \in A$ and  a unique $b \in B$ such that $c=a+b$.

Call a set $A \subseteq \Gamma$  \emph{symmetric} if $A=-A$.

\begin{definition}
     Given $K, F \subseteq \Gamma$, let $\partial_K F$ denote the set of $\gamma \in \Gamma$ such that $(\gamma + K) \cap F \ne \emptyset$ and $(\gamma +K) \setminus F \ne \emptyset$. We refer to  $\partial_K F$ as the \emph{$K$-boundary of $F$}. 
\end{definition}

\begin{definition}
Given  $\epsilon >0$ and $K \Subset \Gamma$, a set $F \Subset \Gamma$ is called \emph{$(K,\epsilon)$-invariant} if $|\partial_K F| < \epsilon |F|$.

\end{definition}

\begin{definition}\label{def:Folner_sequence}
A sequence $(F_n)_{n=1}^\infty$ of finite subsets $F_n \Subset \Gamma$ is called a \emph{F\o lner sequence} in $\Gamma$ if for every $\epsilon>0$ and any $K \Subset \Gamma$ there exists $n_0 \in \mathbb{N}$ such that $F_n$ is $(K,\epsilon)$-invariant for every $n \ge n_0$. 
\end{definition}

It is well-known that any countable abelian group admits a F\o lner sequence. A countable group $\Gamma$ is called \emph{amenable} if it admits a F\o lner sequence.


\begin{definition}\label{def:K_separated_set}
    Given a set $K \subseteq \Gamma$, We say that a set $F \subseteq \Gamma$ is \emph{$K$-separated} if 
    $(v_1+K) \cap (v_2+K) = \emptyset$ whenever $v_1,v_2 \in F$ and $v_1 \ne v_2$.
    We say that $F \subseteq \Gamma$ is a \emph{maximal $K$-separated} set if $F$ is $K$-separated and there is no $K$-separated set that properly contains $F$.
\end{definition}

The following lemma collects some elementary yet useful statements related to the notion of an $A$-separated set:

\begin{lemma}\label{lem:separated_Sets}
    Let  $A,B,C \subset \Gamma$ be subsets. Then:
    
    \begin{enumerate}
         \item $A$ is $B$-separated if and only if $(A-A)\cap (B-B) = \{0\}$.
        \item $A$ is $B$-separated if and only if $B$ is $A$-separated.
        \item  If $A$ is $B$-separated and $(A+B)$ is $C$-separated then $A$ is $(B+C)$-separated.
    \end{enumerate}
\end{lemma}
\begin{proof}
    \begin{enumerate}
        \item The statement that $A$ is $B$-separated means that whenever $a_1,a_2 \in A$, $b_1,b_2 \in B$ satisfy $a_1+b_1 = a_2 +b_2$ then $a_1=a_2$. 
        Equivalently,  $a_1,a_2 \in A$, $b_1,b_2 \in B$ satisfy $a_1-a_2 = b_2 -b_1$ then $a_1=a_2$ and $b_1=b_2$. 
        Rephrasing, $A$ is $B$-separated if and only if  if whenever $\gamma \in (A-A) \cap (B-B)$ then $\gamma = 0$.
        \item 
        The previous formulation shows that the role of the sets  $A$ and $B$ in the statement  ``$A$ is $B$-separated'' are completely symmetric. So 
        $A$ is $B$-separated if and only if $B$ is $A$-separated.
        \item Suppose that  $A$ is $B$-separated and that $A+B$ is $C$-separated. Suppose  $a_1,a_2 \in A$, $b_1,b_2 \in B$ and $c_1,c_2 \in C$ satisfy that 
        \[ a_1+(b_1+c_1)= a_2 +(b_2 +c_2).\]
        Because  $A+B$ is $C$-separated, we have that $a_1+b_1=a_2+b_2$ and $c_1=c_2$.
        Because $A$ is $B$-separated we have conclude that $a_1=a_2$ and $b_1=b_2$.
        We conclude that under the assumptions above, whenever $a_1,a_2 \in A$, $b_1,b_2 \in B$ and $c_1,c_2 \in C$ satisfy that
        \[ a_1+(b_1+c_1)= a_2 +(b_2 +c_2),\]
        Then $a_1=a_2$. Equivalently, we showed that $A$ is $(B+C)$-separated.
    \end{enumerate}
\end{proof}

It is easy to check that an increasing union of $K$-separated sets is again a $K$-separated set, so any $K$-separated set is contained in a maximal 
$K$-separated set.
The following basic lemma regarding maximal $K$-separated sets will be useful:
\begin{lemma}\label{lem:max_k_sep}
    Let $K \subseteq \Gamma$ be a non-empty symmetric set and let $F \subseteq \Gamma$ be a maximal $K$-separated set. Then $F+K+K= \Gamma$.
\end{lemma}
\begin{proof}
    Since $F$ is a maximal $K$-separated set, for any $v \in \Gamma$ there exists $k,k_1 \in K$ and $v_1 \in F$ such that $v+k = v_1 + k_1$. It follows that
    $v= v_1 + k_1 -k$. Since $K$ is symmetric, $-k \in K$, so $v \in F + K +K$, proving that $\Gamma = F+K+K$.
\end{proof}

\begin{definition}
    A subset $D \subseteq \Gamma$ is a \emph{fundamental domain} for $\Gamma_0 \le \Gamma$ if for each $v \in \Gamma$ the coset $v+ \Gamma_0$ intersects $D$ exactly once. 
\end{definition}
Equivalently, $D \subseteq \Gamma$ is a fundamental domain for $\Gamma_0 \le \Gamma$ if and only if $\Gamma = D \oplus \Gamma_0$.
It can be directly verified that a subset $D \subseteq \Gamma$ is a fundamental domain for  $\Gamma_0 \le \Gamma$ if and only if $D$ is a maximal $\Gamma_0$-separated set.

We denote  the space of subgroups of $\Gamma$ by $\Sub(\Gamma)$. The space $\Sub(\Gamma)$ admits a standard topology called the \emph{Chabauty topology}.  

For a countable discrete group $\Gamma$, 
by viewing $\Sub(\Gamma)$ as a closed subset of $\{0,1\}^\Gamma$, the Chabauty topology can be interpreted as the relative topology induced from the product topology on $\{0,1\}^\Gamma$.

A basis for the Chabauty topology is given by sets of the form 
\[ N(A,B)= \left\{\Gamma_0 \in \Sub(\Gamma)~:~ \Gamma_0 \cap B = A \right\},~ A\Subset B \Subset \Gamma.\]
Thus, the statement $\lim_{n\to \infty}\Gamma_n = \Gamma_0$ means that for any finite set $B \Subset \Gamma$ there exists
$n_0 \in \mathbb{N}$ such that for any $n \ge n_0$ we have $\Gamma_n \cap B = \Gamma_0 \cap B$.

\subsection{Subshifts}
\begin{definition}
Let $A$ be a finite set.
For $x \in A^\Gamma$ and $v \in \Gamma$, we denote the value of $x$ at $v$ by $x_v$. Given $w \in \Gamma$ and $x \in A^\Gamma$, let $\sigma_w(x) \in A^\Gamma$ be given by $\sigma_{w}(x)_v = x_{v+w}$ for every $v \in \Gamma$. This defines an action of $\Gamma$ on $A^\Gamma$ called the \emph{shift action}\footnote{Since in this paper we restrict to commutative groups $\Gamma$, any left-action is also a right action, and the convention we chose is both standard and simple. In the case of non-commutative groups (in multiplcative notation), to get an action from the left one uses $(\sigma_g(x))_h=x_{g^{-1}h}$.}. We equip  $A^\Gamma$ with the product topology, where $A$ is given the discrete topology. This makes $A^\Gamma$ a compact metrizable topological space. 
A \emph{$\Gamma$-subshift} is a closed shift-invariant subset of $A^\Gamma$ for some finite set $A$.
In this context, we refer to $A^\Gamma$ as the \emph{$\Gamma$-full-shift} (or simply the full-shift when $\Gamma$ is fixed) over the alphabet $A$.
Given $F \subseteq \Gamma$ and $x \in A^\Gamma$, we let $x_F \in A^F$ denote the restriction of $x$ to $F$. 
\end{definition}

\begin{definition}
Let $X \subseteq A^\Gamma$ be a $\Gamma$-subshift and $F \Subset \Gamma$ a finite set.
We denote
\[
\cL_F(X) := \left\{ w \in A^F:~ \exists x \in X \mbox{ s.t. } x_F = w\right\}.
\]
\end{definition}
The elements of  $\cL_F(X)$ are called  \emph{$X$-admissible $F$-patterns}.
\begin{definition}
For $x \in A^\Gamma$, the \emph{stablizer} of $x$ is given by
\[
\stab(x) = \{ v\in \Gamma:~ \sigma_v(x)=x\}.
\]
\end{definition}

Clearly,  $\stab(x)$ is  a subgroup of $\Gamma$.

\begin{definition}\label{def:X_Gamma_0}
Given a $\Gamma$-subshift $X$ and a subgroup $\Gamma_0 \le \Gamma$ 
let \[
X_{[\Gamma_0]} = \left\{x \in X~:~ \Gamma_0 \le \stab(x) \right\}.
\]
and
\[
X_{\Gamma_0} = \begin{cases}
X_{[\Gamma_0]} & \mbox{ if } [\Gamma : \Gamma_0] = \infty\\
\left\{x \in X~:~ \stab(x) = \Gamma_0 \right\} & \mbox{ if } [\Gamma:\Gamma_0]<\infty. 
\end{cases}
\]
\end{definition}

\begin{definition}\label{def:overline_X_Gamma_0}
Given a $\Gamma$-subshift $X \subseteq A^{\Gamma}$ and a subgroup $\Gamma_0 < \Gamma$  let $\overline{X}_{[\Gamma_0]},\overline{X}_{\Gamma_0} \subseteq A^{\Gamma / \Gamma_0}$ denote the natural images of $X_{[\Gamma_0]}$ and $X_{\Gamma_0}$ respectively. Namely,

\[
\overline{X}_{[\Gamma_0]} = \left\{ \overline{x} \in A^{\Gamma/\Gamma_0}~:~ \exists x \in X_{[\Gamma_0]} \mbox{ s.t. } x_v=\overline{x}_{v +\Gamma_0} ~ \forall v \in \Gamma \right\}\]
and
\[
\overline{X}_{\Gamma_0} = \left\{ \overline{x} \in A^{\Gamma/\Gamma_0}~:~ \exists x \in X_{\Gamma_0} \mbox{ s.t. } x_v=\overline{x}_{v +\Gamma_0} ~ \forall v \in \Gamma \right\}.\]

\end{definition}

It is clear that both $\overline{X}_{[\Gamma_0]}$ and $\overline{X}_{\Gamma_0}$ are closed $(\Gamma/\Gamma_0)$-invariant sets, and that there are continuous bijections $\overline{X}_{\Gamma_0} \leftrightarrow X_{\Gamma_0}$ and   $\overline{X}_{[\Gamma_0]} \leftrightarrow X_{[\Gamma_0]}$ . 
Thus, $\overline{X}_{\Gamma_0}$ and $\overline{X}_{[\Gamma_0]}$ are (possibly empty)  $(\Gamma/\Gamma_0)$-subshifts.

\begin{definition}
    Given a $\Gamma$-subshift $X$, let 
    \[ \Ker(X) = \{ v \in \Gamma~:~ \sigma_v(x)=x \mbox{ for every } x \in X\}.\]
    Equivalently,
    \[ \Ker(X) = \bigcap_{ x \in X} \stab(x).\]
\end{definition}

The action of $\Gamma$ on $X$ is called \emph{faithful} if  $\Ker(X)$ is equal to the trivial subgroup.

\begin{definition}
Let $W \Subset \Gamma$ be  finite set, and let $\mathcal{F} \subset A^W$
be a set of $W$-patterns.
 We say that a pattern
$w \in A^W$ \emph{occurs} in $x \in A^\Gamma$ at $v \in \Gamma$ if  $\sigma_v(x)_{W} = w$.
\end{definition}

\begin{definition}\label{def:SFT}
 A $\Gamma$-subshift $Y \subseteq A^\Gamma$ is call a  \emph{subshift of finite type (SFT)}  if there exists a finite set $W\Subset \Gamma$ and a set of  patterns $\mathcal{F} \subset A^W$ such that 
 \[
 Y= \{ y \in A^{\Gamma}~:~ \sigma_v(y)_W \not \in \mathcal{F} \mbox{ for all } v \in \Gamma\}.
 \]
 In this case we say that $\mathcal{F} \subset A^W$ is a \emph{defining set of forbidden patterns for $Y$}.
\end{definition}

For a given shift of finite type $Y \subseteq A^\Gamma$ is possible to find many different  sets $W_1 \Subset \Gamma$ and $\mathcal{F}_1 \subseteq A^{W_1}$   such that $\mathcal{F}_1$ is  a defining set of forbidden patterns for $Y$.


\begin{definition}
Given  a finite set $W_0 \Subset \Gamma$, a function $\Phi:A^{W_0} \to B$ and another set $E \subseteq \Gamma$ let $\Phi^E:A^{W_0+E} \to B^{E}$ be given by
\[\Phi^E(x)_v = \Phi(\sigma_{v}(x)_{W_0}).
\]
\end{definition}


\begin{definition}
Let $\rho:X \to Y$ be  map between $\Gamma$-subshifts $X \subseteq A^\Gamma$, $Y \subseteq B^\Gamma$. A finite set $W \Subset \Gamma$ is a \emph{coding window} for $\rho$ if there exists  a function $\Phi:A^{W} \to B$  such that $\rho(x)_0 = \Phi(x_W)$ for all $x \in X$. In this case we say that   $\Phi$ is a \emph{sliding block code} for $\rho$.
\end{definition}

From the definition of the product topology, any map $\rho:X \to Y$ admits a coding window and a sliding block code. This basic fact is known as the \emph{Curtis-Hedlund-Lyndon Theorem}. As in the case of a defining set of forbidden patterns for a subshift of finite type, there is no uniqueness: A map $\rho:X\to Y$ can admit many different coding windows and many different sliding block codes.

\begin{definition}
    Let $\rho:X\to Y$ be a map between subshifts $X\subseteq A^\Gamma$ and $Y\subseteq B^\Gamma$, and let $\Phi:A^W \to B$  be a sliding block code for $\rho$. We say that a finite set $W_0 \Subset \Gamma$ is an \emph{injectivity window} for $\rho$ if for any $w^{(1)},w^{(2)} \in \cL_{W+W_0}(X)$   such that $w{(1)}_0 \ne w^{(2)}_0$ with have  $\Phi^{W_0}(w^{(1)}) \ne \Phi^{W_0}(w^{(2)})$. 
\end{definition}
By a standard compactness argument, a map $\rho:X \to Y$ is injective iff and only if it admits a finite injectivity coding window.

\begin{definition}\label{def:X_Delta}
    
Given a countable  group  $\Gamma$, a subgroup $\Delta < \Gamma$, and a $\Gamma$-subshift $X \subseteq A^\Gamma$ the \emph{$\Delta$-projection} of $X$ is defined to be the $\Delta$-subshift $X^{(\Delta)} \subseteq A^{\Delta}$ given by:
\[ X^{(\Delta)} = \left\{ x_\Delta :~ x \in X \right\}.\]
Let $X^{[\Delta]} \subseteq A^{\Gamma}$ denote the $\Gamma$-subshift given by:
\[X^{[\Delta]} = \left\{ x \in A^\Gamma~:~ \forall \gamma \in \Gamma~ \sigma_\gamma(x)_{\Delta} \in X^{(\Delta)}\right\}.\]
\end{definition}

It is straightforward  to check that $X^{(\Delta)}$ is indeed a $\Delta$-subshift. 
It is also clear that $X^{[\Delta]}$  is a $\Gamma$ subshift.

It follows directly from the definitions that  $X= X^{(\Gamma)}= X^{[\Gamma]}$.

Clearly, $X \subseteq X^{[\Delta]}$, and more generally, if $\Delta \le \tilde \Delta \le \Gamma$ then $X^{[\tilde \Delta]} \subseteq X^{[\Delta]}$.
Moreover:
\begin{lemma}\label{lem:X_Delta_SFT}
Let $X$ be a $\Gamma$-subshift of finite type, then there exists a finitely generated subgroup $\Delta \le \Gamma$ such that $X= X^{[\Delta]}$.    
\end{lemma}
\begin{proof}
Let $X$ be a   $\Gamma$-subshift of finite type. So  exists a finite set $W \Subset \Gamma$ and $\mathcal{F} \subseteq A^W$ such that $\mathcal{F}$ is a defining set of forbidden patterns for $X$. It follows that $X= X^{[\Delta]}$, where $\Delta$ is the group generated by $W$.
\end{proof}

\begin{definition}[Topological entropy for subshifts]
Let $X$ be a $\Gamma$-subshift. 
The topological entropy of $X$, denoted by $h(X)$, is given by
\[
h(X) = \lim_{n \to \infty} \frac{1}{|F_n|}\log(|\cL_{F_n}(X)|),
\]
where $(F_n)_{n=1}^\infty$ is any F\o lner sequence in $\Gamma$. 
For convenience, we denote
\[h(\emptyset) = -\infty.\]
\end{definition}

The existence of the limit (which is actually an infimum), and the fact that it does not depend on the choice of F\o lner sequence are well-known. See for instance \cite{downarowicz2016shearer}.

A simpler, equivalent definition is the following:
\[
h(X) = \inf_{F \Subset \Gamma} \frac{1}{|F|}\log(|\cL_F(X)|),
\]
where the infimum is over all finite subsets of $\Gamma$.

The equivalence of these definitions can be obtained from the following lemma, which we will use subsequently:
\begin{lemma}\label{lem:top_entropy_inv_set}
For any $\epsilon>0$ there exists $\delta>0$ and a finite set $W \Subset \Gamma$ such that
for any $(W,\delta)$-invariant set $F \Subset \Gamma$ the following inequality holds:
\[
\log |\cL_F(X)| \ge (1-\epsilon)|F| h(X).
\]
\end{lemma}

A short proof of \Cref{lem:top_entropy_inv_set} can be obtained using Shearer's inequality \cite{downarowicz2016shearer}.

The topological entropy of a subshift is invariant under isomorphism. 
Furthermore, if $\rho:X \to Y$ is a map, then $h(\rho(X)) \le h(X)$.

It can be easily verified that for any $\Gamma_0 \le \Gamma$ and any $\Gamma$-subshift $X$ we have:
\[
h(\overline{X}_{\Gamma_0}) = \inf \left\{ \frac{1}{|K|}\log \left| \cL_K(X_{\Gamma_0}) \right| ~:~ K \Subset \Gamma~,~ \Gamma_0 \mbox{ is } K \mbox{-separated}\right\}
\]
and
\[
h(\overline{X}_{[\Gamma_0]}) = \inf \left\{ \frac{1}{|K|}\log \left| \cL_K(X_{[\Gamma_0]}) \right| ~:~ K \Subset \Gamma~,~ \Gamma_0 \mbox{ is } K \mbox{-separated}\right\}.
\]

Note that when $\Gamma$ is a finite group and $X$ is a $\Gamma$-subshift, we have

\[ h(X) = \frac{1}{|\Gamma|}\log |X|. \]

The following result is probably folklore, we include it for completeness. 
\begin{proposition}\label{prop:subgroup_entropies_upperbound}
    Let $X$ be a $\Gamma$-subshift. Then 
    \[\limsup_{\Gamma_0 \to \{0_\Gamma\}} h(\overline{X}_{[\Gamma_0]}) \le h(X).\]
\end{proposition}
\begin{proof}
 We need to prove that for every $\epsilon >0$ there exists a finite set $K \Subset \Gamma$ such that for any $K$-separated subgroup $\Gamma_0 < \Gamma$ it holds that $ h(\overline{X}_{\Gamma_0}) \le h(X)+ \epsilon$.
Fix $\epsilon >0$. Since $h(X) = \inf_{K \Subset \Gamma}\frac{1}{|K|}\log |\cL_K(X)|$, there exists $K \Subset \Gamma$ such that 
\[ \frac{1}{|K|}\log |\cL_K(X)|  \le  h(X) +\epsilon.\]
Let $\Gamma_0 \le \Gamma$ be any $K$-separated subgroup.
Then 
\[h(\overline{X}_{[\Gamma_0]} )\le \frac{1}{|K|} \log |\cL_K(X_{[\Gamma_0]})|.\]
Since $\cL(X_{[\Gamma_0]}) \subseteq \cL(X)$, we have that
\[ \frac{1}{|K|} \log |\cL_K(X_{[\Gamma_0]})| \le \frac{1}{|K|} \log |\cL_K(X)|,\]
so
\[h(\overline{X}_{[\Gamma_0]}) \le  h(X) +\epsilon.\]
\end{proof}

In the classical case $\Gamma= \Z$, \Cref{prop:subgroup_entropies_upperbound} coincides with the well-known fact that for any $\Z$-subshift $X$ the number of points with period at most $n$ is bounded above by $e^{nh(X)+o(n)}$ as $n \to \infty$.

We introduce some specific notation to avoid repetition of the conditions that appear in the statement of \Cref{thm:relative_embedding_thm}:

\begin{definition}\label{def:E_X_Y}
Given $\Gamma$-subshifts $X,Y$, a  $\Gamma$-subshift $Z \subseteq X$, an injective map $\hat \rho:Z \to Y$ and $\Gamma_0 \le \Gamma$,
we use the notation  $X \hookrightarrow_{\hat \rho} Y$ to indicate that $\hat \rho$ extends to an embedding of $X$ into $Y$, and use the notation $X \cong_{\hat \rho} Y$ to indicate that $\hat \rho$ extends to an isomorphism of $X$ and $Y$.
We write $E(X,Y,\hat \rho,\Gamma_0)$ to indicate that  either $h(\overline{X}_{\Gamma_0}) < h(\overline{Y}_{\Gamma_0})$ and $\Ker(Y_{\Gamma_0}) \le \Ker(X_{\Gamma_0})$ or 
$X_{\Gamma_0} \cong_{\hat \rho} Y_{\Gamma_0}$.    
\end{definition}

Note that when $Z= \emptyset$ and $\hat \rho:Z \to Y$ is the ``empty map'', then $X \hookrightarrow_{\hat \rho} Y$ simply means that $X$ embeds in $Y$ and $X \cong_{\hat \rho} Y$ simply means that $X$ is topologically conjugate to $Y$.

\section{Markers in the domain: Krieger's Marker Lemma}
\label{sec:krieger_marker_lemma}

In this section  we state and prove a version of Krieger's Marker Lemma 
 that applies to $\Gamma$-subshifts, where $\Gamma$ is a countable abelian group (\Cref{lem:krieger_marker_lemma}). Krieger's 
original lemma is stated for $\Z$-subshifts \cite[Lemma 2]{MR693975}.
Several variants of this lemma can be found in the literature, for instance \cite[Lemma $3.4$]{MR3540453}, \cite[Lemma 2.1]{MR3359054} and  \cite[Lemma 3.2]{MR4311113}, all with very similar proofs. 
The main reason that we reprove the lemma here is that our formulation of \Cref{lem:krieger_marker_lemma} involves a detailed treatment of  
non-free actions (a slightly different treatment of non-free actions appears in \cite[Lemma 3.2]{MR4311113}).
We only deal here with abelian groups purely for simplicity of notation, to maintain consistency of notation with the rest of the paper: Using suitable (multiplicative) notion, Krieger's marker lemma is valid for any action of a countable groups on a compact zero dimensional space, essentially with the same proof as our proof of  \Cref{lem:krieger_marker_lemma} below.

\begin{definition}
Let $X$ be a $\Gamma$-subshift and let $P \subset \Gamma\setminus \{0\}$ be a finite set. We say that $A \subset X$ is a \emph{$P$-marker} if 
\[
A \cap \sigma_{\gamma}(A) =\emptyset, \mbox{ for all } \gamma \in P.
\]
\end{definition}
The following basic lemma says that markers can be  effectively  merged:
\begin{lemma}\label{lem:merge_clopen_markers}
Suppose that $X$ is a $\Gamma$-subshift
and that $P \subset \Gamma \setminus \{0\}$ is finite and symmetric. Suppose $A,B \subseteq X$ are clopen $P$-markers. Then there exists a clopen $P$-marker $C \subseteq X$ such that $A \subseteq C$,  $B  \subseteq C \cup \bigcup_{\gamma \in P}\sigma_{\gamma}C$ and $C \subseteq (A\cup B)$.
\end{lemma}

\begin{proof}
Let 
\[
C := A \cup \left( B \setminus \bigcup_{\gamma \in P} \sigma_{\gamma}(A)\right).
\]
Clearly,  $C \subseteq (A \cup B)$.
The set  $C$ is clopen because it is contained in the   Boolean algebra spanned by $\{A,B\}\cup\{\sigma_\gamma(A) : \gamma \in P\}$. It is clear that $A \subseteq C$. 

Next, we  check that $B \subseteq  C \cup \bigcup_{\gamma \in P}\sigma_{\gamma}C$:
Otherwise, suppose that  $x \in B \setminus  \left(C \cup \bigcup_{\gamma \in P} \sigma_\gamma(C)\right)$ then in particular  $x \not\in C$ so $x \not \in A$ because $A \subseteq C$. Also, because $x \not \in B \setminus C$, we have that 
$x \in \sigma_\gamma(A)$ for some $\gamma \in P$. But since $A \subseteq C$, it follows that $\sigma_\gamma(A) \subseteq \sigma_\gamma(C)$, so $x \in \sigma_\gamma C$ for this $\gamma \in P$, contradicting the assumption that $x  \in B \setminus  \left(C \cup \bigcup_{\gamma \in P} \sigma_\gamma(C)\right)$.  
This shows that $B \subseteq C \cup \bigcup_{\gamma \in P} \sigma_\gamma(C)$.

It remains to check that $C$ is a $P$-marker.
 For any  $\gamma \in P$ we have:
    \[C \subseteq A \cup \left( B  \cap \sigma_\gamma (A^c) \right)\]
 Because $P$ is symmetric also $-\gamma \in P$ so
    \[C \subseteq A \cup \left( B  \cap \sigma_\gamma^{-1} A^c \right).\]
Applying $\sigma_\gamma$ to both sides, we get:
    \[\sigma_\gamma (C) \subseteq \sigma_\gamma(A) \cup \left( \sigma_\gamma(B)  \cap  A^c \right)\]

    It follows that
    $C \cap \sigma_\gamma C$
    is contained in the union of the following $4$ sets:
    \begin{enumerate}
        \item $E_1 := A \cap  \sigma_\gamma A$,
        \item $E_2:= A \cap \left( \sigma_\gamma B  \cap  A^c \right)$,
        \item $E_3 := \left( B  \cap \sigma_\gamma A^c \right) \cap  \sigma_\gamma A$,
        \item $E_4:= \left( B  \cap \sigma_\gamma A^c \right) \cap \left( \sigma_\gamma B  \cap  A^c \right)$.
    \end{enumerate}
By assumption $A$ is a $P$-marker so, $E_1 \ = \emptyset$. Clearly,  $E_2 \subseteq  A \cap (A^c) = \emptyset$ and $E_3 \subseteq \sigma_\gamma (A^c) \cap \sigma_{\gamma}(A) = \emptyset$. Also, $E_4 \subseteq B \cap \sigma_\gamma (B) = \emptyset$ by assumption that $B$ is a $P$-marker.
    This shows that $C\cap \sigma_\gamma C =\emptyset$ for all $\gamma \in P$.
\end{proof}

\begin{lemma}[Krieger's Marker Lemma for $\Gamma$-subshifts]\label{lem:krieger_marker_lemma}
Let $X$ be a $\Gamma$-subshift.
Let $P \Subset \Gamma\setminus \{0\}$ be a finite symmetric set, and let $V \subseteq X$ be a clopen set such that $\sigma_\gamma(x) \ne x$ 
for every $\gamma \in P$ and every $x \in V$.
Then there exists a clopen set $C \subseteq V$  
so that 
\begin{equation}
C \cap \sigma_\gamma (C) = \emptyset \mbox{ for every } \gamma \in P \mbox { and }
V \subseteq C \cup \bigcup_{\gamma \in P} \sigma_\gamma (C).
\end{equation}
\end{lemma}
\begin{proof}
For every $x \in V$ and $\gamma \in P$ we have that $x \ne \sigma_{\gamma}(x)$. We can thus find for every $x \in V$ a clopen set $E_x \subseteq V$ such that $x \in E_x$ and $E_x \cap \sigma_{\gamma}(E_x) = \emptyset$ for every $\gamma \in P$ (that is $E_x$ is a $P$-marker).

Since $V$ is compact, we can find a finite subcover $\mathcal{U} = \{D_1,\ldots,D_n\}$ of $\{E_x :~ x \in V\}$. 
Thus each $D_j \subseteq V$ is a clopen $P$-marker and $V \subseteq \bigcup_{j=1}^n D_j$.
We sequentially construct clopen $P$-markers $C_1,\ldots,C_n \subseteq V$ such that $\bigcup_{k=1}^j D_k \subseteq C_j \cup \bigcup_{\gamma \in P} \sigma_{\gamma}(C_j)$.
We start by setting  $C_1 = D_1$.
After $C_1,\ldots,C_{j-1}$ have been constructed, apply
\Cref{lem:merge_clopen_markers} with $A := C_{j-1}$ and $B:=D_j$ to obtain a clopen $P$-marker $C_j$ such that $A \subseteq C_j$ and 
$B \subseteq C_j \bigcup_{\gamma \in P}\sigma_\gamma (C_j)$.
Then $C= C_n$ is a clopen set that satisfies  the conclusion of the lemma. 
\end{proof}




As mentioned in the beginning of this section, an inspection of the proof of \Cref{lem:merge_clopen_markers} and \Cref{lem:krieger_marker_lemma} reveals  that commutativity of the group $\Gamma$  is not used in the proof, except for notational purposes. Also, the fact that $X$ is a $\Gamma$-subshift rather than a general compact $\Gamma$-space is only used to produce a clopen basis. In other words, expansiveness of the action is never used. Thus, the proof of \Cref{lem:krieger_marker_lemma} directly extends to the following (mentioned for future reference, the generality will not be used in the current paper):
\begin{lemma}[Krieger's Marker Lemma for countable groups]\label{lem:krieger_marker_lemma_ctbl}
Let $\Gamma \curvearrowright X$ be an action of a countable group $\Gamma$ on a totally disconnected compact metric space. Let $P \Subset \Gamma$ be a finite symmetric set with $1_\Gamma \not\in P$, and let $V \subseteq X$ be a clopen set such that $g(x) \ne x$ for every $g \in P$ and every $x \in V$.
Then there exists a clopen set $C \subseteq V$  
so that 
\begin{equation}
C \cap g(C) = \emptyset \mbox{ for every } g \in P \mbox { and }
V \subseteq C \cup \bigcup_{g \in P} g (C).
\end{equation}
\end{lemma}

As mentioned in the earlier, closely related statements have already appeared in the literature. 


\section{Retractions and the map extension property for subshifts}\label{sec:retracts}


In this section we introduce the \emph{map extension property} for $\Gamma$-subshifts, and the notion of an \emph{absolute retract}. Both of these properties make  sense in the category of  general topological dynamical systems, and also in other categories.  In topology the notion of an absolute retract has been introduced by Borsuk  in the 1930's, see \cite{MR1674915}. We show that a $\Gamma$-subshift has the map extension property if and only if it is an absolute retract for a certain class of subshifts. A $\mathbb{Z}$-subshift has the map extension property if and only if it is a topologically mixing subshift of finite type. As we will shortly explain, in the context of subshifts of finite type over countable abelian groups, the map extension property turns out to be equivalent to the notion of \emph{contractability} \cite{poirier2024contractible}.

\begin{definition}\label{def:right_squigarrow}
    Let $X,Y$ be $\Gamma$-subshifts. 
    We write $X \rightsquigarrow Y$ if the following holds: For every  $x \in X$ there exists $y \in Y$ such that $\stab(x) \le \stab(y)$.
\end{definition}

The condition $X \rightsquigarrow Y$ is a obviously a necessary condition for the existence of map from $X$ into $Y$. Indeed,  if there exists a map $\pi:X \to Y$ then $X \rightsquigarrow Y$ because  for any $x \in X$ we have  $\stab(x) \le \stab(\pi(x))$.

\begin{definition}\label{def:map_extension_property}
    A $\Gamma$-subshift $Y$ has the \emph{map extension property} if the following holds:
    For every $\Gamma$-subshift $X$ such that $X \rightsquigarrow Y$  and any map $\tilde \pi: \tilde X \to Y$ from a (possibly empty) $\Gamma$-subshift $\tilde X \subseteq X$ there exists a map $\pi:X \to Y$  that extends $\tilde \pi$.
\end{definition}

When $\Gamma$ is a finitely generated abelian group, the map extension property  for a  $\Gamma$ subshift already implies that the subshift is of finite type. For subshifts of finite type over a countable abelian group, the map extension property is equivalent to the notion of \emph{(strong) contractibility} introduced recently by Poirier and Salo \cite{poirier2024contractible}:

\begin{definition}[Poirier and Salo \cite{poirier2024contractible}]\label{def:contractible}
    A $\Gamma$-subshift $Y$ is \emph{contractible} of there exists a map $\psi:\{0,1\}^\Gamma \times Y \times Y \to Y$ satisfying 
    \begin{equation}\label{eq:contraction}
         \psi(\overline{0},y^{(0)},y^{(1)}) = y^{(0)} \mbox{ and }
         \psi(\overline{1},y^{(0)},y^{(1)}) = y^{(1)}
    \end{equation}
    for every $y^{(0)},y^{(1)} \in Y$, where $\overline{0},\overline{1} \in \{0,1\}^\Gamma$ are the two fixed points.  
    
    A map $\psi:\{0,1\}^{\Gamma} \times Y \times Y \to Y$ that satisfies \eqref{eq:contraction} is called a \emph{contraction homotopy}.
    \end{definition}
\begin{definition}\label{def:storng_contractable}
      A contraction homotopy $\psi:\{0,1\}^{\Gamma} \times Y \times Y \to Y$ that further satisfies $\phi(z,y,y) = y$ for every $z \in \{0,1\}^\mathbb{Z}$ and every $y \in Y$,  is called a \emph{strong contraction homotopy}. If a $\Gamma$-subshift $Y$ admits a strong contraction homotopy, we say that $Y$ is \emph{strongly contractable}.
\end{definition}

The term ``contractible'' comes from the observation that a topological space is $Y$ is contractible in the usual sense if and only there exists a continuous function $\psi:[0,1] \times Y \times Y \to Y$ satisfying \eqref{eq:contraction} for every $y^{(0)},y^{(1)} \in Y$, where this time $\overline{0}$ and $\overline{1}$ are the two endpoints of the interval $[0,1]$. 
Poirier and Salo \cite{poirier2024contractible} used the term ``equiconnected'' instead of ``strongly contractable'', because it is a direct analog of the corresponding notion of equiconnectedness in topology, with $\{0,1\}^\Gamma$ taking the role of the interval $[0,1]$. Poirier and Salo proved that a $\mathbb{Z}^d$-subshift is (strongly) contractible SFT if and only if it satisfies the map extension property \cite{poirier2024contractible}. Their proof directly extends to $\Gamma$-subshifts, where $\Gamma$ is a finitely generated abelian group.

For the reader's convenience we recall the elegant argument for the easy implication, which assumes nothing about the countable group $\Gamma$. The proof of the converse implication in \cite{poirier2024contractible} uses a version of Krieger's Marker lemma very similar to \Cref{lem:krieger_marker_lemma}. We will not use the converse implication in this paper.

\begin{proposition}[Poirier- Salo \cite{poirier2024contractible}]\label{prop:map_ext_contractible}
 Any $\Gamma$-subshift with the map extension property is strongly contractible.
\end{proposition}
\begin{proof}
Suppose that $Y$ has the map extension property. Consider the subshift $X = \{0,1\}^\Gamma \times Y \times Y$. Clearly $Y$ is a factor of $X$, and so in particular $X \rightsquigarrow Y$. 

Consider the subshifts
\[\tilde X_1 = \left(\left\{\overline{0},\overline{1}\right\} \times Y \times Y\right)  \subseteq X,\]
\[ \tilde X_2 = \left\{ (z,y^{(1)},y^{(2)}) \in \{0,1\}^\Gamma \times Y \times Y~:~ y^{(1)} = y^{(2)} \right\}. \]
and
\[
\tilde X = \tilde X_1 \cup \tilde X_2.
\]
and the map $\tilde \pi: \tilde X \to Y$ given by
\[
\tilde \pi(z,y,y) = y \mbox{ for every } z\in \{0,1\}^\Gamma \mbox{ and every } y \in Y,
\]
\[
\tilde \pi(\overline{0},y^{(0)},y^{(1)}) = y^{(0)} \mbox{ and }\tilde \pi(\overline{1},y^{(0)},y^{(1)}) = y^{(1)}
\]
for every $y^{(0)},y^{(1)} \in Y$.
By the map extension property, $\tilde \pi$ extends to a map $\psi:\{0,1\}^\Gamma \times Y \times Y \to Y$. It is easily verified that the map $\psi$ is  a strong contraction homotopy. Hence $Y$ is indeed strongly contractible.
\end{proof}

Under certain assumptions on the group $\Gamma$, Poirier and Salo have proved that contractible $\Gamma$-subshifts  have dense periodic points \cite[Theorem $1.4$]{poirier2024contractible}. These assumptions hold whenever $\Gamma$ is a finitely generated abelian group.

We need a statement that gives us more control over the collection of subgroups that occur as stablizers of points of a subshift with the map extension property.
The following lemma is an intermediate step. 
\begin{lemma}\label{lem:contractible_periodic_points1}
    Let $Y$ be a contractible $\Gamma$-subshift, and let $\psi:\{0,1\}^\Gamma \times Y \times Y \to Y$ be a strong contraction homotopy with coding window $W \Subset \Gamma$. For any surjective homomorphism $\phi:\Gamma \to \mathbb{Z}$ and any $v \in \Gamma$ such that 
    $\phi(v) > 4 \max |\phi(W)|$ and any $y^{(0)} \in Y$ 
    there exists 
    $y \in Y$
    such that 
    $\sigma_v(y)=y$,
    and furthermore $\stab(y^{(0)}) \cap \Ker(\phi) \subseteq \stab(y)$.
\end{lemma}
\begin{proof}
Let $\phi:\Gamma \to \mathbb{Z}$ be a surjective homomorphism.
Denote $N_0= \max |\phi(W)|$. Choose $v_0 \in \Gamma$ such that $\phi(v_0)=N_0$.
Let $v \in \Gamma$ satisfy  $\phi(v) > 4 N_0$.

Denote \[N= \phi(v), ~N_1 = \lfloor N/2 \rfloor,\]
and choose $v_1 \in \Gamma$ such that $\phi(v_1)= N_1$.

Let $z \in \{0,1\}^\Gamma$ be given by 
\[
z_w = \begin{cases}
0 & \phi(w) \le 0\\
1 & \phi(w)  > 0.
\end{cases}
\]

We record a few simple observations that follow from the fact that $\psi$ is a strong contraction homotopy with coding window $W$:
\begin{enumerate}
    \item For any $\tilde y \in Y$, the point
    \begin{equation}\label{eq:hat_y_def}
\hat y:= \psi(z,\tilde y,\sigma_{-v}(\tilde y))
    \end{equation}
    
    satisfies $\hat y_w = \hat y_{w+v}$ for every $w \in \Gamma$ such that
    $ -N+N_0 < \phi(w) < -N_0$ and $\hat y_w =\tilde  y_w$ for every $w \in \Gamma$ such that $\phi(w) < -N_0$.
    \item If $\tilde y \in Y$ satisfies $\tilde y_w = \tilde y_{w+v}$ for every $w \in \Gamma$ such that $|\phi(w)| < N_0$ then the point $\hat y \in Y$ given by \eqref{eq:hat_y_def}
    satisfies $\hat y_w = \tilde y_{w+v}$ whenever $-\phi(v) < \phi(w) <0$ and also $\hat y_w= \tilde y_w$ for all $w \in \Gamma$ such that $\phi(w) < 0$.
    \item Suppose $\tilde y \in Y$ satisfies $\tilde y_w =\tilde  y_{w+v}$ for every $w \in \Gamma$ such that $|\phi(w)| < kN_1+N_0$ for some $k \in \mathbb{N}$. Then 
    \[
    \hat y = \sigma_{-kv_1}\phi(z,\sigma_{kv_1}(\tilde y),\sigma_{kv_1-v}(\tilde y))
    \]
    satisfies $\hat y_w =\tilde  y_{w+v}$ for every $w \in \Gamma$ such that $|\phi(w)| < (k+1)N_1+N_0$ and also $\hat y_w= \tilde y_w$ for all $w \in \Gamma$ such that $\phi(w) < kv_1$.
\end{enumerate}
Using the above observations, starting with any $y^{(0)} \in Y$, we can construct by induction a sequence  of points $y^{(1)},\ldots, y^{(n)},\ldots \in Y$  such that  
$y^{(n)}_w = y^{(n)}_{w+v}$ for any $w \in \Gamma$ satisfying $|\phi(w)| < nN_1 +N_0$, and also 
$y^{(n+1)}_w = y^{(n)}_{w}$ for any $w \in \Gamma$ satisfying $|\phi(w)| < n N_1$.
Thus, the limit $y = \lim_{n \to \infty} y^{(n)}$ exists, and satisfies $\sigma_v(y)=y$ .
It also satisfies  $\stab(y^{(0)}) \cap \Ker(\phi) \subseteq \stab(y)$.
\end{proof}

We remark that the statement of \Cref{lem:contractible_periodic_points1} remains valid with essentially the same proof for any countable group $\Gamma$, without assuming commutativity.

Recall that definition of $Y^{[\Delta]}$ as defined in \Cref{def:X_Delta}.

\begin{lemma}\label{lem:contractible_Y_Delta}
    Let $Y$ be a contractible $\Gamma$-subshift.  
    Then there exists a finitely generated subgroup $\Delta < \Gamma$ so that $Y = Y^{[\Delta]}$.
\end{lemma}
\begin{proof}
 Let $\psi: \{0,1\}^\Gamma \times Y \times Y \to Y$ be a contraction homotopy with coding window $W \Subset \Gamma$, and let $\Delta$ be the subgroup of $\Gamma$ generated by $W$. In this case we claim that $Y  = Y^{[\Delta]}$. Indeed for any $y^{(1)},y^{(2)} \in Y$ and any $A \subseteq \Gamma/\Delta$ we can find a point $y \in Y$ that agrees with $y^{(1)}$ on the cosets of $\Delta$ that belong to $A$ and agrees with $y^{(2)}$ on the cosets of $\Delta$ that do not belong to $A$.
Let $z \in \{0,1\}^\Gamma$  be the point satisfying that $z_w=1$ if the $\Delta$-coset of $w$ is in $A$ and $0$ otherwise.
Then the point $y= \psi(z,y^{(1)},y^{(2)})$ agrees with $y^{(1)}$ on the $\Delta$-cosets that belong to $A$ and agrees with $y^{(2)}$ on the cosets of $\Delta$ that do not belong to $A$.
By repeated applications, for any tuple of distinct $\Delta$-cosets $(v_1+\Delta,\ldots,v_k +\Delta)$ with $v_1,\ldots, v_k \in \Gamma$ and any points $y^{(1)},\ldots, y^{(k)}$ we can find $y \in Y$ that agrees with $y^{(i)}$ on $v_i + \Delta$. This proves that  $Y = Y^{[\Delta]}$.
\end{proof}

\begin{proposition}\label{prop:map_extension_implies_finitely_many_forbidden_periods}
    Let $\Gamma$ be a countable abelian group and let $Y$ be a strongly contractible $\Gamma$-subshift. Then there exists a finite set $F \Subset \Gamma \setminus \{0\}$ such that for any subgroup $\Gamma_0 < \Gamma$, with $\Gamma_0 \cap F =\emptyset$ there exists $y \in Y$ such that $\Gamma_0 \le \stab(y)$.
\end{proposition}

\begin{proof}
    By \Cref{lem:contractible_Y_Delta} we can assume without loss of generality that  $\Gamma$ is a finitely generated  abelian group. In this case, by the structure theorem for finitely generated abelian groups we can assume that $\Gamma = \Z^d \times G$, where $G$ is a finite abelian group. Let $Y$ be a contractible $\Gamma$-subshift, and let $\psi:\{0,1\}^\Gamma \times Y \times Y \to Y$ be a strong contraction homotopy with coding window $W \Subset \Gamma$. 
    
    We first prove the following auxilary statement: There exists a finite set $F \Subset \Gamma$ (that depends on the set $W$) such that for any subgroup $\Gamma_0 \le \Gamma$ such that $F \cap \Gamma_0 = \emptyset$ there exists an integer $0 \le n \le d$, surjective homomorphisms $\phi_1,\ldots,\phi_n : \Gamma \to \Z$ and $v_1,\ldots,v_n \in \Gamma_0$ such that $v_1,\ldots,v_n$ is a generating set for $\Gamma_0$ and so that
\begin{itemize}
    \item $\phi_i(v_j)=0$ for all $0 \le i < j \le n$.
    \item $\phi_i(v_i)> 2(\max \phi_i(W))$ for every $1\le i \le n$.
\end{itemize}

Indeed, by taking $F$ to be a set that contains the non-identity elements of the torsion subgroup of $\Gamma$, we get that any subgroup $\Gamma_0 < \Gamma$ that does not intersect $F$ is torsion free, hence the  projection of $\Gamma_0$ onto $\Z^d$ is injective. So without loss of generality we can prove the statement under the assumption that $\Gamma= \Z^d$. We choose $R>0$ large enough, in a manner that depends on the set $W$ and the integer $d$ (the precise conditions of $R$ will be clarified soon). Let $F$ denote the intersection of the euclidean ball of radius $R$ around $0$ with $\Z^d \setminus \{0\}$, and let  $\Gamma_0 < \Gamma= \Z^d$ be a subgroup that does not intersect $F$. We can assume without loss of generality that $\Gamma_0$ is a finite index subgroup of $\Z^d$.
In other words, $\Gamma_0$ is a lattice whose shortest non-zero element has length at least $R$. Let $v_1,\ldots,v_d \in \Gamma_0$ be a Korkine–Zolotarev reduced basis  of $\Gamma_0$ as in \cite{Korkine1873}. Then there exists a constant $c_1 >0$ depending only on $d$ such that for any $1\le i \le d$ the orthogonal projection of $v_i$ onto the orthogonal complement of the span of $v_1,\ldots,v_{i-1}$ has length at least $c_1 R$.
Thus, there exists another constant $c_2>0$ that depends only on $d$, and primitive vectors $w_1,\ldots,w_d \in \Z^d$  such that $w_i$ belongs to the  orthogonal complement of the span of $v_1,\ldots,v_{i-1}$ and $\langle v_i,w_i \rangle > c_2 R \|w_i \|$. Define $\phi_i:\Z^d \to \Z$ by $\phi_i(v)=\langle v,w_i \rangle$. Then for every $i=1,\ldots,d$ we have that  $\phi_i$ is a surjective homomorphism that satisfies $\phi_i(v_i) \ge c_2 R \|w_i\|$. Since $\max \phi_i(W) \le \| w_i \| \max_{w \in W}\| w\|$, if $R > 2 c_2^{-1} \max_{w \in W}\| w\|$ we have that $\phi_i(v_i)> 2(\max \phi_i(W)$. This proves the auxiliary statement.

Now repeatedly apply \Cref{lem:contractible_periodic_points1} to produce a sequence of points $y^{(1)},\ldots,y^{(d)} \in Y$ such that $v_1,\ldots,v_i \in \stab(y^{(i)})$. In particular, $\Gamma_0$ is contained in  $\stab(y^{(d)})$.
This completes the proof. 

\end{proof}

Next, we consider the notion of retractions, retracts and absolute retracts in the context of symbolic dynamics.
Recall that a subspace $Y$ of a topological space $X$  is called a \emph{retract} if there is a continuous function $r:X \to Y$ such that the restriction of $r$ to $Y$ is the identity. Such a  continuous function $r:X \to Y$ is called a retraction.
Let $\mathcal{C}$ be a class of topological spaces, closed under homeomorphisms and passage to closed subsets. Following Borsuk \cite{Borsuk1931}, a topological space  $Y$ is called an \emph{absolute retract for the class $\mathcal{C}$} if $Y\in \mathcal{C}$ and whenever $Y$ is a closed subspace of $X \in \mathcal{C}$ then $Y$ is a retract of $X$.

The notions of a  retract and universal retract naturally adapt to the category of dynamical systems, in particular to $\Gamma$-subshifts:

\begin{definition}
    Let $Y \subseteq X$ be $\Gamma$-subshifts. We say that $Y$ is a retract of $X$ if there exists a map $r:X \to Y$ such that the restriction of $r$ to $Y$ is the identity.  A map $r$ as above is called a retraction. 

\end{definition}

As shown by the following proposition, various important properties of subshifts are preserved under retractions:

\begin{prop}\label{prop:closed_under_retracts}
    The following families of subshifts are closed under taking retracts:
    \begin{enumerate}
        \item sofic shifts
      \item topologically mixing subshifts
    \item strongly irreducbile subshifts
      \item subshifts with the map extension property
        \item shifts of finite type
    \end{enumerate}
\end{prop}
\begin{proof}
The classes of sofic shifts, topological mixing subshifts and strongly irreducible subshifts are each closed under taking (subshift) factors. In particular, they are closed under taking retracts. 

Let us show that a retract of a subshift with the map extension property also has the map extension property. 
Suppose that $\hat Y$ is a subshift with the map extension property, that $Y \subseteq \hat Y$ and that $r:\hat Y \to Y$ is a retraction map. 
Let $\tilde X \subseteq X$ be subshifts so that $X \rightsquigarrow Y$ and $\tilde \pi:\tilde X \to Y$ be a map. Since $Y \subseteq \hat Y$ it follows that $X \rightsquigarrow \hat Y$. Since $\hat Y$ has the map extension property, $\tilde \pi$ extends to a map $\hat \pi:X \to \hat Y$. Then $\pi = r \circ \hat \pi: X \to Y$ is a map that extends $\tilde \pi$.

It remains to show that a retract of any retract of a  shift of finite type is also a shift of finite type.
Let $X \subseteq A^\Gamma$ be a shift of finite type with a defining set of  forbidden patterns $\mathcal{F}_1 \Subset A^{W_1}$.
Suppose that $Y \subseteq X$  is a retract, with $r:X \to Y$ a retraction.
Let $W_0 \Subset \Gamma$ be a coding window for $r$ such that $0 \in W_0$, and let  $\Phi:A^{W_0} \to A$ be a sliding block code for $r$.
Let 
    \[
    \mathcal{F}_0 = \left\{ w \in A^{W_0} ~:~\Phi(w)_0 \ne w_0 \right\}.
    \]
    Now suppose that $x \in X$ satisfies that $\sigma_v(x)_{W_0} \not \in \mathcal{F}_0$ for all $v \in \Gamma$. It follows that $r(x)=x$, so $x \in Y$. This means that 
    \[
    Y = \left\{ x \in X~:~  \sigma_v(x)_{W_0} \not \in \mathcal{F}_0 \mbox{ for all } v \in \Gamma \right\}.
    \]
    Denote $W = W_0 \cup W_1$ and let
    \[\mathcal{F} = \left\{ w \in A^W ~:~~ w_{W_0} \in \mathcal{F}_0 \mbox{ or } w_{W_1} \in \mathcal{F}_1 \right\}.\]
    Then $\mathcal{F}$ is a defining set of forbidden patterns for $Y$. This proves that $Y$ is a shift of finite type.
\end{proof}

\begin{definition}
    A $\Gamma$-subshift $Y$ is called an \emph{absolute retract} if  whenever $Y$ embeds in a $\Gamma$-subshift $X$, then $Y$ is a retract of $X$.
\end{definition}

An absolute retract always has a fixed point for the following simple reason:  Any subshift $Y$ embeds in a subshift with a fixed point, and any retract of a subshift with a fixed point must admit a fixed point. 
\begin{definition}
Let $\mathcal{C}$ be a class of $\Gamma$-subshifts, closed under isomorphism and passing to subshifts.
    A $\Gamma$-subshift $Y$ is called an \emph{absolute retract for the class $\mathcal{C}$} if $Y \in \mathcal{C}$ and whenever $Y$ embeds in a $\Gamma$-subshift $X \in \mathcal{C}$ then $Y$ is a retract of $X$.
\end{definition}

\begin{definition}\label{def:RAR}
    Given a finite set $\mathcal{G} \Subset \Sub(\Gamma)$, we say that a $\Gamma$-subshift $X$ is \emph{$\mathcal{G}$-free} if for every $x \in X$ and $\Gamma_0 \in \mathcal{G}$ we have that $\Gamma_0 \not \subseteq \stab(x)$. 
\end{definition}

\begin{definition}\label{def:witness_G_free}
    Let $Y$ be a $\Gamma$-subshift, and let $\mathcal{G}$
    be a finite set of finitely generated subgroups of $\Gamma$.
    We say that a finite set $K \Subset \Gamma$ \emph{witnesses that $Y$ is $\mathcal{G}$-free} if 
    each $\Gamma_0 \in \mathcal{G}$ admits a finite generating set $F_{\Gamma_0} \Subset \Gamma$ so that
    for every $w \in \cL_{K}(Y)$ and every $\Gamma_0 \in \mathcal{G}$ there exists $\gamma \in F_{\Gamma_0}$ and $v \in K$ such that $v+ \gamma\in K$ and $w_{v} \ne w_{v+\gamma}$.
\end{definition}

\begin{lemma}\label{lem:witness_G_free}
Let $\mathcal{G}\Subset \Sub(\Gamma)$ be a finite set of finitely generated subgroups of $\Gamma$, and let $Y$ be a $\Gamma$-subshift.
Then $Y$ is $\mathcal{G}$-free if and only if there exists a finite set $K \Subset \Gamma$ with $0 \in K$ that witnesses that $Y$ is $\mathcal{G}$-free.
\end{lemma}
\begin{proof}
If $K \Subset \Gamma$ witnesses that $Y$ is $\mathcal{G}$-free then
 each $\Gamma_0 \in \mathcal{G}$ admits a finite generating set $F_{\Gamma_0} \Subset \Gamma$ so that
for every $y \in Y$ and every $\Gamma_0 \in \mathcal{G}$ there exists $v \in K$ and $\gamma \in F_{\Gamma_0}$ such that $y_v \ne y_{v + \gamma}$, and so $\Gamma_0 \not \subseteq \stab(y)$. This proves that $Y$ is $\mathcal{G}$-free.

For the converse direction:
For each $\Gamma_0 \in \mathcal{G}$, fix a finite generating set $F_{\Gamma_0} \Subset \Gamma$.
Given $K \Subset \Gamma$ say that $w \in \cL_K(Y)$ is \emph{$F_{\Gamma_0}$-periodic} if $w_v = w_{v + \gamma}$ whenever $\gamma \in F_{\Gamma_0}$,  and $v,v + \gamma \in K$.
Suppose that $Y$ is $\mathcal{G}$-free. Then for any $y \in Y$ and every $\Gamma_0 \in \mathcal{G}$ there exists a finite set $K_{y,\Gamma_0} \Subset \Gamma$ such that the restriction of $y$ to $K_{y,\Gamma_0}$ is not $F_{\Gamma_0}$ periodic. Take $K_y = \bigcup_{\Gamma_0 \in \mathcal{G}}K_{y,\Gamma_0}$, then the restriction of $y$ to $K_{y}$ is not $F_{\Gamma_0}$ periodic for all $\Gamma_0 \in \mathcal{G}$. For each $y \in Y$ let
\[ U_y = \left\{ \tilde y \in Y~:~ \tilde y_{K_y} = y_{K_y} \right\}.\]
Then $U_y \subseteq Y$ is an open neighborhood of $y$.
By compactness of $Y$, the open cover $\{U_y\}_{y \in Y}$ admits a finite subcover $U_{y_1},\ldots,U_{y_N}$.
Let $K =\bigcup_{j=1}^N K_{y_j}$, then for every $y \in Y$ we have that $y_K$ is  not $F_{\Gamma_0}$ periodic for all $\Gamma_0 \in \mathcal{G}$. 
This proves  that $K$  witnesses that $Y$ is $\mathcal{G}$-free.
\end{proof}

\begin{definition}\label{def:Y_K_D}
  Given a $\Gamma$-subshift $Y \subseteq A^\Gamma$ and finite non-empty sets $K,D \Subset \Gamma$ let
  \[
  Y^{(K,D)} = \left\{ y \in A^{\Gamma} ~:~ \forall v \in \Gamma \; \exists w \in D \mbox{ s.t. } \sigma_{v+w}(y)_K \in \cL_{K}(Y)\right\}.
  \]
\end{definition}
In words, given a $\Gamma$-subshift $Y \subseteq A^\Gamma$, 
the subshift $Y^{(K,D)} \subseteq A^{\Gamma}$ consists of elements of $A^\Gamma$ in which every $K+D$ pattern contains a $Y$-admissible $K$-pattern.

\begin{lemma}\label{lem:Y_K_D_G_free}
    Let  $\mathcal{G} \Subset \Sub(\Gamma)$ be a finite set of finitely generated subgroups,  let $Y \subseteq A^\Gamma$ be $\mathcal{G}$-free subshift, and let
    $K \Subset \Gamma$ be a finite set that witnesses that $Y$ is $\mathcal{G}$-free. For any non-empty finite set $D \Subset \Gamma$, the subshift
     $Y^{(K,D)}$ is a $\mathcal{G}$-free subshift such that $Y \subseteq Y^{(K,D)}$.
\end{lemma}
\begin{proof}
    It is clear that $Y \subseteq Y^{(K,D)}$, because for any $y \in Y$ and any $v \in \Gamma$ and $w \in D$ we have that $\sigma_{v+w}(y)_K \in \cL_K(Y)$.
    Because $K$ witness that $Y$ is $\mathcal{G}$-free, it follows that $D+K$ witnesses that $Y^{(K,D)}$ is $\mathcal{G}$-free.
\end{proof}

We proceed with some technical lemmas that will eventually be used to establish significant properties for subshifts with the map extension property. These properties will be exploited in our proof of the main theorem.

\begin{lemma}\label{lem:map_ext_G_free}
Let $\mathcal{G} \Subset \Sub(\Gamma)$ be a finite set of finitely generated subgroups. Then:
\begin{enumerate}
    \item[(i)] \label{enum:map_ext_G_free_1}
    Any $\mathcal{G}$-free $\Gamma$-subshift $Y$ is contained in a $\mathcal{G}$-free $\Gamma$-subshift of finite type $\hat Y$ with the map extension property (possibly over a bigger alphabet).
    \item[(ii)] \label{enum:uniform_retract} If $Y$ is an absolute retract for the class of $\mathcal{G}$-free $\Gamma$-subshifts then there exists a finite set $F \Subset \Gamma$ so that for any $\mathcal{G}$-free $\Gamma$-subshift $X$ with $Y \subseteq X$ there exists a retraction map $r: X \to Y$ such that whenever $y_F \in \cL(Y)$ then $r(y)_0=y_0$.
\end{enumerate}

\end{lemma}

\begin{proof}
Let 
 $\mathcal{G} \Subset \Sub(\Gamma)$ be a finite set, and let $Y$  be a $\mathcal{G}$-free $\Gamma$-subshift.
 
 Choose a finite symmetric generating set $F_{\Gamma_0} \Subset \Gamma$ for each $\Gamma_0 \in \mathcal{G}$. 
 By \Cref{lem:witness_G_free} there exists a finite set $K_0 \Subset \Gamma$ with $0 \in K_0$ that witnesses that $Y$ is $\mathcal{G}$-free.
Let \[P= \bigcup_{\Gamma_0 \in \mathcal{G}}F_{\Gamma_0},\]
and let $\mathcal{Q}$ denote the collection of symmetric subsets $Q \subseteq P$ that satisfy $Q \cap (\Gamma_0 \setminus \{0\}) \ne \emptyset$ for all $\Gamma_0 \in \mathcal{G}$.

We first prove $(i)$.
We are about to define a $\Gamma$-subshift $\hat Y \subseteq \hat A^{\Gamma}$ 
and show that $\hat Y$ is $\mathcal{G}$-free, contains $Y$ and has the map extension property.
Let us first define the ``alphabet'' $\hat A$:\[
\hat A = A \uplus \mathcal{Q} \uplus \{\hat a\},
\]
where $Y \subseteq A^\Gamma$.

We regard the finite set $\mathcal{Q}$ as another alphabet disjoint for $A$, and $\hat a$ as a new ``symbol'' not contained in $A$ or in $\mathcal{Q}$.

Let $\hat Y \subseteq \hat A^{\Gamma}$ 
be the $\Gamma$-subshift of finite type defined by the following constraints: For every $y \in \tilde Y$ and any $v \in \Gamma$: 
\begin{enumerate}
    \item  There exists $u \in K_0$ such that either $\sigma_{v-u}(y)_{K_0} \in \cL_{K_0}(Y)$ or there exists $Q \in \mathcal{Q}$  and $\gamma \in Q \cup \{0\}$ such that $\sigma_{v-u+ \gamma}(y)=Q$.
    \item If there exists $Q \in \mathcal{Q}$ such that $y_v = Q$, then for every $\gamma \in Q$ we have that $y_{v+\gamma} \ne Q$.
\end{enumerate}

It is clear that $Y \subseteq \hat Y$.
If $y \in Y$ then for every $v \in \Gamma$ and every $u \in K_0$ we have that $\sigma_{v-u}(y)_{K_0} \in \cL_{K_0}(Y)$, and $y_v \not \in Q$, so $y \in \hat Y$. This shows that $Y \subseteq \hat Y$. Let us prove that $\hat Y$ is $\mathcal{G}$-free:
Given $y \in \hat Y$ and $\Gamma_0 \in \mathcal{G}$,
we need to show that 
$\Gamma_0 \not \le \stab(y)$.
Consider an arbitrary point $y \in \hat Y$.
If there exists $v \in \Gamma$ such that $\sigma_v(y)_{K_0} \in \cL_{K_0}(Y)$, 
then $\Gamma_0 \not \le \stab(y)$ for any $\Gamma_0 \in \mathcal{G}$, because $K_0$ witnesses that $Y$ is $\mathcal{G}$-free. 
Otherwise, there exists $v \in \Gamma$ and $Q \in \mathcal{Q}$ such that $\sigma_v(y)=Q$.
In this case for every  $\gamma \in Q$ we have $y_{v+\gamma} \ne y_v$, 
so $\stab(y) \cap Q = \emptyset$. Since $(\Gamma_0 \setminus \{0\}) \cap Q \ne \emptyset$ this shows that $\Gamma_0 \not \le \stab(y)$.

Let us show that $\hat Y$ has the map extension property. Suppose that $X$ is a $\Gamma$-subshift such that $X \rightsquigarrow \hat Y$. In particular, $X \subseteq B^\Gamma$ is $\mathcal{G}$-free subshift, and suppose that $\tilde \pi: \tilde X \to \hat Y$ is a map from a $\Gamma$-subshift $\tilde X \subseteq X$. Let $W \Subset \Gamma$ be a coding window for $\tilde \pi$, and let $K = K_0 + W$.
We will now describe a map $\pi:X \to \hat Y$ that extends $\tilde \pi$
As before, since $X$ is compact and $\mathcal{G}$-free there exists a finite set $K_2 \Subset \Gamma$ that witnesses that $X$ is $\mathcal{G}$-free.
 For every $Q \in \mathcal{Q}$ let $L_Q \subseteq B^{K_2}$ be the set of patterns $w \in \cL_{K_2}(X)$ such that:
\begin{itemize}
    \item For every $\gamma \in Q$ there exists $v \in K_2$ such that $v+\gamma \in K_2$ and $w_v \ne v_{v+\gamma}$.
    \item For every $\gamma \in P\setminus Q$ and every  $v \in K_2$ such that $v+\gamma \in K_2$ we have $w_v = v_{v+\gamma}$.
\end{itemize}
The sets $\{L_Q\}_{Q \in \mathcal{Q}}$ are mutually disjoint and 
$\cL_{K_2}(X)) \subseteq \biguplus_{Q \in \mathcal{Q}}L_Q$.
For every $Q \in \mathcal{Q}$ let $V_Q =\{ x \in X~:~ x_{K_2} \in L_Q\}$. 
Then $V_Q$ is a clopen set in $X$. Furthermore for every $x \in V_Q$ and every $\gamma \in Q$ we have $\sigma_\gamma(x) \ne x$.
By \Cref{lem:krieger_marker_lemma} there exists a clopen set $C_Q \subseteq V_Q$ such that 
\[
C_Q \cap \sigma_\gamma(C_Q) = \emptyset \mbox{ for every } \gamma \in Q \mbox{ and } V_Q \subseteq C_Q \cup \bigcup_{\gamma \in Q}\sigma_\gamma(C_Q).
\]
Define a map $\pi:X\to \hat Y$ as follows:
\[
\pi(x)_v = \begin{cases}
    \tilde \pi(\sigma_v(x))_0 & \mbox{ if } \exists u \in K_0 \mbox{ s.t. } \sigma_{v-u}(x)_{K_0+W} \in \cL_{W}(\tilde X)\\
    Q & \mbox{ if } \sigma_v(x) \in C_Q \mbox{ for some } Q \in \mathcal{Q} \mbox{ and we are not in the previous case},\\
    \hat{a} & \mbox{ otherwise.}
\end{cases}
\]

Notice that if there exists $u \in K_0$ such that $\sigma_{v-u}(x)_{K_0+W} \in \cL_{K_0+W}(\tilde X)$, then  it follows that $\sigma_v(x)_{W} \in \cL_W(\tilde X)$ and so $\tilde \pi(\sigma_v(x))_0$ is well defined, because $W$ is a coding window for $\tilde \pi$.
Moreover, if $\sigma_v(x)_K \in \cL_{K_0+W}(\tilde X)$ then there exists $\tilde x \in \tilde X$ such that such that $\pi(x)_{K_0}= \tilde \pi(\tilde x)_{K_0}$, and in particular $\pi(x)_0 = \tilde \pi(x)_0$.
Since $X \subseteq\bigcup_{Q \in \mathcal{Q}} V_Q$ 
and $V_Q \subseteq C_Q \cup \bigcup_{\gamma \in Q}\sigma_\gamma(C_Q)$ for every $Q \in \mathcal{Q}$,
it follows that for every $x \in X$ and $v \in \Gamma$, either there exists $u \in K_0$ such that $\sigma_{v-u}(x)_{K} \in \cL_{K_0+W}(X)$, in which case $\sigma_{v-u}(\pi(x)) \in \cL_{K_0}(Y)$ or there exists $Q \in \mathcal{Q}$ and $\gamma \in Q \cup \{0\}$ such that $\sigma_{v+\gamma}(\hat \pi(y))=Q$. We have thus shown that $\pi(x) \in \hat Y$ for every $x \in X$, completing the proof of $(i)$.

We now prove $(ii)$. Let $Y$ be an absolute retract for the class of $\mathcal{G}$-free subshifts, and let $\hat Y$ be the subshift defined in the first part of the proof. Since $Y \subseteq \hat Y$ and $\hat Y$ is $\mathcal{G}$-free, there exists a retraction map $\hat r:\hat Y \to Y$. Let $W \Subset \Gamma$ be a coding window for $\hat r$. Let $X$ be a $\mathcal{G}$-free subshift that contains $Y$. 
Again, choose a finite set $K_2 \Subset \Gamma$ that witnesses that $X$ is $\mathcal{G}$-free.
 For every $Q \in \mathcal{Q}$ let $L_Q$, $V_Q \subset X$ and $C_Q$ be an in the previous part of the proof.
 Define a map $\tilde r:X\to \hat Y$ as follows:
\[\tilde r(x)_v = \begin{cases}
   \sigma_v(x)_0 & \mbox{ if } \exists u \in K_0 \mbox{ s.t. } \sigma_{v-u}(x)_{K_0} \in \cL_{K_0}(Y)\\
    Q & \mbox{ if } \sigma_v(x) \in C_Q \mbox{ for some } Q \in \mathcal{Q} \mbox{ and we are not in the previous case},\\
    \hat{a} & \mbox{ otherwise.}
\end{cases}
\]

As in the previous part, it follows that indeed $\hat r(x) \in \hat Y$ for every $x \in X$, so $\tilde r:X\to \hat Y$ is well defined.
Clearly $\tilde r(x)_0 = x_0$ whenever $x_{K_0} \in \cL_{K_0}(Y)$. Let $K = K_0 +W$.
Let $r = \hat r \circ \tilde r$. Then $r: X \to Y$ is a retraction map, and $r(y)_0= y_0$ whenever $x \in X$ and $x_{K} \in \cL_K(Y)$. This completes the proof of $(ii)$.

\end{proof}



\begin{proposition}\label{prop:map_ext_iff_RAR}
    A $\Gamma$-subshift $Y$ has the map extension property if and only if  there exists a finite set $\mathcal{G} \Subset \Sub(\Gamma)$ such that $Y$ is an absolute retract for the class of $\mathcal{G}$-free subshifts.
\end{proposition}
\begin{proof}
    Suppose that $Y$ has the map extension property. By \Cref{prop:map_extension_implies_finitely_many_forbidden_periods}, 
    there exists a finite set $F \Subset \Gamma \setminus \{0\}$ 
    such that for every $\Gamma_0 \le \Gamma$ 
    with $\Gamma_0 \cap F = \emptyset$ there exists $y \in Y$ such that $\Gamma_0 \le \stab(y)$.
    Let $\mathcal{G} \Subset \Sub(\Gamma)$ denote the collection of subgroups $\Gamma_0 \le  \Gamma$ that are generated by a subset of $F$, and so that $\Gamma_0 \not \le \stab(y)$ for every $y \in Y$.
    It follows that $Y$ is  $\mathcal{G}$-free, and that $X \rightsquigarrow Y$ for any $\mathcal{G}$-free $\Gamma$-subshift $X$.
     Let $X$ be a $\mathcal{G}$-free subshift that contains $Y$.
     Since $Y$ has the map extension property and $Y \subseteq X$, it follows that then the identity map on $Y$ extends to a retraction from $X$ to $Y$.
     This shows that $Y$ is an absolute retract for the class of $\mathcal{G}$-free subshifts. 
     
    For the converse direction: Suppose that there exists a finite set $\mathcal{G} \Subset \Sub(\Gamma)$ such that $Y$ is an absolute retract for the   class of $\mathcal{G}$-free subshifts. 
    Let $X$ be a $\Gamma$-subshift such that $X \rightsquigarrow Y$, and let $\pi:\tilde X \to Y$ be a map from a closed, $\Gamma$-invariant subset $\tilde X \subseteq X$. 
    Our goal is to show that $\tilde \pi$ extends to a map $\pi:X \to Y$. By \Cref{lem:map_ext_G_free}, 
    since $X$ is $\mathcal{G}$-free, exists a $\mathcal{G}$-free subshift $\hat Y$ that contains $Y$ and has the map extension property. Since $Y \subseteq \hat Y$ it follows that 
    $\tilde \pi$ extends to a map $\hat \pi:X \to \hat Y$. 
    Since $Y$ is an absolute retract for the class of $\mathcal{G}$-free subshifts, there exists a retraction  map $r:\hat Y \to Y$.
     Setting $\pi= r \circ \hat \pi$, we see that $\pi:X \to Y$ extends $\tilde \pi$.
 \end{proof}

\begin{proposition}\label{prop:locally_correctable_entropy_subgroups}
    Let $\Gamma$ be a countable abelian group and let $Y$ be a $\Gamma$-subshift with the map extension property. Then 
    \[ \lim_{\Gamma_0 \to \{0\}}h(\overline{Y}_{[\Gamma_0]}) = h(Y).\]
    If furthermore $\Ker(Y)=\{0\}$ then
    \[ \lim_{\Gamma_0 \to \{0\}}h(\overline{Y}_{\Gamma_0}) = h(Y).\]
\end{proposition}
\begin{proof}
Suppose that $Y \subseteq A^{\Gamma}$ is a $\Gamma$-subshift with the map extension property.
The inequality 
\[\limsup_{\Gamma_0 \to \{0\}}h(\overline{Y}_{[\Gamma_0]}) \le h(Y)\]
hold by \Cref{prop:subgroup_entropies_upperbound} for any subshift, regardless of the map extension property.
So in order to conclude the first part of the statement, we only need to prove the inequality
\[\liminf_{\Gamma_0 \to \{0\}}h(\overline{Y}_{[\Gamma_0]}) \ge h(Y).\]

By \Cref{lem:contractible_Y_Delta}, it suffices to prove this under the additional assumption that $\Gamma$ is a finitely generated abelian group. Furthermore, using \Cref{prop:subgroup_entropies_upperbound}, it suffices to prove that for any $\epsilon >0$ there exists a finite set $K \Subset \Gamma$ such that for any $K$-separated subgroup $\Gamma_0 < \Gamma$ of finite index we have $h(\overline{Y}_{[\Gamma_0]}) \ge  h(Y) - \epsilon$.
By \Cref{prop:map_ext_iff_RAR} there exists a finite set $\mathcal{G} \Subset \Sub(\Gamma)$ such that $Y$ is an absolute retract for the class of $\mathcal{G}$-free subshifts. 
In particular, $Y$ is $\mathcal{G}$-free, and so by \Cref{lem:witness_G_free} there exists a finite set $K_0 \Subset \Gamma$ that witnesses that $Y$ is $\mathcal{G}$-free. 
By part $(ii)$ of \Cref{lem:map_ext_G_free},  there exists another finite set $K_1 \Subset \Gamma$ such that for any $\mathcal{G}$-free subshift 
$\hat Y$ that contains $Y$ there exists a retraction map $r:\hat Y \to Y$ so that $r(y)_0=r_0$
for any $y \in \hat Y$ with $y_{K_1} \in \cL_{K_1}(Y)$. 

For a finitely generated abelian group $\Gamma$, it is not difficult to show For any $\delta >0$ and any finite set $F \subset \Gamma$ there exists a finite set $K \Subset \Gamma$ so that any $K$-separated subgroup finite-index subgroup $\Gamma_0$ admits a fundamental domain $D$ which is $(F,\delta)$-invariant. 
This can be deduced from \Cref{lem:Voronoi_k_epsilon_inv} in \Cref{sec:partial_tilings}, by taking $D$ to be the  disjointified Voronoi cell of $0$ with respect to $\Gamma_0$ (see \Cref{def:disj_Voronoi}).
By \Cref{lem:top_entropy_inv_set}, for any $\epsilon >0$, by choosing $F$ and $\delta$ appropriately as functions of $K_1$ and $\delta$, for  any $(F,\delta)$-invariant we have 
\[
 \frac{1}{|D|} \log | \cL_{D \setminus \partial_{K_1} D}(Y)
| > h(Y) -\epsilon.
\]

So we can complete the proof by showing that for any finite index subgroup $\Gamma_0 < \Gamma$ that admits a fundamental domain $D \Subset \Gamma$ 
with $K_0 \subseteq D$ and
$D \setminus \partial_{K_1} D \ne \emptyset$ we have
\begin{equation}\label{eq:h_Y_Gamma_lower_ineq}
h(\overline{Y}_{[\Gamma_0]}) \ge \frac{1}{|D|} \log | \cL_{D \setminus \partial_{K_1} D}(Y)
|.    
\end{equation}

We assume that $Y \subseteq A^\Gamma$.
Given a finite index subgroup $\Gamma_0 < \Gamma$ that admits a fundamental domain $D \Subset \Gamma$ such that $K_0 \subseteq D$ and
$D \setminus \partial_{K_1} D \ne \emptyset$
 Consider the subshift

\[
\hat Y = \left\{ y \in A^\Gamma~:~ \exists v \in \Gamma  \mbox{ s.t. for every } w \in \Gamma_0 ~\sigma_{v+w}(y)_{D} \in \cL_D(Y)  \right\}.
\]

Because $K_0 \subseteq D$ and $K_0$ witnesses that $Y$ is $\mathcal{G}$-free, it follows that $\hat Y$ is also $\mathcal{G}$-free. Clearly, $Y \subseteq \hat Y$. Thus there exists a  retraction map $r:\hat Y \to Y$ so that $r(y)_0=r_0$
for any $y \in \hat Y$ with $y_{K_1} \in \cL_{K_1}(Y)$. For any $p \in \cL_D(Y)$ we can define a point $y \in \hat Y_{[\Gamma_0]}$ such that $\sigma_w(y)_D = p$ for every $w \in \Gamma_0$. It follows that $r(y)_{D \setminus \partial_{K_1} D} = y_{D \setminus \partial_{K_1} D}$. Since $r(y) \in Y_{[\Gamma_0]}$, it follows that 
\[
|\cL_{D}(Y_{[\Gamma_0]})| \ge |\cL_{D \setminus \partial_{K_1} D}(Y)|.
\]
Since $h(Y_{[\Gamma_0]}) = \frac{1}{|D|} \log |\cL_{D}(Y_{[\Gamma_0]})|  $, we have proven \eqref{eq:h_Y_Gamma_lower_ineq}.

The proof that \[ \lim_{\Gamma_0 \to \{0\}}h(\overline{Y}_{\Gamma_0}) = h(Y)\]
under the assumption that $\Ker(Y) = \{0\}$ is similar, so we omit the details.
\end{proof}

\begin{definition}\label{def:K_disjoint}
Let $F_1,F_2,K \subsetneq \Gamma$ be  subsets of $\Gamma$ and  $\emptyset \ne K \Subset \Gamma$. We say that $F_1,F_2$ are \emph{$K$-disjoint} if $(K+F_1) \cap (K+F_2) = \emptyset$.
\end{definition}
\begin{definition}\label{def:SI}
A subshift $X \Subset A^\Gamma$ is called \emph{strongly irreducible} if there exists $K \Subset \Gamma$ such that for any pair of $K$-disjoint sets $F_1,F_2 \subseteq \Gamma$ and any $x^{(1)},x^{(2)} \in X$ there exists $x \in X$ such that $x_{F_1} = x^{(1)}_{F_1}$ and $x_{F_2} = x^{(2)}_{F_2}$. We say that $K \Subset \Gamma$ as above is an \emph{irreducibility window} for $X$. 
\end{definition}

\begin{lemma}\label{lem:G_free_embeds_in_SI_SFT}
    Let $\mathcal{G} \Subset \Sub(\Gamma)$ be a finite set, and let $Y \subseteq A^\Gamma$ 
    be a $\mathcal{G}$-free subshift. Then there exists a finite set $\hat A$ with $A \subset \hat A$ and a $\mathcal{G}$-free subshift $\hat Y \subseteq \hat A^\Gamma$ such that $Y \subseteq \hat Y$ and so that $\hat Y$ is a strongly irreducible shift of finite type.
\end{lemma}
\begin{proof}
Let  $Y \subseteq A^\Gamma$ be a $\mathcal{G}$-free subshift. By \Cref{lem:witness_G_free} there exists 
a finite set $W \Subset \Gamma$ 
that witnesses that $Y$ is $\mathcal{G}$-free.
Replacing $W$ by $W \cup (-W) \cup \{0\}$,  we can
assume that $0 \in W$ and that $W$ is symmetric.
Let $\hat A$ be a finite set with $|\hat A| \ge |W|$ such that $A \subseteq \hat A$,
and  let $\mathcal{F} \subset \hat A^W$ denote the set of patterns that are not $\mathcal{G}$-free in the sense that there exists $\Gamma_0 \in \mathcal{G}$ such that for every $v \in \Gamma_0$ if $u \in W$ and $u+v \in W$ then $w_u=w_{u+v}$.
Let $\hat Y \subseteq \hat A^\Gamma$ the subshift of finite type defined by the set of forbidden words $\mathcal{F}$.
Then $\hat Y$ is $\mathcal{G}$-free, as witnessed by the set $W$. Clearly $Y \subseteq \hat Y$. 
We claim that $\hat Y$ is strongly irreducible. 
To see this, it suffices to show that for any  $F \subseteq \Gamma \setminus \{0\}$ and any pattern $w \in \hat A^F$ which is  $\mathcal{F}$-free (in the sense that sub-patterns from $\mathcal{F}$ do not occur in it) can be extended to a $\mathcal{F}$-free pattern $\tilde w \in \hat A^{F \cup \{0\}}$.

Let   $F \subseteq \Gamma \setminus \{0\}$ and an $\mathcal{F}$-free pattern $w \in A^F$. Since $|W \setminus \{0\}| < |\hat A|$ there exists  $\hat a \in \hat A$ such that $w_{v} \ne \hat a$ for any $v \in (W \cap F) \setminus \{0\}$. Let  $\tilde w \in \hat A^{F \cup \{0\}}$ be defined by
\[
\tilde  w_v =
\begin{cases}
\hat a & \mbox{ if } v= 0\\
w_v & \mbox{ otherwise.}
\end{cases}
\]
Then $\tilde W$ is also $\mathcal{F}$-free $\tilde w_F =w$.
\end{proof}

\begin{prop}\label{prop:abs_retract_prop}
Let $\Gamma$ be an infinite countable abelian group and let $Y$ be a $\Gamma$-subshift with the map extension property.
Then $Y$ is a  strongly irreducible shift of finite type.
\end{prop}
\begin{proof}
Suppose that $Y$ has the map extension property. Then by \Cref{prop:map_ext_iff_RAR}, there exists a finite set $\mathcal{G} \Subset \Sub(\Gamma)$ such that  $Y$ is an absolute retract for the class of $\mathcal{G}$-free subshifts.

Then by \Cref{lem:G_free_embeds_in_SI_SFT}, $Y$ embeds in a $\mathcal{G}$-free strongly irreducible SFT $\hat Y$.
By \Cref{prop:closed_under_retracts} it follows that $Y$ is also a strongly irreducible subshift of finite type.
\end{proof}

We remark that alternative proofs of \Cref{prop:abs_retract_prop} (as well as stronger statements) can be found in \cite{poirier2024contractible}. 

\begin{lemma}\label{lem:stab_of_abs_retract}
    Let $Y$ be a $\Gamma$-subshift with the map extension property. 
    Then $\overline{Y}_{[\Gamma_0]}$ has the map extension property for any subgroup $\Gamma_0 \le \Gamma$.
\end{lemma}

\begin{proof}
Let $\Gamma$ be a countable abelian group.
    Suppose that $Y$ is  a $\Gamma$-subshift with the map extension property and $\Gamma_0 < \Gamma$.
    We want to prove that $\overline{Y}_{[\Gamma_0]}$ also has the map extension property.
    Suppose that $\tilde X \subseteq X$ are $\Gamma/\Gamma_0$ subshifts with $X \rightsquigarrow \overline{Y}_{[\Gamma_0]}$ and that $\tilde \pi:\tilde X \to \overline{Y}_{[\Gamma_0]}$ is a map of $\Gamma/\Gamma_0$-subshifts. 
    For any set finite set $A$, there is a natural continuous injective map $\psi_{\Gamma_0}: A^{\Gamma/\Gamma_0} \to A^{\Gamma_0}$ given by $\psi_{\Gamma_0}(x)_v = x_{v +\Gamma_0}$ for any $x \in A^{\Gamma/\Gamma_0}$, $v \in \Gamma$ and $x \in A^{\Gamma/\Gamma_0}$.
    Then $\psi_{\Gamma_0}^{-1}(\tilde X) \subseteq \psi_{\Gamma_0}^{-1}(X)$ are $\Gamma$-subshifts with $\Gamma_0 <  \ker(\psi_{\Gamma_0}^{-1}(X))$, and 
    $\psi_{\Gamma_0}^{-1}(X) \rightsquigarrow Y$ 
    because $X \rightsquigarrow \overline{Y}_{[\Gamma_0]}$. The map $\tilde \pi$ naturally lifts to a map 
    $\overline{\pi}:\psi_{\Gamma_0}(\tilde X) \to Y$. By the map extension property,
    the  map $\overline{\pi}$ extends to a map $\hat \pi:\psi_{\Gamma_0}(X) \to Y$.
    Since  $\Gamma_0 <  \ker(\psi_{\Gamma_0}^{-1}(X))$ 
    it follows that the image of $\psi_{\Gamma_0}(X))$ under $\hat \pi$ is contained in $Y_{[\Gamma_0]}$.
    Hence $\hat \pi$  induces a well-defined map $\pi:X \to \overline{Y}_{[\Gamma_0]}$ that extends $\tilde \pi$.

\end{proof}


In preparation for the next result, we  recall Boyle's extension lemma \cite[Lemma 2.2]{MR753922}:

\begin{lemma}[Boyle's Extension Lemma \cite{MR753922}]\label{lem:boyle_extension_lemma}
    Let $X$ be a $\mathbb{Z}$-subshift, let $\tilde X \subseteq X$ be a closed shift-invariant set, and let $Y \subseteq B^\Z$ be mixing subshift of finite type such that $X \rightsquigarrow Y$.
    Let  $\tilde \pi:\tilde X \to Y$ be a map. Then $\tilde \pi$ extends to a map $\pi:X \to Y$.
\end{lemma}

Boyle's extension lemma implies an alternative charterization of topologically mixing $\mathbb{Z}$-subshifts of finite type:    
\begin{prop}\label{prop:mixing_Z_SFT_abs_retract}
A $\mathbb{Z}$-subshift  has the map extension property  if and only if it is a topologically mixing subshift of finite type.
\end{prop}
\begin{proof}
 From \Cref{prop:abs_retract_prop} it follows that any $\Gamma$-subshift which has the map extension property is a strongly irreducible SFT. In particular, any  $\Z$-subshift which is a  restricted absolute retract is a mixing shift of finite type.
 Now suppose $Y \subset A^\mathbb{Z}$  is a mixing SFT. 
 Let $\mathcal{G} \subset \Sub(\Z)$ denote the collection of subgroups $\Gamma_0 \le \Z$ such that for every $y \in Y$ the subgroup $\Gamma_0$ is not contained in $\stab(x)$. We claim that $\mathcal{G}$ is a finite set. Indeed, since $Y$ is a mixing $\Z$-SFT, there exists $n_0 \in \mathbb{N}$ such that for every $n \ge n_0$ there exists $y \in Y$ with $\sigma^n(y)=y$. Thus, every subgroup $\Gamma_0 \in \mathcal{G}$ must be of the form $n \Z$ with $n < n_0$. 
 It follows from Boyle's Extension (\Cref{lem:boyle_extension_lemma}) that $Y$
 has the  map extension property.
\end{proof}

Although the map extension property is a rather strong property for $\Z^d$-subshifts, it is satisfied by many  ``naturally occuring'' subshifts of finite type:


\begin{definition}
    Let $X \subseteq A^\Gamma$ be a subshift. We say that $a \in A$ is a \emph{safe symbol} or $X$ if for any $F \Subset \Gamma$ and any $w \in \cL_F(X)$ and any $\tilde w \in A^F$ any pattern $\tilde w \in A^F$ that satisfies $\tilde w_v \in \{ a, w_v\}$ for all $v \in F$ is in $\cL_F(X)$.
    Equivalently, $a \in A$ is a safe symbol for $X$ if for any $x \in X$ changing the value of $x$ at any $v \in \Gamma$ to the ``symbol'' $a$ results in a point of $X$.
\end{definition}

\begin{proposition}
 Any $\Gamma$-SFT with a safe symbol is an absolute retract for the class of subshifts, and in particular has the map extension property.
\end{proposition}
\begin{proof}
    Let $X \subseteq A^\Gamma$ be a $\Gamma$-SFT with a safe symbol $a \in A$, and suppose that $Y \subseteq \tilde A^\Gamma$ is another subshift such that $X \subseteq Y$.
    Let $\mathcal{F} \subset A^W$ be a defining set of forbidden patterns for $X$. 
    We can define a retraction $r:Y \to X$ as follows:
    \[r(y)_v = \begin{cases}
        y_v & (\sigma_{-v-w}(y))_W \not \in \mathcal{F} \mbox{ for all } w \in W\\
        a & \mbox{ otherwise}.
    \end{cases}\]
\end{proof}


The following proposition shows that the subshift of proper $k$-colorings of  the  Cayley graph of $\Gamma$  corresponding to a symmetric generating set with less than $k$ generators has the map extension property. Another proof of the same result appears in \cite{poirier2024contractible}.
\begin{prop}\label{prop:k_coloring_RAR}
For any finite symmetric set $F \Subset \Gamma \setminus \{0\}$ and any $k \in \mathbb{N}$, consider  the subshift $X_{k,F}$ given by 
 \[
    X_{k,F} = \left\{x \in \{1,\ldots,k\}^\Gamma~:~ \forall \gamma 
    \in \Gamma, v \in F~ x_{\gamma} \ne x_{\gamma +v} \right\}.
    \]
If $k > |F|$, then $X_{k,F}$
has the map extension property.
\end{prop}
\begin{proof}
    Let $F \Subset \Gamma \setminus \{0\}$ be a finite symmetric set and $k >|F|$.
    Let $\mathcal{G} \Subset \Sub(\Gamma)$ denote the set of cyclic subgroup generated by elements of $F$:
    \[
    \mathcal{G} = \left\{ \langle v \rangle ~:~ v \in F \right\}.
    \]
    We will prove that $X_{k,F}$ is an absolute retract for the class of $\mathcal{G}$-free subshifts. Since $x_{0} \ne x_{v}$ for any $x \in X_{k,F}$ and $v \in F$, it is clear that $v \not \in \stab(x)$ for any $x \in X_{k,F}$, so $X_{k,F}$ is $\mathcal{G}$-free.
    Let $X$ be any $\mathcal{G}$-free subshift such that $X_{k,F} \subseteq X$.
    By \cref{lem:krieger_marker_lemma} there exists a clopen set $C \subset X$ such that $C \cap \sigma_v(C) = \emptyset$ for every $v \in F$ and $X= C \bigcup_{v \in V} \sigma_v(C)$. Define a map $\alpha:X \to \{0,1\}^\Gamma$ by declaring for $x \in X$ and $v \in \Gamma$
    \[\alpha(x)_v= \begin{cases}
    1 & \sigma_v(x) \in C\\
    0 & \mbox{otherwise}.
    \end{cases}
    \]
    For $x \in X$ let 
    \[B(x) = \{ x_v~:~ v \in F\}.\]
    Let $F= \{ v_0,v_1,\ldots,v_{|F|}\}$ be some enumeration of the elements of $F \cup \{0\}$. 
    For every $x \in X$ let 
    \[ t(x) = \min \{ 0\le j \le |F|~:~ \sigma_v(x) \in C\}.\]
    For every $0 \le j \le |F|$, define a map  $r^{(j)}:X \to X$ as follows be declaring for $x \in X$ and $v \in \Gamma$:
    \[
    r^{(j)}(x)_v =
    \begin{cases}
        x_v & \mbox{ if  } t(X) \ne j  \mbox{ or } x_v \in \{1,\ldots,k\} \setminus B(\sigma_{v}(x)).\\
        \min\left( \{1,\ldots,k\} \setminus B(\sigma_{v}(x)) \right) & \mbox{ otherwise. }
    \end{cases}
    \]
    Note that for every $v \in F \cap \{0\}$ we have that 
    Since $|B(\sigma_{v}(x)| \le |F| < k$, it follows that $ \{1,\ldots,k\} \setminus B(\sigma_{v}(x))  \ne \emptyset$, so $r^{(j)}$ is well-defined.
    
    Define $\tilde r^{(0)} = r^{(0)}$ and  $\tilde r^{(j)}= r^{(j)} \circ \tilde r^{(j-1)}$
    for every $ 1\le j \le |F|$. 
    Let $r =\tilde r^{(|F|)}:X \to X$. We claim that $r:X \to X_{k,F}$ is a retraction.
    By induction it follows that if $x \in C \cup \bigcup_{\ell=1}^j \sigma_{v_\ell}^{-1}(C)$  then $x_0 \in \{1\ldots,k\} \setminus B(x)$. In particular, since $X=  \bigcup_{\ell=1}^{|F|} \sigma_{v_\ell}^{-1}(C)$, it follows that $r(x)_0 \in \{1\ldots,k\} \setminus B(x)$ for every $x \in X$. So for every $x \in X$, $r(x) \in \{1,\ldots,l\}^\Gamma$ and $r(x)_v \ne r(x)_{v+W}$ for every $v \in \Gamma$ and $w \in F$. This proves that $r(x) \in X_{k,F}$ for every $x \in X$. It is easy to check that $r^{(j)}(x) = x$ for every $ 0 \le j \le |F|$ and $x \in X_{k,F}$, so $r(x)=x$ for every $x \in X_{k,F}$. This proves that $r:X \to X_{k,F}$ is indeed a retraction.
\end{proof}


As particular case, when $\Gamma=Z^2$ and  $F = \{(\pm 1,0),(0,\pm 1)\}$  is the standard generating set, the subshift $X_{k,F}$ appearing in the statement of \Cref{prop:k_coloring_RAR} is the subshift of $k$-colorings of the standard Cayley graph of $\Z^d$. It is known that for any $d >1$ the subshift $X_{F_3}$ of proper $3$-colorings of the standard Cayley graph does not contain any strongly irreducible subshift \cite[Proposition 12.2]{MR4311117}, so certainly an analog of \Cref{thm:relative_embedding_thm} does not hold with $Y=X_{3,F}$. The subshift of  $k$-colorings  of the $\Z^2$-lattice, was shown to be strongly irreducible if and only if $k \ge 4$ \cite{MR4247630}.

\section{Entropy minimal subshifts  and necessity of the conditions for embedding}\label{sec:entropy_min_subshifts}
In this section we recall the notion of entropy minimality for subshifts, 
and prove that for an entropy minimal  $\Gamma$-subshift $Y$, the conditions in the statement of \Cref{thm:relative_embedding_thm} are necessary for the existence of an  embedding $\rho:X \to Y$ that extends  $\hat \rho$.

A $\Gamma$-subshift $Y$ is called \emph{entropy minimal}
if the only $\Gamma$-subshift $X \subseteq Y$  that satisfies $h(X)=h(Y)$ is $X=Y$. Equivalently, any proper subshift of $Y$ has strictly lower topological entropy.  

The following easy result is a consequence of the fact that isomorphic subshifts have equal topological entropy and that a topological entropy of a subsystem of $Y$ never exceeds the topological entropy of $Y$. Recall the definition of condition $E(X,Y,\hat \rho,\Gamma_0)$ as defined in \Cref{def:E_X_Y}.

\begin{proposition}
    Let $\Gamma$ be a countable abelian group and let $X,Y$ be $\Gamma$-subshifts, let $Z \subseteq X$ be a closed $\Gamma$-invariant set, and let $\hat \rho:Z \to Y$ be an embedding. If $X \hookrightarrow_{\hat \rho} Y$ then for any subgroup $\Gamma_0 \le \Gamma$ 
    we have that $h(\overline{X}_{\Gamma_0})\le h(\overline{Y}_{\Gamma_0})$.
    Furthermore, if for some subgroup $\Gamma_0 \le \Gamma$ the subshift $\overline{Y}_\Gamma$ is entropy minimal,
    then the condition $E(X,Y,\hat \rho,\Gamma_0)$ holds.
\end{proposition}

It is well-known  that any strongly irreducible $\Gamma$-subshift  is entropy minimal, see for instance \cite{MR2645044,MR1764932}.
In particular, if $Y$ is a $\Gamma$-subshift with the map extension property, then by \Cref{lem:stab_of_abs_retract} and \Cref{prop:abs_retract_prop} we have that $\overline{Y}_{\Gamma_0}$ is strongly irreducible for any $\Gamma_0 \le \Gamma$.

It follows that the conditions in the statement of \Cref{thm:relative_embedding_thm} for extending an embedding of a subsystem to an embedding are necessary when the target $Y$ is has the map extension property:
\begin{corollary}
Let $\Gamma$ be a countable abelian group and let $X,Y$ be $\Gamma$-subshifts, where $Y$ has the map extension property let $Z \subseteq X$ be a closed $\Gamma$-invariant set, and let $\hat \rho:Z \to Y$ be an embedding. If $X \hookrightarrow_{\hat \rho} Y$ then for any subgroup $\Gamma_0 \le \Gamma$, the condition  $E(X,Y,\hat \rho,\Gamma_0)$ holds.
\end{corollary}

\section{Truncated Voronoi diagrams and Partial tilings}
\label{sec:partial_tilings}

Throughout this section $\Gamma$ will be a finitely generated abelian group of the form
$\Gamma = \Z^d \times G$, where $G$ is a finite abelian group.

\begin{definition}
    A \emph{partial tiling} of $\Gamma$ is a collection of pairwise disjoint finite subsets of $\Gamma$. 
We denote the set of partial tilings of $\Gamma$ by $\PTiling(\Gamma)$.
A \emph{tiling} of $\Gamma$ is a partition of $\Gamma$ into pairwise disjoint finite subsets.
For a partial tiling $\tau \in \PTiling(\Gamma)$  we refer to the elements of $\tau$ as tiles.
\end{definition}
The space $\PTiling(\Gamma)$ admits a natural Polish topology, which we now describe:  We describe a basis for this topology.
For $F \Subset \Gamma$ and $\tau \in \PTiling(\Gamma)$ let $N(F,\tau) \subseteq \PTiling(\Gamma)$ denote the collection of partial tilings $\tilde \tau \in \PTiling(\Gamma)$ that satisfy the property that for $T \subseteq F$ we have that $T \in \tilde \tau$ if and only if $T \in \tau$.
 A closed subset  $P \subseteq \PTiling(\Gamma)$ is compact if and only if every $\gamma \in \Gamma$ is covered by one of finitely many tiles (among all possibilities of tilings in $\Gamma$).
The group $\Gamma$ acts on $\PTiling(\Gamma)$ by translations.
Although we will not use this fact, we mention that any compact $\Gamma$-invariant subsystem of $\PTiling(\Gamma)$ is isomorphic to a $\Gamma$-subshift. 

\begin{definition}
  Let $\tau \in \PTiling(\Gamma)$ be a partial tiling.
 Denote $\bigcup \tau \subseteq \Gamma$ the union of tiles in $\tau$. 
    We say that $\tau \in \PTiling(\Gamma)$ is \emph{$(K,\epsilon)$-invariant} for $K \Subset \Gamma$ and $\epsilon >0$
    if  each $T \in \tau$ is $(K,\epsilon)$-invariant.
    Given a partial tiling $\tau \in \PTiling(\Gamma)$, $T \in \tau$ and $K \Subset \Gamma$, we let 
    \[\partial_K^\tau T = \left\{ v \in \Gamma~:~ (v+K) \cap T \ne \emptyset \mbox{ and } (v+K) \not \subseteq \bigcup \tau\right\}.\]
    We refer to the set $\partial_K^\tau T$ as the \emph{exterior $K$-boundary of $T$ with respect to $\tau$}.
    We say that a partial tiling $\tau$ \emph{has $(K,\epsilon)$-small exterior boundary} if every $T \in \tau$ satisfies 
    \[ |\partial_K^\tau T | \le \epsilon |T|.\]
\end{definition}

\begin{definition}
We identify the space $\{0,1\}^\Gamma$ with the space of subsets of $\Gamma$ via the obvious bijection 
\[ x \in \{0,1\}^\Gamma \Leftrightarrow \left\{ v \in \Gamma~:~ x_v =1\right\}.\]
Given $F\Subset \Gamma$, we say that a $\Gamma$-subshift $Z \subseteq \{0,1\}^\Gamma$ is $F$-separated if every $z \in Z$ is $F$-separated (when viewed as a subset of $\Gamma$).
Given a $\Gamma$-subshift $Z$, $K \Subset \Gamma$ and $\epsilon>0$,  
we say that a map $\tau:Z \to \PTiling(\Gamma)$ is a \emph{$(K,\epsilon)$-invariant tiling of $Z$} if for every $z \in Z$ the partial tiling $\tau(z) \in \PTiling(\Gamma)$ is $(K,\epsilon)$-invariant.
If $Z \subseteq \{0,1\}^\Gamma$ is a $\Gamma$-subshift, we say that a map  $\tau:Z \to \PTiling(\Gamma)$ is a \emph{pointed partial tiling} if for every $z \in Z$ and for every tile $T \in \tau(z)$ there exists a unique $v \in T$ such that $z_v =1$. In that case we say that $v$ is the \emph{center} of the tile $z$.
\end{definition}

Our main goal is to prove that any $\Gamma$-subshift $X$ admits a    ``well-behaved''  continuous, equivariant map into the space of $(K,\epsilon)$-invariant partial tilings, for arbitrary $K \Subset \Gamma$ and $\epsilon >0$:

\begin{proposition}\label{prop:good_partial_tilings}
    Let $\Gamma$ be a finitely generated abelian group, and let $X$ be a $\Gamma$-subshift.
    For every $K \Subset \Gamma$ and $\epsilon >0$ there exists a finite 
    set $P \Subset \Gamma \setminus \{0\}$  so that given:
    \begin{itemize}
        \item A finite subset $W\Subset \Gamma$,
        \item $\epsilon_1 >0$,
        \item A set  of patterns $\mathcal{F} \subseteq \cL_{W}(X)$  so that $x_W \not \in \mathcal{F}$ for any $x \in X$ such that $\stab(x) \cap P = \emptyset$.
    \end{itemize}
    There exists a map $\alpha:X \to \{0,1\}^\Gamma$
    and a pointed $(K,\epsilon)$-invariant tiling $\tau: \alpha(X) \to \PTiling(\Gamma)$ 
     such that  for every $x \in X$:

    \begin{enumerate}
   
        \item If $v \not\in \bigcup\tau(\alpha(x))$ then $\sigma_{v}(x)_{W} \in \mathcal{F}$.
        \item If $v \in \Gamma$ satisfies $\alpha(x)_v=1$ then $\sigma(x)_{W} \not\in \mathcal{F}$ and there exist $T \in \tau(\alpha(x))$ such that $v +K \subseteq T$.
        \item The partial tiling $\tau(\alpha(x)) \in \PTiling(\Gamma)$ has $(W, \epsilon_1)$-small exterior boundary.
      
    \end{enumerate}
    
\end{proposition}
The statement of the \Cref{prop:good_partial_tilings} above is slightly involved because in our application it will be essential that the set $P \Subset \Gamma\setminus \{0\}$ can be chosen depending only on $K \Subset \Gamma$ and $\epsilon >0$, and independently from  $\epsilon_1>0$, $W\Subset \Gamma$ and  $\mathcal{F} \subseteq \cL_{W}(Y)$. In our application the parameters $\epsilon_1>0$, $W\Subset \Gamma$ and $\mathcal{F} \subseteq \cL_{W}(Y)$ will actually depend on the set $P \Subset \Gamma\setminus \{0\}$ and the previously chosen parameters.

In the case where $X$ is an aperiodic subshift, we actually get a map into the space of $(K,\epsilon)$-invariant partial tiling: 
\begin{corollary}\label{cor:apriodic_inv_tiling}[Lightwood \cite{MR1972240}]
    Let $\Gamma$ be a finitely generated abelian group, and let $X$ be an aperiodic $\Gamma$-subshift. Then for any $K \Subset \Gamma$ and $\epsilon$ there exists a map from $\tau:X \to \PTiling(\Gamma)$  such that for every $x \in X$ $\tau(x)$ is a $(K,\epsilon)$-invariant tiling.
\end{corollary}
For context, we explain why \Cref{cor:apriodic_inv_tiling} is a direct consequence of \Cref{prop:good_partial_tilings}:
\begin{proof}
Let $X$ be an aperiodic $\Gamma$-subshift and let $K\Subset \Gamma$ and $\epsilon >0$ be given.
    Apply \Cref{prop:good_partial_tilings} on $X$ with the given $K$ and $\epsilon$.
    We obtain a corresponding set  $P \Subset \Gamma \setminus \{0\}$ (and a map $\alpha:X \to \{0,1\}^\Gamma$). Choose any $W,\tilde W \Subset \Gamma$ and any $\epsilon_1 >0$. Take $\mathcal{F}= \emptyset$ then $\mathcal{F}\subseteq \cL_W(X)$ and since $\stab(x) \cap P = \emptyset$ for any $x \in X$, in particular $x_W \not \in \mathcal{F}$ for any $x \in X$.
    Then there exists a map $\tau:X \to \PTiling(X)$ such that $\tau(x)$ is $(K\epsilon)$-invariant and such that each $\tau(x)$ is a tiling.
\end{proof}

We remark that the assumption that $\Gamma$ is finitely generated can be removed with out much difficulty both in \Cref{cor:apriodic_inv_tiling} and in \Cref{prop:good_partial_tilings}.
We also remark that the conclusion of \Cref{cor:apriodic_inv_tiling} holds for various non-abelian amenable groups. See \cite{MR4539364} and also \cite{bland2022embedding}, although in the more general setting the proof proceeds via different methods.

However, in the presence of points with non-trivial stablizer, one can easily see that it is impossible to get a map into the space of $(K,\epsilon)$-invariant partial tilings for arbitrary $K \Subset \Gamma$  and $\epsilon >0$. 

The proof below of \Cref{prop:good_partial_tilings} proceeds via the notion of Voronoi diagrams (that still makes sense in any metric space, in particular for finitely generated groups). The arguments involve some elementary convex analysis in euclidean spaces. Similar results, proven using similar methods, can be found in the literature \cite{MR1951240,MR3540453,MR1972240}. We present  self-contained proofs here for the sake of completeness and also due to certain less standard variations that we seem to require for our specific application.

Recall that  $\Gamma$ is a finitely generated abelain group of the form $\Gamma=\Z^d \times G$, where $G$ is a finite abelian group. Let $P_{\Gamma,\Z^d}:\Gamma \to \Z^d$ and $P_{\Gamma, G}: \Gamma \to G$ denote the obvious projection maps given by
\[
P_{\Gamma,\Z^d}(v,g) = v \mbox{ for } v\in \Z^d \mbox{ and } g \in G
\]
and
\[
P_{\Gamma,G}(v,g) = g \mbox{ for } v\in \Z^d \mbox{ and } g \in G
\]
On the group $\Gamma = \Z^d \times G$ we consider the metric $\dist_\Gamma:\Gamma \times \Gamma \to [0,+\infty)$ given by:
\[
\dist_\Gamma\left( \gamma_1,\gamma_2\right) = \dist_{\mathbb{R}^d}(P_{\Gamma,\Z^d}(\gamma_1),P_{\Gamma,\Z^d}(\gamma_2)) + \delta(\gamma_1,\gamma_2),~  \gamma_1,\gamma_2 \in \Gamma,
\]
where $\dist_{\mathbb{R}^d}:\mathbb{R}^d \times \mathbb{R}^d \to [0,+\infty)$ is the Euclidean distance in $\mathbb{R}^d$ and
\[
\delta(\gamma_1,\gamma_2) = \begin{cases}
0 & \gamma_1=\gamma_2\\
1 & \mbox{otherwise.}
\end{cases}
\]

\begin{definition}\label{def:Voronoi_diagram}
Let $C \subseteq \Gamma$ be a subset of $\Gamma$.
    The \emph{Voronoi cell} of $c \in C$ with respect to $C$, denoted by $V(c,C)$, is defined be the set of point in $v \in \Gamma$  such that $\dist_\Gamma$ distance between  $c$ and $v$ is smaller than the distance to any other point in $C$:
       \[
V(c,C) =\left\{
\gamma \in \Gamma:~ \dist_\Gamma(c,\gamma) \le \inf_{\tilde c \in C}\dist_\Gamma(\tilde c,\gamma) \right\}.
\]
    The \emph{Voroni diagram} of $C$ is the collection of Voronoi cells.

Given $R >0$ the  \emph{$R$-truncated Voronoi cell} of $c \in C$ with respect to $C$, denoted by $V_R(c,C)$ is the intersection of the Voronoi cell of $c$ with the ball of $\dist_\Gamma$ radius $R$ around $C$:
       \[
V_R(c,C) =\left\{
\gamma \in  V(c,C) ~:~ \dist_\Gamma(c,\gamma) \le R \right\}.
\]
The \emph{$R$-truncated Voronoi diagram} of $C$ is the collection of $R$-truncated Voronoi cells. 
\end{definition}

In general, the $R$-truncated Voronoi cells corresponding to a subset $C$ are not necessarily  pairwise disjoint because it could happen that for some $v \in \Gamma$ there exist  $c_1,c_2 \in C$ with $c_1, \ne c_2$ such that  \[\dist_\Gamma(c_1,\gamma)=\dist_\Gamma(c_1,\gamma)=\inf_{\tilde c \in C}\dist_\Gamma(\tilde c,\gamma) \le R.\] 
For our application we would like to extract a partial tiling from the $R$-truncated Voronoi diagram. For this purpose we need to ``disjointify'' the cells. 

\begin{definition}\label{def:disj_Voronoi}

Fix an arbitrary total order $\prec_\Gamma$ on the group $\Gamma$, with the property that whenever $\gamma_1,\tilde \gamma_1,\gamma_2,\tilde \gamma_2 \in \Gamma$ satisfy that $P_{\Gamma,\Z^d}(\gamma_1)=P_{\Gamma,\Z^d}(\tilde \gamma_1) \ne P_{\Gamma,\Z^d}(\gamma_2)=P_{\Gamma,\Z^d}(\tilde \gamma_2)$ and $\gamma_1 \prec_{\Gamma} \gamma_2$ then $\tilde \gamma_1 \prec \tilde \gamma_2$.  We do not require that the total order $\prec_\Gamma$ is invariant, as this can only be done in the case where $G$ is the trivial group, otherwise $\Gamma$ has non-trivial torsion, hence is not orderable.

The \emph{disjointified $R$-truncated Voronoi cell} of $c \in C$, denoted by $\overline{V}_R(c,C)$ is the subset of $V_R(c,C)$ obtained by ``breaking ties according to $\prec_{\Gamma}$'':
\[
\overline{V}_R(c,C) = \left\{ \gamma \in V_R(c,C)~:~ (\gamma-c) \prec_{\Gamma}  (\gamma-\tilde c) \mbox{ for every } \tilde c \in M(\gamma,C)\setminus \{c\} \right\},
\]     
where $M(\gamma,C) \subseteq C$ is the set is the set of minimizers of $\dist_{\Gamma}(\gamma,\cdot)$ in $C$:
\[
M(\gamma,C) = \left\{ c \in C~:~ \dist_{\Gamma}(\gamma,c) \le \inf_{\tilde c \in C}\dist_{\Gamma}(\gamma,\tilde c)
\right\}.
\]
The \emph{disjointified $R$-truncated Voronoi digaram of $C \subset \Gamma$}, 
denoted $\overline{V}(C)$ is the collection \[ \overline{V}(C) = \{\overline{V}_R(c,C)~:~ c\in C\}\] of disjointified  $R$-truncated Voronoi cells.
\end{definition}

\begin{definition}
    Call a finite set $F \Subset \Gamma$ \emph{convex} if there exists a compact convex subset $\tilde F \subset \mathbb{R}^d$ such that $F= (\tilde F \cap \Z^d) \times G$.
\end{definition}

\begin{lemma}\label{lem:disjoint_Vor_PTiling}
   For any $C \subseteq \Gamma$ and $R>0$ we have:
   \begin{enumerate}
       \item The disjointified $R$-truncated Voronoi diagram of $C$ is a partial tiling of $\Gamma$.
       \item $ (C+B_R^\Gamma )\subseteq \biguplus_{c \in C}\overline{V}_R(c,C)= \bigcup_{c \in C} V_R(c,C)$,
       where \[B_R^\Gamma = \left\{ \gamma \in \Gamma ~:~ \dist_{\mathbb{R}^d}(0,P_{\Gamma,\Z^d}(\gamma)) \le R\right\}.\]
       \item If   $C \subset \Gamma$ is a $(\{0\}\times G)$-separated set
       then every $c \in C$, the  $R$-truncated Voronoi cell $V_R(c,C) \Subset \Gamma$ is  convex.
   \end{enumerate}
\end{lemma}
\begin{proof}
\begin{enumerate}
    \item   Fix $C \subseteq \Gamma$  
    and $R>0$. 
    For every $c \in c$ the  disjointified  $R$-truncated Voronoi cell $\overline{V}_R(c,C)$ is a finite, because the number of elements in $\Gamma$ whose distance from $c$ is bounded by $R$ is finite.
    We need to prove that $\overline{V}_R(c_1,C) \cap \overline{V}_R(c_2,C)$ whenever $c_1,c_2 \in  C$ and $c_1 \ne c_2$. Choose any district elements $c_1,c_2 \in C$. Suppose that $v \in \overline{V}_R(c_1,C) \cap \overline{V}_R(c_1,C)$. In particular, $c_1,c_2 \in M(\gamma,C)$. 
    Since $c_1 \ne c_2$ it follows that $v- c_1 \ne v- c_2$ so either $v-c_1 \prec_{\Gamma} v- c_2$ or 
    $v-c_2 \prec_{\Gamma} v- c_1$.
    In the first case we have that $v \not \in  \overline{V}_R(c_2,C)$ and in the second case we have that $v \not \in  \overline{V}_R(c_1,C)$, either way we have a contradiction.
    \item If $v \in C+ B_R^\Gamma$ then there exists $c \in C$ such that $\dist_{\mathbb{R}^d}(P_{\Gamma,\Z^d}(v)) \le R$. It follows that $v \in \overline{V}_R(\tilde c,C)$,
    where $\tilde c \in M(v,C)$ is the element of $M(v,c)$ such that $(v-\tilde c)$ is maximal with respect to $\prec_{\Gamma}$.
    \item The fact that  every cell of the  disjointified $R$-truncated Voronoi cell $\overline{V}_R(c,C) \Subset \Gamma$ is  convex follows from the classical and well-known that every cell of a Voronoi diagram of a discrete subset of $\mathbb{R}^d$ is a convex polygon.
\end{enumerate}
  
\end{proof}

The following simple lemma is closely related to  Lemma 3.5 in \cite{MR1972240}, except that in \cite{MR1972240} the Voronoi diagram  of syndetic sets was considered, rather than the disjointified $R$-truncated Voronoi diagram of an arbitrary subset of $\Gamma$ (also in \cite{MR1972240} only the case $\Gamma = \Z^d$ was considered, though this it is a trivial modification to consider the case $\Gamma= \Z^d \times G$, with $G$ finite):

\begin{lemma}\label{lem:R_truncted_voronoi continuous}
    For any $R>0$ the function  $V_R:\{0,1\}^\Gamma \to \PTiling(\Gamma)$ that sends $V$ to the disjointified $R$-truncated Voronoi diagram of (the set corresponding to) $z$ is continuous and equivariant.
\end{lemma}
\begin{proof}
     By \Cref{lem:disjoint_Vor_PTiling} $\overline{V}_R(z) \in \PTiling(\Gamma)$ for every $z \in Z$. Equivariance follows directly from the fact that Euclidean distances are invariant under translations in $\mathbb{R}^d$.
\end{proof}

The following result is implicitly contained in \cite{MR1972240}, and also in  \cite{MR1951240}, where it was shown  that for any $\Z^d$ subshift $X$ we have $h(X) = \lim_{n \to \infty}\frac{1}{|C_n|}\log |\cL_{C_n}(X)|$  for any sequence $(C_n)_{n=1}^\infty$ of bounded convex subsets in $\Z^d$ such that for every $r >0$ there exist $n_0 \in \mathbb{N}$ so that $C_n$ contains a ball or radius $r$ for all $n \ge n_0$. A similar statement (with a similar proof) can also be found in \cite[Lemma 3.5]{MR3540453}. Again, we include a proof for completeness. 

\begin{lemma}\label{lem:conv_is_invariant}
For any $K \Subset \Gamma$ and any $\epsilon >0$ there exists  a finite set $F_0 \Subset \Gamma$ such that any convex set $F \Subset \Gamma$ with $F_0\subseteq F$ is $(K,\epsilon)$-invariant.
\end{lemma}

For a subset $A \Subset \mathbb{R}^d$ and $t \in (0,\infty)$ we use the notation $t A =\left\{ tv~:~ v \in A\right\}$.
Also, for $t>0$ let $B^d_{t} \subseteq \mathbb{R}$ denote the Euclidean ball of radius $t$ around $0$.
We  denote the Lebesgue measure of a measurable set $A \subseteq \mathbb{R}^d$ by $m_d(A)$.
For a set $\tilde F \subseteq \mathbb{R}^d$, $\partial \tilde F$ denotes the boundary of $\tilde F$  with respect to the usual topology on $\mathbb{R}^d$. Given $r_0 >0$ we denote
\[ \partial_{r_0} \tilde F  = \partial \tilde F + B^d_{r_0}.\]

The proof of \Cref{lem:conv_is_invariant} is based on the following lemma about convex sets in $\mathbb{R}^d$:

\begin{lemma}\label{lem:conv_Rd_bdry}
 For any convex compact set  $\tilde F \subseteq \mathbb{R}^d$ that contains $B^d_r$ , the Euclidean ball of radius $r$ around $0$ and any  $0< r_0 < r$, the following inclusion holds:
\begin{equation}
    \partial_{r_0} \tilde F  \subseteq \overline{\left(1+ \frac{r_0}{r}\right)\tilde F \setminus \left( 1- \frac{r_0}{r}\right)\tilde F}.
\end{equation}
Thus, 
\[
m_d( \partial_{r_0} \tilde F) \le \left( ( 1+\frac{r_0}{r})^d -( 1-\frac{r_0}{r})^d \right) m_d(\tilde F).
\]
\end{lemma}
\begin{proof}
By assumption,  $B^d_{r} \subseteq \tilde F$, so $B^d_{r_0} \subseteq \frac{r_0}{r} \tilde F$.
It follows that
\[\partial_{r_0} \tilde F = \partial \tilde F+ B^d_{r_0}  \subseteq \partial \tilde F + \frac{r_0}{r} \tilde F \subseteq \tilde F +  \frac{r_0}{r} \tilde F = \left( 1+ \frac{r_0}{r}\right) \tilde F.\]
In the last equality we use the fact that for a convex set $A \subseteq \mathbb{R}^d$ and $t_1,t_2 >0$ it holds that $t_1 A+ t_2A = (t_1+t_2)A$.

Also, by similar considerations for any $t < \left (1-\frac{r_0}{r}\right)$ we have that
\[ t\tilde F + B^d_{r_0} \subseteq t - \frac{r_0}{r} \tilde F \subseteq \tilde F \setminus \partial \tilde F.\]
It follows that $\partial_{r_0}\tilde F \subseteq \overline{\mathbb{R}^d \setminus \left( 1- \frac{r_0}{r}\right)\tilde F}$.
This proves that
\[
\partial_{r_0} \tilde F  \subseteq \overline{\left(1+ \frac{r_0}{r}\right)\tilde F \setminus \left( 1- \frac{r_0}{r}\right)\tilde F}.
\]
Thus (because the boundary of convex sets has zero Lebesgue measure),
\[
m_d\left(\partial_{r_0} \tilde F   \right) \le m_d\left(\left(1+ \frac{r_0}{r}\right)\tilde F \right)-  m_d\left(\left(1- \frac{r_0}{r}\right)\tilde F \right).
\]
From this it follows that
\[
m_d( \partial_{r_0} \tilde F) \le \left( ( 1+\frac{r_0}{r})^d -( 1-\frac{r_0}{r})^d \right) m_d(\tilde F).
\]
\end{proof}
For $k \in \mathbb{N}$ we denote
\[B_k = \{-k,\ldots,k\}^d \times G \Subset \Gamma.\]

\begin{proof}[Proof of \Cref{lem:conv_is_invariant}]

For any finite subset $K \Subset \Gamma$ we can find $r_0 >0$ such that 
$K \subseteq (B_{r_0}^d \cap \Z^d) \times G$.
For such $r_0 >0$, and  any convex set $\tilde F\subseteq \mathbb{R}^d$ it holds:
\[|\partial_K \left((\tilde F \cap \Z^d) \times G\right)| \le |G| \cdot m_d(\partial_{r_0} \tilde F).\]

Also, for any convex set $\tilde F \subseteq \mathbb{R}^d$:
\[\left| |G| \cdot m_d(\tilde F) - |(\tilde F \cap \Z^d) \times G| \right| \le m_d(\partial_{\sqrt{d}} \tilde F).\]

Let $F \subset \Gamma$ be a convex set that contains $B_j$. By our definition of a convex set in $\Gamma$, there exists a convex set $\tilde F \subset \Z^d$ so that $ F = (\tilde F \cap \Z^d) \times G$. Since $B_j \subset F$. If $j$ is sufficiently big it follows that $B_{r}^d \subseteq \tilde F$.

It thus suffices to prove that for any $r_0 >0$ and any $\delta >0$ there exists $r>0$ such that for any compact convex subset  $\tilde F \subseteq \mathbb{R}^d$ that contains the Euclidean ball of radius $r$ around $0$ the following inequality holds:
\[
m_d(\partial_{r_0} \tilde F) \le \delta \cdot m_d(\tilde F).
\]
Fix $r_0 >0 $ and $\delta >0$.
Since 
\[ \lim_{r \to \infty} ( 1+\frac{r_0}{r})^d -( 1-\frac{r_0}{r})^d  =0,\]
we can choose $r >0$ sufficiently big so that 
\[
( 1+\frac{r_0}{r})^d -( 1-\frac{r_0}{r})^d < \delta.
\]

Then by \Cref{lem:conv_Rd_bdry} it follows that $m_d(\partial_{r_0} \tilde F) \le \delta \cdot m_d(\tilde F)$ for any convex set $\tilde F$ that contains $B_{r}^d$.
\end{proof}

We now proceed to state and prove results needed to ensure that the cells of the disjointified $R$-truncated Voronoi diagram have ``sufficiently small external boundary'', provided that $R$ is sufficiently big.
\begin{lemma}\label{lem:convex_r_boundary_basic}
For any convex set  $A \subseteq \mathbb{R}^d$ and any $0 < r_0 <R$  such that $0 \in A$ the following inequality holds:
\begin{equation}\label{eq:ext_bdr_ineq}
    m_d\left( A \cap  B_{R+r_0} \setminus B_{R-{r_0}} \right) \le  \left( \frac{(R+r_0)^d-(R-r_0)^d}{(R-r_0)^d}\right)m_d(A). 
\end{equation}
\end{lemma}
\begin{proof}
Let for $t >0$ let $S^{d-1} \subset \mathbb{R}^d$ denote the unit sphere  in $\mathbb{R}^d$, namely, $S^{d-1} = \partial B_1^d$.
Let $A \subseteq \mathbb{R}^d$ be a convex set such that $0 \in A$ and $0 < r_0 <R$.
For $t >0$, consider the set $A_t \subseteq S^{d-1}$ given by:
\[
A_t = \left\{ v \in S^{d-1} ~:~ tv \in A\right\}.
\]

Given a Borel subset $B \subseteq S^{d-1}$ let $H_{d-1}(B)$ denote the $d-1$-dimensional Hausdorff measure of $B$.
Then there exists a constant $c>0$ (that depends on the normalization of $H_{d-1}$) so that:
\[m_d(A) = c\int_0^\infty t^{d-1} H_{d-1}(A_t) dt.\]
and so 
\[m_d\left( A \cap  B_{R+r_0} \setminus B_{R-{r_0}} \right) = c \int_{R-r_0}^{R+r_0} t^{d-1} H_{d-1}(A_t) dt.\]

Because $0 \in A$ and $A$ is convex, it follows that for every $0 <t_1 <t_2$ we have $A_{t_2} \subseteq A_{t_1}$. Thus,

\[m_d(A) \ge c \int_0^{R-r_0} t^{d-1} H_{d-1}(A_{R-r_0})dt = \frac{c  H_{d-1}(A_{R-r_0}) }{d} \cdot (R-r_0)^d.\]
and
\[m_d\left( A \cap  B_{R+r_0} \setminus B_{R-{r_0}} \right) \le  c \int_{R-r_0}^{R+r_0} t^{d-1} H_{d-1}(A_{R-r_0}) dt =\frac{c  H_{d-1}(A_{R-r_0}) }{d} \cdot \left((R+r_0)^d-(R-r_0)^d\right).\]
The proof of the lemma follows directly by combining these inequalities.
\end{proof}

\begin{lemma}\label{lem:Voronoi_k_epsilon_inv}
    For any $K \Subset \Gamma$ and $\epsilon >0$ there exists  a finite set $F \Subset \Gamma$ and $R_0 >0$   such that for any  $F$-separated set $C \subseteq \Gamma$  and any $R>R_0$, the disjointified $R$-truncated Voronoi diagram of $C$ is $(K,\epsilon)$-invariant.
\end{lemma}
\begin{proof}
Fix $K \Subset \Gamma$ and $\epsilon >0$. 
By \Cref{lem:conv_is_invariant} there exists $F_0 \Subset \Gamma$ so that any convex set that contains $F_0$ is $(K,\epsilon)$-invariant. Since the property of being  $(K,\epsilon)$-invariant remains unchanged under translations, it is clear that any  any convex set that contains $v+ F_0$ for some $v \in \Gamma$ is also $(K,\epsilon)$-invariant. Choose $R_1 > \max \{ \dist_{\mathbb{R}^d}(0,P_{\Gamma,\Z^d}(v)) ~:~ v \in F_0\}$ 
Let $F \Subset \Gamma$ be a finite set such that  $(\{0\} \times G )\subseteq F$, and also so that $F$ contains the set $B_{R_1}^\Gamma = \left\{ \gamma \in \Gamma~:~ \dist_{\mathbb{R}^d}(0,P_{\Gamma,\Z^d}(v)) \le R_1\right\}$.
Note that since $F \Subset \Gamma$ such that $(\{0\} \times G )\subseteq F$, it follows that any $F$-separated set $C \subseteq \Gamma$ is also $(\{0\}\times G)$-separated.
Choose $R_0 > 2R_1$.
By \Cref{lem:disjoint_Vor_PTiling}, every cell $\overline{V}_R(c,C)$ is convex, and $c+F \subseteq \overline{V}_R(c,C)$, so $\overline{V}_R(c,C)$ is $(K,\epsilon)$-invariant.

\end{proof}

\begin{lemma}\label{lem:Voronoi_extrior_bdry}
    There exists $F_0 \Subset \Gamma$ so that for any $K \Subset \Gamma$ and any $\epsilon >0$ there exists $R_0>0$  such that for any  $R>R_0$ and any $F_0$-separated $C \subseteq \Gamma$
    the disjointified $R$-truncated Voronoi diagram of $C$  has  $(K,\epsilon)$-small exterior boundary.
\end{lemma}
\begin{proof}
Using \Cref{lem:Voronoi_k_epsilon_inv} we choose $F_0 \Subset \Gamma$ and $R_1$ so that for any $F_0$-separated set $C \subseteq \Gamma$ and any $R> R_1$, the size of any cell in the disjointified $R$-truncated Voronoi diagram of $C$ is at least half the size of the (non-disjointified)  $R$-truncated Voronoi cell. 
Given $K \Subset \Gamma$, we can find $r_0 >0$ such that $K \subseteq (B_{r_0}^d \cap \Z^d) \times G$. 
Thus, if $F \Subset \Gamma$ is a cell of the $R$-truncated Voronoi diagram of $C$, 
there exists a compact convex set $\tilde F \subseteq \mathbb{R}^d$ 
such that $F= (\tilde F \cap \Z^d) \times G$.
The exterior $K$-boundary of $F$ is contained in \[\tilde F \cap (B_{R+r_0}^d \setminus B_{R-r_0}^d) \cap \Z^d)\times G.\]
As in the proof of \Cref{lem:conv_is_invariant}, it thus suffices to prove that for any $\delta >0$, provided that $R$ is sufficiently big, we have
\[
m_d\left( \tilde F \cap (B_{R+r_0}^d \setminus B_{R-r_0}^d)\right) \le \delta m_d( \tilde F).
\]
This follows from \Cref{lem:convex_r_boundary_basic}, provided that $R$ is big enough so that
\[
\frac{(R+r_0)^d-(R-r_0)^d}{(R-r_0)^d} < \delta.
\]
\end{proof}


\begin{proof}[Proof of \Cref{prop:good_partial_tilings}:]
Let $\epsilon >0$ and $K \Subset \Gamma$ be given. 
By \Cref{lem:Voronoi_k_epsilon_inv} there exists  $F \Subset \Gamma$ and $R_0 >0$   such that for any
 $F$-separated set $C \subseteq \Gamma$  and any $R>R_0$,
the disjointified $R$-truncated Voronoi diagram of $C$ is $(K,\epsilon)$-invariant. By further enlarging $F$, we can assume that $F_0 \subseteq F$, where $F_0 \Subset \Gamma$ is as in \Cref{lem:Voronoi_extrior_bdry}, and that $F$ is symmetric and so that $B_{R_1} \subseteq F$, where $R_1= 2\max\{ \dist_{\mathbb{R^d}}(0,v)~:~ v\in K\}$ and $B_{R_1}= \{ v \in \Gamma~:~ \dist_{\mathbb{R}^d}(0,P_{\Gamma,\Z^d}(v)) \le R_1\}$.

Let $P = (F+F) \setminus \{0\}$. Now let $X$ be a $\Gamma$-subshift, $W \Subset \Gamma$
    and let $\mathcal{F} \subseteq \cL_W(X)$  be a set of patterns so that $x_W \not \in \mathcal{F}$ for any $x \in X$ such that $\stab(x) \cap P = \emptyset$. 
    By \Cref{lem:krieger_marker_lemma}, there exists a clopen set $C \subset X$ such that $C \cap \sigma_v(C) = \emptyset$ for every $v \in P$ and so that for every $x \in X$ either there exists $v \in P \cup\{0\} = B_k$ such that $\sigma_v(x) \in C$ or $x_W \in \mathcal{F}$.
   
    Define  $\alpha:X \to \{0,1\}^\Gamma$ by
    \[\alpha(x)_v = \begin{cases}
    1 & \sigma_v(x) \in C\\
    0 & \mbox{ otherwise}. 
    \end{cases}
    \]
    It is clear that $\alpha:X \to \{0,1\}^\Gamma$ is equivariant. Continuity of $\alpha$ follows from the fact that $C$ is a clopen set.
    
    Because $C \cap \sigma_v(C) = \emptyset$ for every $v \in P$ it follows that for every $x \in X$ $\alpha(x)$ corresponds to an $F$-separated.

    Using \Cref{lem:Voronoi_extrior_bdry} we can find $R>\max\{R_0,R_1\}$ so that for any $F$-separated $C \subseteq \Gamma$ the disjointified $R$-truncated diagram of $C$ has $( W,\epsilon_1)$-small exterior boundary.

    Let $\tau:\{0,1\}^\Gamma \to \PTiling(\Gamma)$ be the map given by $\tau(z)= \overline{V}_R(z)$  be the map that assigns to every $z \in \{0,1\}^\Gamma$ the disjointified $R$-truncated diagram of $\alpha(x) \in\{0,1\}^\Gamma$ (interpreted as a subset of $\Gamma$). By \Cref{lem:R_truncted_voronoi continuous} $\tau:\{0,1\}^\Gamma \to \PTiling(\Gamma)$ is continuous and equivariant. For every $x \in X$, since every $\alpha(x)$ corresponds to an $F$-separated set, by the choice of $F$ it follows that $\tau(\alpha(x))$ is a pointed $(K,\epsilon)$-invariant tiling. 
    If $v \not \in \bigcup \tau(\alpha(x))$ then $\dist_{\mathbb{R}^d}(v,w) \ge R$ for every $w \in \Gamma$ such that $\alpha(x)_w=1$. This means that $\sigma_{v}(x) \not \in \bigcup_{u \in P \cup \{0\}}\sigma_u(C)$, so $\sigma_{-v}(x)_W \in \mathcal{F}$. This proves  that property $1$ is satisfied.
    Because $B_{R_1} \subseteq F$ and $\alpha(x)$ is $F$-separated, it follows that whenever $\alpha(x)_v=1$ then the unique $T \in \tau(\alpha(x))$ corresponding to the disjointified $R$-truncated Voronoi cell of $v$ contains $v+K$. This  proves  that property $2$ is satisfied. By the choice of $R$ using \Cref{lem:Voronoi_extrior_bdry},  the disjointified $R$-truncated diagram of the set corresponding to $\alpha(x)$ has $( W,\epsilon_1)$-small exterior boundary for every $x \in X$. This proves  that property $3$ is satisfied.
\end{proof}

\section{Good marker patterns  via the map extension property}\label{sec:marker_patttern}

One of the key ideas in Krieger's proof of his embedding theorem is to designate a special pattern in the range of the embedding to mark the occurrences of some clopen set in the domain. These ``marker patterns'' need to satisfy several condition. In particular, they need to be ``locally identifyable'', which in particular requires them to have no self overlaps. Also, the subshift obtained from $Y$ by forbidding this marker pattern needs to have enough entropy``big enough to embed $X$''. 
The purpose of this section is to formulate and prove a lemma that guarantees suitable marker patterns in the domain. The slightly involved formulation is given in  \Cref{lem:marker_pattern_RAR}.

We first state a basic result about the existence of patterns without self-overlaps.

\begin{definition}\label{def:self_overlap}
Given $F \Subset \Gamma$, $v \in \Gamma$ and  a pattern $w \in A^F$, we say that $w$  \emph{has a self-overlap $v$} if 
\[ w_{F \cap (v+F)} = \sigma_v(w)_{F \cap (v+F)}.\]
We say that $w^{(1)},w^{(2)} \in  A^F$ have an overlap at $v\in \Gamma$ if 
\[ w^{(1)}_{F \cap (v+F)} = \sigma_v(w^{(2)})_{F \cap (v+F)}.\]
\end{definition}

\begin{lemma}\label{lem:marker_pattern_basic}
    Let $G$ be a finite abelian group, $ d\ge 1$ and $\Gamma= \Z^d \times G$, 
    and for $k \in \mathbb{N}$ let 
    \begin{equation}
    B_k = \{-k,\ldots,\ldots,k\}^d \times G \Subset  \Gamma.
    \end{equation}
    and
    \[
    Q_k = B_k \setminus B_{\lfloor k/10 \rfloor}.
    \]
    and let $Y$ be a non-trivial strongly irreducible  $\Gamma$-subshift of finite type.
    Then there exists $n_0$ such that for every $n \ge n_0$ there exist $w \in \cL_{Q_n}(Y)$ that has no self-overlaps in  $B_n \setminus \Ker(Y)$.
\end{lemma}
Similar results have appeared in the literature. For instance in \cite{MR1307966} patterns without self-overlaps have been constructed for strongly irreducible $\Z^2$-subshifts. Also see \cite{MR2643712} where the existence of non-overlapping pattern is established using a non-constructive argument using only positive positive entropy.
We briefly sketch a proof below. 
\begin{proof}
    Let $Y$ be a non-trivial strongly irreducible  $\Gamma$-subshift of finite type. In particular, $Y$ is topologically mixing and non-trivial so $\Ker(Y)$ is a finite subgroup of $\Gamma$. Since $Y$ is strongly irreducible subshift of finite type there exists $N \in \mathbb{N}$ so that for any $m \in \mathbb{N}$,  any set $F \Subset \Gamma$ and any $B_{n+N}$-separated set $K$ and any $f:K \to \cL_{B_m}(Y)$ there exists $x \in X$ with $\sigma_v(x)_{B_m} = f(v)$ for all $v \in K$. Choose $m \in \mathbb{N}$ sufficiently big,     
    and then choose $n \in \mathbb{N}$ much bigger. Since pairwise dijoint translates of $B_m$ tile the group $\Gamma$, there is a $B_{n+N}$ separated subset $K_{n,m}$ of $Q_n$ of size at least $\frac{|Q_n|- |\partial_{B_m}Q_n|}{|B_m|}$. Choose $f \in \cL_{B_m}(Y)^{K_{n,m}}$ uniformly at random. In other words, for every $v \in K_{n,m}$ choose $f(v) \in \cL_{B_m}(Y)$ uniformly at random, so that the random variables $(f(v))_{v \in K_{n,m}}$ are jointly independent. Let $w \in \cL_{B_m}(Y)$ be such that $\sigma_v(w)_{B_m}= f(v)$ for every $v \in K_{n,m}$. If $m$ is sufficiently big, there exists $p <1$ such that the probability that a uniformly chosen $w \in \cL_{B_m}(Y)$ has a self-overlap in  $B_{2N} \setminus \Ker(Y)$ is at most $p$. Also  if $m$ is sufficiently big, the probability that independently chosen $w^{(1)},w^{(2)} \in \cL_{B_m}(Y)$ have an overlap in $B_{\lfloor m/w \rfloor -N}$ is at most $p$. It follows that for any $v \in B_n \setminus \Ker(Y)$ the probability that the randomly chosen $w$ has a self-overlap $v$ is at most $p^{\frac{|Q_{n-m} \cap (v+Q_{n-m})|}{|B_m|}}$.
    Since $Q_{n-m} \cap (v+Q_{n-m})$ contains a translate of $B_{\lfloor n/3 \rfloor}$ for every $v \in B_m$, it follows that the probability that $w$ has an overlap at $v \in B_m \setminus \Ker(Y)$ is at most $p^{(n-m)/3m}$. By a union-bound, the probability that $w$ has a self-overlap in $B_n \setminus \Ker(Y)$ is bounded from above by $(2n+1)^d p^{(n-m)/3m}$, which tends to $0$ as $n\to \infty$, for fixed $m$. 
\end{proof}

\begin{lemma}\label{lem:marker_pattern_RAR}
    Let $\Gamma$ be an infinite finitely generated abelian group, let $Y$ be a non-trivial $\Gamma$-subshift with the map extension property, and let $\hat Y \subset Y$ be a (possibly empty) $\Gamma$-subshift strictly contained in $Y$. For any given $\epsilon >0$, and any finite set $F_0 \Subset \Gamma$ there exist:
    \begin{itemize}
        \item A subshift $S \subset Y$ with $\hat Y \subseteq S$.
        \item A retraction map $\tilde r:Y \to S$.
        \item A map $\kappa:S \times \{0,1\}^\Gamma \to Y$.
        \item A finite symmetric set $F \Subset \Gamma$ such that $F_0 \Subset F$.
        \item Patterns $w^{(0)},w^{(1)} \in \cL_{F}(Y)$ such that $w^{(0)} \not \in \cL_{F}(S)$.
    \end{itemize}
    So that the following hold:
    \begin{enumerate}
        \item[(1)] For any $y \in Y$ if $\tilde r(y)_0 \ne y_0$ then there exists $w \in F$ such that $\sigma_w(y)_F = w^{(0)}$. In particular, if $y \in Y$ and $y_{F+F} \in \cL_{F+F}(S)$ then $\tilde r(y)_0=y_0$.
        \item[(2)] For every $y \in S$ and $z \in \{0,1\}^\Gamma$ if $\tilde y = \kappa(y,z)$ then $\tilde y_{F}= w^{(0)}$ if and only if $y_F= w^{(1)}$ and $z_0 =1$. 
        \item[(3)] For every $y \in S$ and $z \in \{0,1\}^\Gamma$ if $\tilde y = \kappa(y,z)$ and $z_v \ne 1$ for all $v \in F$ then $\tilde y_0 = y_0$
          \item[(4)] For any $\Gamma_0 \le \Gamma$, if $Y_{\Gamma_0} \ne \emptyset$
    then $S_{\Gamma_0} \ne \emptyset$. 
        \item[(5)]  For any $\Gamma_0 \le \Gamma$ it holds that
 $\Ker( Y_{\Gamma_0})= \Ker(S_{\Gamma_0})$ and 
    and $h(\overline{S}_{\Gamma_0}) \ge h(\overline{Y}_{\Gamma_0})-\epsilon$.
    \end{enumerate}
\end{lemma}

\begin{proof}
    By the structure of finitely generated abelian groups we can assume that $\Gamma = \Z^d \times G$, where $G$ is a finite abelian group and $d \ge 1$. 
    Let $Y$ be a non-trivial $\Gamma$-subshift with the map extension property, and let $\hat Y \subset Y$ be a (possibly empty) closed, $\Gamma$-invariant set such that $\hat Y \ne Y$. 
    Note that since $Y$ is a non-trivial subshift with the map extension property, it is topologically mixing and in particular $\Ker((Y)$ is finite hence contained in $\{0\} \times G \le \Gamma$. Replacing $\Gamma$ with $\Gamma/ \Ker(Y)$, we can conveniently assume that $\Ker(Y)= \{0\}$
    Fix any $\epsilon >0$ and any finite set $F_0 \Subset \Gamma$.
     For any $n \in \mathbb{N}$ let $B_n, Q_n \Subset \Gamma$ be defined as in \Cref{lem:marker_pattern_basic}, and let 
       \[ W_{n,k} =\left\{ w \in \cL_{B_n}(Y)~:~ w_{Q_n} = w^{(1)}_{Q_n} \right\} .\]
         Since $Y$ has the map extension property, by \Cref{prop:abs_retract_prop},  $Y$ is strongly irreducible. By \Cref{lem:marker_pattern_basic}, for  sufficiently big $n \in \mathbb{N}$, we can find a pattern $w^{(1)} \in \cL_{B_n}(Y)$ such that $w^{(1)}_{Q_n}$ has no self-overlaps in $B_n$.
     Because $Y$ is strongly irreducible, we can choose  $n \in \mathbb{N}$ sufficiently big and $\frac{1}{20} n < k < \frac{1}{11} n$ so that the following are also satisfied:
     \begin{itemize}
         \item For every $\hat w \in \cL_{Q_n}(Y)$ and $\tilde w \in \cL_{B_k}(Y)$, there exists $w \in \cL_{B_n}(Y)$ so that $w_{Q_n} = \hat w$ and $w_{B_k}= \tilde w$. 
         \item   $F_0 \subseteq B_{\lfloor n/11 \rfloor}$.
         \item  $|\cL_{B_k}(Y)\setminus \cL_{B_k}(\hat Y)| \ge  |B_n|+2$.
         \item   \[
     \log \left(\frac{ |W_{n,k}| }{|W_{n,k}|-1} \right) \le \epsilon,
    \]
    \end{itemize}

       The properties above grantees that there exists two distinct words $\tilde w^{(0)},\tilde w^{(2)} \in \cL_{B_{k}}(Y)\setminus \cL_{B_{k}}(\tilde Y)$ that do not occur in $w^{(1)}$, and also (under our assumption that $\Ker(Y)= \{0\}$) and so that
     $\sigma_v(\tilde w^{(i)}) \ne \tilde w^{(i)}$ for every $v \in \{0\} \times G$ and $i \in \{0,2\}$.
    Also, 
    there exists $w^{(0)},w^{(2)} \in \cL_{B_n}(Y)$ such that $w^{(0)}_{Q_n} = w^{(1)}_{Q_n}= w^{(2)}_{Q_n}$ and $w^{(i)}_{B_k} = \tilde w^{(i)}$ for $i \in \{0,2\}$.

    Let $F= B_n$, and 
    let 
    
    \[ S = Y^{ w^{(0)}} = \left\{ y \in Y~:~ \sigma_v(y)_{B_n} \ne w^{(0)} \mbox{ for all } v \in \Gamma \right\}.\]

     Since $\tilde w \not \in \cL_{B_k}(\hat Y)$, it follows that $w^{(0)} \not \in \cL_{B_n}(\hat Y)$, so $\hat Y \subseteq S$.

     Let $\tilde r:Y \to S$  be the map  that replaces every occurrence of $w^{(0)}$ with an occurrence of $w^{(1)}$. The map $r$ is well defined because  the occurrences of $w^{(0)}$ in any $y \in Y$ are $B_n$-separated and also $B_{n-k}$-disjoint from the occurrences of $w^{(1)}$. The map $\tilde r$ is a retraction map from $Y$ to $S$. It is clear that condition $(1)$ in the statement of the lemma holds.

     Let $\kappa: S \times \{0,1\}^\Gamma \to Y$ be the map defined as follows: For $y \in S$ and $z \in \{0,1\}^\Gamma$ the point $\tilde y \in Y$ is obtained by replacing the  occurrences of $w^{(1)}$ in the set $\{ v \in \Gamma~:~ z_v=1\}$ by occurrences of $w^{(0)}$. Because the occurrences of $w^{(1)}$ are $B_n$-separated, this is well defined. From the definition it follows that conditions $(2)$ and $(3)$ in the statement of the lemma are satisfied.

We will now show that for any $y \in Y$ there exists $\hat y \in S$ with $\stab(y) = \stab(\hat y)$. 
Take any $y \in Y$.
If $w^{(0)}$ does not occur in $y$, then $y \in S$ and we can take $\hat y= y$.
Otherwise, 
there exists $v \in \Gamma$ such that $w^{(0)}$ occurs in $y$ at $v+w$ for every $w \in \stab(y)$.
Replacing $y$ by $\sigma_{-v}(y)$,  we can assume that  $w^{(0)}$ occurs in $y$ at  every  $w \in \stab(y)$.
Let $\hat y \in Y$ be the point obtained from $y$ by replacing each occurrence of $w^{(0)}$ in $\stab(y)$ by an occurrence of $w^{(2)}$, and replacing every other occurrence of $w^{(0)}$ or $w^{(2)}$ in $y$ by an occurrence of $w^{(1)}$. It is clear that $\stab(y) \le \stab(\hat y)$. On the other hand, since $w^{(2)}$ only occurs in $y$ at $\stab(y)$, it follows that $\stab(\hat y) \le \stab(y)$. 

From this it follows that if $Y_{\Gamma_0} \ne \emptyset$ then $S_{\Gamma_0} \ne \emptyset$, 
and also that for every $\Gamma_0 \le \Gamma$ we have $\Ker(S_{\Gamma_0})=\Ker(Y_{\Gamma_0})$. In particular, condition $(4)$ in the statement of the lemma is satisfied.

Let $\Gamma_0 \le \Gamma$ be a subgroup which is not $B_n$-separated, then $w^{(0)}$ does not occur in $Y_{\Gamma_0}$, and so
$S_{\Gamma_0}= Y_{\Gamma_0}$, and in particular if $Y_{\Gamma_0} \ne \emptyset$ then $S_{\Gamma_0} \ne \emptyset$, and in that case $\Ker(S_{\Gamma_0})=\Ker(Y_{\Gamma_0})$ and $h(S_{\Gamma_0})=h(Y_{\Gamma_0})$.
Take any $B_n$-separated  subgroup $\Gamma_0 \le \Gamma$.

and let $D_{\Gamma_0}$ be a fundamental domain for $\Gamma_0$.

    Then

    \[
    \frac{|\cL_{B_m}(Y_{\Gamma_0})|}{|\cL_{B_n}(S_{\Gamma_0})|} \le \left(\frac{ |W_{n,k}| }{|W_{n,k}|-1}\right)^{|B_m \cap D_{\Gamma_0}|}
    \]

    It follows that 
    \[
    \frac{1}{|B_m \cap D_{\Gamma}|}\log\left| \cL_{B_m}(Y_{\Gamma_0})\right| -\frac{1}{|B_m \cap D_{\Gamma}|}\log\left| \cL_{B_m}(S_{\Gamma_0}) \right| \le \log \left(\frac{ |W_{n,k}| }{|W_{n,k}|-1} \right).
    \]

    So taking $m \to \infty$ we have that
    \[ h(\overline{Y}_{\Gamma_0}) \le h(\overline{S}_{\Gamma_0}) + \epsilon.\]

This shows that condition $(5)$ is satisfied and completes the proof of the lemma.
\end{proof}
\section{Proof of the embedding theorem for subshifts over a finitely generated abelian group}
\label{sec:emb_fg}

In this section we complete the proof of the more difficult direction of our main theorem for the case that $\Gamma$ is a finitely generated abelian group. Namely, we prove that in the case that $\Gamma$ is a finitely generated abelain group, the conditions $E(X,Y,\hat \rho,\Gamma_0)$ from \Cref{def:E_X_Y} are sufficient to assure that  $X_{\tilde \Gamma} \hookrightarrow_{\hat \rho} Y_{\tilde \Gamma}$.
We phrase a slightly more involved statement for the purpose of the internal inductive nature of our proof:

\begin{proposition}\label{prop:emb_inductive}
     Let $\Gamma$ be a finitely generated  abelian group and  let $X,Y$ be $\Gamma$-subshifts. Suppose that  $Y$ has the map extension property.
      Let $Z \subseteq X$ be a  closed $\Gamma$-invariant set,  
      and let $\hat \rho:Z \to Y$ be an embedding.  
 Suppose that  $E(X,Y,\hat \rho,\Gamma_0)$
 holds   for any subgroup $\Gamma_0 \le \Gamma$. Then $X_{\tilde \Gamma} \hookrightarrow_{\hat \rho} Y_{\tilde \Gamma}$.
\end{proposition}

\begin{proof}
    We will prove the statement by a double induction: The external induction on the rank of the group  $\Gamma$, and the internal induction over the  size of the torsion subgroup of $\Gamma$.

    Let $X \subseteq B^\Gamma$, $Y \subseteq A^\Gamma$, $Z\subseteq X$ and $\hat \rho:Z \to Y$  be as in the statement of the lemma. 
    
    By \Cref{prop:map_ext_iff_RAR} there exists a finite set $\mathcal{G} \Subset \Sub(\Gamma)$ such that $Y$ is an absolute retract for the class of $\mathcal{G}$-free subshifts. In particular, by \Cref{lem:witness_G_free} there exists a finite set $W_0 \Subset \Gamma$ that witnesses that $Y$ is $\mathcal{G}$-free.
By part $(ii)$ of \Cref{lem:map_ext_G_free} , there exists a finite set $F \Subset \Gamma$
so that for any $\mathcal{G}$-free subshift $\hat Y$ there exists a retraction map 
$\hat r:\hat Y \to Y$ such that if $y \in \hat Y$
and $y_F \in \cL_F(Y)$ then $\hat r(y)_0=y_0$. 

Since $E(X,Y,\hat \rho,\{0\})$ holds, either $h(X) < h(Y)$ or $X \cong_{\hat \rho} Y$. In the second case, there is nothing to prove. So we can assume that $h(X) < h(Y)$. 
Let $\eta  = \frac{1}{3}(h(Y)- h(X))$.
Choose $n \in \mathbb{N}$ sufficiently big so that $F \subseteq B_n$, and also so that  for any pattern $w \in \cL_{B_{5n}(Y)}$ and any $y \in Y$ there exists $\tilde y \in Y$ so that $\tilde y_{B_{5n}}=w$ and $\tilde y_{\Gamma \setminus B_{6n}}=y_{\Gamma \setminus B_{6n}}$.

 By \Cref{prop:locally_correctable_entropy_subgroups} there exists a finite set $F_1 \Subset \Gamma \setminus \{0\}$ such that for any $\Gamma_0 < \Gamma$ with $\Gamma_0 \cap F_1 \ne \emptyset$ we have that $h_{\Gamma_0}(Y) > h(Y) - \eta$. By \Cref{prop:subgroup_entropies_upperbound} there exists a finite set $F_2 \Subset \Gamma \setminus \{0\}$ such that for any $\Gamma_0 < \Gamma$ with $\Gamma_0 \cap F_2 \ne \emptyset$ we have that $h_{\Gamma_0}(X) < h(X) + \eta$. Let $F_0= F_1 \cup F_2$. It follows that $h_{\Gamma_0}(X) < h_{\Gamma_0}(Y) -\eta$ whenever $\Gamma_0 < \Gamma$ satisfies that $F \cap \Gamma_0 \ne \emptyset$. 
 
 Let $\mathcal{G}_0$ denote the family of subgroups non-trivial subgroups of $\Gamma$ generated by a subset of the elements of $F_0$.
Let $\hat Y= \hat \rho(Z) \cup \bigcup_{\Gamma_0 \in \mathcal{G}_0} Y_{\Gamma_0}$.

Applying  \Cref{lem:marker_pattern_RAR} with $X,Y, \hat Y$ $F_0$ as above and $\epsilon = \eta$,
we can find a finite symmetric set $F \subset \Gamma$, subshift $S \subset Y$, 
a retraction map $\tilde r:Y \to S$, a map $\kappa:S \times \{0,1\}^\Gamma \to Y$ and patterns
$w^{(0)},w^{(1)} \in \cL_F(Y)$ such that $w^{(0)} \not \in \cL_F(S)$ so that conditions $(1)-(5)$ in the statement of \Cref{lem:marker_pattern_RAR} hold.

Observe that the subshift $S$ also has the map extension property, since it is a retract of $Y$. 
Furthermore, $h(S) > h(X)$ and for any $\Gamma_0 \le \Gamma$ the condition $E(X,S,\hat \rho, \Gamma_0)$ holds.
Since  $h(X) < h(S)$, 
by \Cref{lem:top_entropy_inv_set} we can find  $\epsilon >0$ and $K \Subset \Gamma$ so that for any $(K,\epsilon)$-invariant set $T \Subset \Gamma$  we have 
\begin{equation}\label{eq:patterns_size_X_Y}
\log \left|\cL_{T + F+F}(X)\right| \le  (1-\epsilon) \log \left|\cL_{T \setminus \left( \partial_{(F+F)}T \cup (F+F)\right)}(S)\right|.
\end{equation}

Applying  \Cref{prop:good_partial_tilings} with $X$, $K \subset \Gamma$ and $\epsilon >0$  we obtain a suitable set $P \Subset \Gamma \setminus \{0\}$.

Let $\mathcal{G}_1 \Subset \Sub(\Gamma)$ denote the union of $\mathcal{G}_0$ with the family of subgroups non-trivial subgroups of $\Gamma$ generated by a subset of the elements of $P$. 

Order the elements of  $\mathcal{G}_1$ as $\mathcal{G}_1 = \left\{\Gamma_1,\ldots,\Gamma_m\right\}$. We will prove that for every $1\le j \le m$ there exists an embedding $\rho_j: (Z \cup \bigcup_{\ell=1}^j X_{\Gamma_\ell}) \to \tilde Y$ that extends $\tilde \rho$.
The construction proceeds in $m$ steps, where in step $j$ we extend $\rho_j$ to $\rho_{j+1}$ using the main induction hypothesis with $\Gamma$ replaced by $\Gamma/\Gamma_{j+1}$ (noting that $\Gamma/\Gamma_{j+1}$ either has lower rank than $\Gamma$ or smaller torsion).

So now we can conveniently replace  $Z$ with $Z \cup \bigcup_{j=1}^m X_{\Gamma_j}$ and replacing $\hat \rho$ with $\rho_m$, and thus assume that $X_{\Gamma_j} \subseteq Z$ for every $1\le j \le m$.

Let $W_1 \Subset \Gamma$
be a coding window for $\hat \rho$, and  let  $\hat \Phi:B^{W_1} \to A$ be a sliding block code for $\hat \rho$.  By possibly enlarging  $W_1 \Subset \Gamma$, we can also assume that $F\subseteq W_1$ (and in particular that $0 \in W_1$), and also that $W_1$ is symmetric. Since $Z$ is relatively of finite type on $X$ We can further enlarge $W_1$ so that if $x \in X$ then $x \in Z$ of and only if  and $\sigma_v(x)_W  \in \mathcal{F}$ for all $v \in \Gamma$. 
In particular, if $x \in X$ and $x_W \not \in \mathcal{F}$  
then $x \not \in Z$, and in particular $\stab(x) \cap P = \emptyset$.
Let $\mathcal{F} = \cL_{W_1}(Z)$.
Then indeed  if $x \in X$ and $x_W \not\in \mathcal{F}$ 
then $\stab(x) \cap P = \emptyset$.
Let $W_2$ be a symmetric  injectivity window for $\hat \rho$ (with respect to the sliding block $\hat \Phi$), with $0 \in W_2$.
 Let $\tilde W = W_1 +W_2$. Then $\hat \rho$ extends to an embedding of the subshift of finite type 
 $\tilde Z$
 given by 
 \[
 \tilde Z = \left\{ z \in B^\Gamma ~:~ \sigma_v(z)_{\tilde W} \in \cL_{\tilde W}(Z) \mbox{ for all } v \in \Gamma
 \right\}.
 \]
Since $Z \subseteq \tilde Z$, we can conveniently replace $Z$ by $\tilde Z$.

 Given a map $\alpha:X \to \{0,1\}^\Gamma$, a pointed partial tiling $\tau:\alpha(X) \to \PTiling(\Gamma)$, $x \in X$ and $v \in \Gamma$ such that $\alpha(x)_v=1$, let $T_v(x) \Subset \Gamma$ denote the tile in $\tau(\alpha(x))$ that contains $v$. and let
\[T^{-}_v(x) = T_v(x) \setminus \left((v+F+F) \cup \partial_{F}T_v(x) \right)\]
and
\[ T^{+}_v(x) = (T_v(x)+F) \cup \partial_{W}^{\mathcal{T}} T_v(x) .\]



Applying \Cref{prop:good_partial_tilings} with the previously obtained set $P \Subset \Gamma \setminus \{0\}$, $W= W_1$ and $\tilde W =  W_1 +W_2$ and $\epsilon_1>0$ sufficiently small, 
we obtain the existence of  a map $\alpha:X \to \{0,1\}$
    and a pointed $(K,\epsilon)$-invariant tiling $\tau: \alpha(X) \to \PTiling(\Gamma)$ 
     such that  for every $x \in X$:
\begin{enumerate}
    \item For every $v \in \Gamma$ such that $\alpha(x)_v=1$, letting $T^- = T^-_v(x)$ and $T^+=T^+_v(x)$ it holds that 
    $\left| \cL_{T^+}(X)\right| \le \left|\cL_{T^-}(Y) \right|$.
    \item If $v \in \Gamma$ is not contained in some tile $T \in \tau(\alpha(x))$ then $\sigma_v(x)_{W_1+W_2} \in \cL(Z)$.
\end{enumerate}

Given $x \in X$,  
let \[P_v^+(x) = x_{T_v^+(x)} \in \cL_{T_v^+(x)}(X).\]
For every $T,T^-,T^+ \Subset \Gamma$ that occur as $T_0(x),T_0^-(x),T_0^+(x)$ for some $x \in X$ 
with $\alpha(x)_0 =1$, choose an injective function 
\[\Phi_{T^-,T^+}:\cL_{T^+}(X) \to \cL_{T^-}(\tilde Y).\]
Now, choose a function
\[\Phi^+_{T,T^+}:\cL_{T^+}(X) \to \cL_{T}(\tilde Y)\]
so that $\Phi^+_{T,T^+}(w)_{T^-} = \Phi_{T^-,T^+}(w)$ for every $w \in \cL_{T^+}(X)$, and so that $\Phi^+_{T,T^+}(w)_{F}=w^{(1)}$. The existence of an injective function form $\cL_{T^+}(X)$ to $\cL_{T^-}(Y)$ for every pair of sets $T^-,T^+ \Subset \Gamma$ that occur as $T_0^-(x),T_-^+(x)$ for some $x \in X$  and $v \in \Gamma$ such that $\alpha(x)_v=1$ follows from the inequality 
$\left| \cL_{T^+}(X)\right| \le \left|\cL_{T^-}(Y) \right|$.

We now define a map $\tilde \rho:X \to A^\Gamma$.
For $x \in X$ let $\tilde \rho(x) \in A^\Gamma$ be the unique point $y \in A^\Gamma$ satisfying:
\begin{itemize}
    \item For every $v \in\Gamma$ such that $\alpha(x)_v=1$ ,  $y_{T_v(x)}:= \sigma_{-v}\left( \Phi^+_{T_v,T_v^+}(P_0^+(\sigma_{v}(x))) \right)$.
    \item For every 
    $v \in \Gamma \setminus \bigcup \tau(\alpha(x))$, 
    $y_v:= \tilde \Phi (\sigma_{v}(x)_{W})$.
\end{itemize}

 Let $\tilde Y= \tilde \rho(X)$. 
We will now check that $\tilde Y$ is $\mathcal{G}$-free. It suffices to show that for any $x \in X$ the group $\Gamma_x = \stab(\tilde \rho(x))$ does not contain any subgroup from $\mathcal{G}$. If  $x\not \in  Z$, then there exists $v \in \Gamma$ such that $\sigma_v(x)_{W_1} \not \in \mathcal{F}$ and so there exists $v' \in \Gamma$ such that $\alpha(x)_{v'}=1$. Then $\sigma_v(\tilde \rho(x))_{B_{5n}} = w^{(1)}$, and this witnesses that $\stab(\tilde \rho(x))$ does not contain a subgroup from $\mathcal{G}$. If $x \in Z$ then $\tilde\rho(x)= \hat \rho(x)$. Since  the map $\hat \rho$  is injective , it follows that in this case $\stab(\tilde  \rho(x)) = \stab(x)$. 
Because $X$ is $\mathcal{G}$-free, it follows that also in this case $\stab(\tilde \rho(x))$ does not contain any subgroup from $\mathcal{G}$. We conclude that indeed the subshift $\tilde Y$ is $\mathcal{G}$-free.

It follows that there exists a retraction map $\hat r:\tilde Y \to Y$ so that $\tilde r(y)_v= y_v$ whenever $\sigma_v(y)_{B_{n}} \in \cL_{B_n}(Y)$.
Let $\check \rho:X \to  S$ be defined by $\check \rho = \tilde r \circ \hat r \circ \hat \rho$, and let $\rho:X \to Y$ be defined by:
\[
\rho(x)=\kappa(\check \rho(x),\alpha(x)).
\]
To finish the proof it remains to show that the map $\rho$ is injective.
If $x \in X$, $v \in \Gamma$, $\check y = \check \rho (x)$ then $\sigma_v(\check \rho(y))_{F} = w^{(1)}$ whenever $\alpha(x)_v=1$.
Thus, if $y= \rho(x)$ we have that
$\sigma_v(y)_{F}=w^{(0)}$ if and only if $\alpha(x)_v=1$. 
Now suppose that $x^{(1)},x^{(2)} \in X$ are two distinct points. Let $y^{(1)}=\rho(x^{(1)})$ and $y^{(2)}=\rho(x^{(2)})$. If $\alpha(x^{(1)}) \ne \alpha(x^{(2)})$ then there exists $v \in \Gamma$ such that only one of the patterns $\sigma_v(y^{(1)})_{F},\sigma_v(y^{(2)})_{F}$ is equal to $w^{(0)}$, and in particular in this case $y^{(1)} \ne y^{(2)}$. So consider the case that $\alpha(x^{(1)})=\alpha(x^{(2)})$. 
In this case $\tau(\alpha(x^{(1)}))=\tau(\alpha(x^{(2)}))$. Let $\mathcal{T}= \tau(\alpha(x^{(1)}))$.
By the assumption that $x^{(1)} \ne x^{(2)}$, there exists $v \in V$ such that 
$x^{(1)}_v \ne x^{(2)}_v$. If there exists $T \in \mathcal{T}$ such that $v \in T^+$, then $x^{(1)}_{T^+} \ne x^{(2)}_{T^+}$, and by injectivity of $\Phi_{T,T^+}$, we have that $\hat \rho (x^{(1)})_{T^-} \ne \hat \rho (x^{(2)})_{T^-}$. Since $\hat \rho (x^{(1)})_T, \hat \rho (x^{(1)})_T \in \cL_T(\tilde Y)$, it follows that $\rho(x^{(1)})_{T^-}= \hat \rho(x^{(1)})_{T^-}$ and $\rho(x^{(2)})_{T^-} =\hat \rho(x^{(2)})_{T^-}$ so
$\rho (x^{(1)})_{T^-} \ne  \rho (x^{(2)})_{T^-}$.
If $v \in \Gamma \setminus \bigcup_{T \in \mathcal{T}}T$, 
then $x^{(1)}_{W}, x^{(2)}_{W} \in \cL_{W}(Z)$ and so $\hat \rho(x^{(i)})_u = \tilde \Phi(\sigma_u(x^{(i)})_{W_1})$ for every $u \in W_2$ and $i=1,2$. Since $W_2$ is an injectivity window for $\hat \rho$ ,  there exists $u \in W_2$ such that $\rho(x^{(1)})_{v+u} \ne \rho(x^{(2)})_{v+u}$. Thus $\rho:X \to Y$ is indeed an injective map.
\end{proof}

\section{ Reduction to the finitely generated abelian case}\label{sec:reduction_to_fg}
We have already established our main result for the case where $\Gamma$ is a finitely generated abelian group.
Various results about subshifts of finite type over countable groups can be reduced to the finitely generated case. Sometimes the reduction is not completely trivial. For example, in  \cite{MR4607628} the problem of classification of countable groups that admit an aperiodic subshift has been reduced (in a non-trivial manner)  to the finitely generated case.
In this last section we complete the proof of the embedding theorem for subshifts over general countable abelian groups, by reducing it to the case where the acting group $\Gamma$ is a finitely generated abelian group.

In the following we will say that a sentence $P(\Delta)$ ``holds for all sufficiently large
$\Delta \le \Gamma$ if there exists a finite set $F\Subset \Gamma$ 
such that $P(\Delta)$ holds for any subgroup  
$\Delta \le \Gamma$  such that $F \Subset \Delta$.

If $\rho:X \to Y$ is a map between $\Gamma$-subshifts,
and $W_0 \Subset \Gamma$ is a coding window for $\rho$, 
then  the map $\rho$ naturally extends to a map $\rho^{[\Delta]}:X^{[\Delta]} \to Y^{[\Delta]}$ 
for any $\Delta \le \Gamma$ such that $W_0 \Subset \Delta$. 
Furthermore, if $\rho:X\to Y$ is injective, and $W_1 \Subset \Gamma$ is an injectivity window for $\rho$, then $\rho^{[\Delta]}:X^{[\Delta]} \to Y^{[\Delta]}$ is injective for any $\Delta \le \Gamma$ that contains $W_0\cup W_1$.
It follows that whenever $Z \subseteq Y$ is a closed, 
$\Gamma$-invariant set and $\hat \rho:Z \to Y$ is an injective map then $X \hookrightarrow_{\hat \rho} Y$  implies that $X^{[\Delta]} \hookrightarrow_{\hat \rho} Y^{[\Delta]}$ for all sufficiently large $\Delta \le \Gamma$ and $X \cong_{\hat \rho} Y$ implies that $X^{[\Delta]} \cong_{\hat \rho} Y^{[\Delta]}$ for all sufficiently large $\Delta \le \Gamma$.

A word of caution: If $X,Y$ are $\Gamma$ subshifts such that  $X \cong Y$, it is not necessarily true that $X^{[\Delta]} \cong Y^{[\Delta]}$ for all finitely generated subgroups $\Delta \le \Gamma$.  Moreover, if $Y= Y^{[\Delta]}$  for some finitely generated subgroup $\Delta \le \Gamma$ and  $X \cong Y$, then it does not follow that $X = X^{[\Delta]}$. However,  by the argument above in this case it is still true that $X^{[\tilde \Delta]} \cong Y^{[\tilde \Delta]}$ and that  $X= X^{[\tilde \Delta]}$ for all sufficiently large $\tilde \Delta \le \Gamma$:

\begin{lemma}\label{lem:reduction_to_fg}
    Suppose $X,Y$ are $\Gamma$-subshifts, that $Y$ has the map extension property, and that
    $\hat \rho:Z \to Y$ is an injective map. 
    If $E(X,Y,\hat \rho,\Gamma_0)$ holds for every $\Gamma_0 \le\Gamma$ then for every sufficiently large $\Delta \le \Gamma$ we have  $Y=Y^{[\Delta]}$ and  $E(X^{[\Delta]},Y^{[\Delta]},\hat \rho,\Gamma_0)$ holds for every subgroup  $\Gamma_0 \le \Gamma$.
\end{lemma}
\begin{proof}
    Let $X,Y,Z$ and $\hat \rho:Z \to Y$ be as in the statement.
    Because $Y$ is has the map extension property, by \Cref{prop:map_ext_contractible} it follows that $Y$ contractible so by \Cref{lem:contractible_Y_Delta} $\Delta_0 < \Gamma$ such that $Y= Y^{[\Delta_0]}$, and thus $Y= Y^{[\Delta]}$ for  every sufficiently large $\Delta \le \Gamma$.
    
    We will prove that for any $\Gamma_0 \in \Sub(\Gamma)$
    there exists an open neighborhood $\mathcal{U}_{\Gamma_0} \subseteq \Sub(\Gamma)$ with $\Gamma_0 \in \mathcal{U}_{\Gamma_0}$
   such that for all sufficiently large $\Delta \le \Gamma$ the condition
    $E(X^{[\Delta]},Y^{[\Delta]},\hat \rho,\Gamma_1)$
    holds for any $\Gamma_1 \in \mathcal{U}_{\Gamma_0}$. By compactness of $\Sub(\Gamma)$, it will follow that there exist a finite number of  finitely-generated subgroups $\Delta_1,\ldots, \Delta_N \le \Gamma$ such that for any $\Gamma_0 \le \Gamma$ there exists $1\le j \le N$ so that $E(X^{[\Delta]},Y^{[\Delta]},\hat \rho,\Gamma_0)$ holds for any finitely generated subgroup $\Delta \le \Gamma$ satisfying $\Delta_j \le \Delta$. The proof will be completed by taking $\Delta_0$ to be the group generated by $\{ \Delta_1,\ldots, \Delta_N\}$.  

    Let $\Gamma_0 \le \Gamma$ be an arbitrary subgroup. 

    We first show that there exists an open neighborhood $\mathcal{W}_{\Gamma_0} \subseteq \Sub(\Gamma)$ of $\Gamma_0$ such that $\Ker(Y^{[\Delta]}_{\Gamma_0}) \le \Ker(X^{[\Delta]}_{\Gamma_0})$ for any $\Gamma_1 \in \mathcal{W}_{\Gamma_0}$ and any sufficiently large subgroup $\Delta \le \Gamma$.
    If $Y^{[\Gamma_0]}$ consists of a unique fixed point, then $X^{[\Gamma_0]}$ must also consist of at most a single fixed point. In this case the set of subgroups $\mathcal{W} = \left\{ \Gamma_1 \in \Sub(\Gamma)~:~ | \cL_{\{0\}}(X_{\Gamma_1}) |  = 1 \right\}$ is an open neighborhood of $\Gamma_0$ that satisfies the requirement.
  
    Otherwise, since $Y=Y^{[\Delta_0]}$, it follows that   the group $\Ker(\overline{Y}_{\Gamma_0})$ is contained in $\Delta_0/(\Gamma_0 \cap \Delta_0)$, hence $\Ker(\overline{Y}_{\Gamma_0})$ is finitely generated. 
    It follows that for every $\Gamma_0$ there exist finite  subsets $F,K \Subset \Gamma$ with $F\cup \{0\} \subset K$ such that the condition that $\Ker(Y_{\Gamma_0}) \le \Ker(X_{\Gamma_0})$ is equivalent to the condition that for any
    $w \in \cL_{K}(X_{\Gamma_0})$ and any $v \in F$ we have $w_0 =w_v$. Let
    \[
    \mathcal{W}_{\Gamma_0}
   = \left\{ \Gamma_1 \in \Sub(\Gamma)~:~ \cL_K(X_{\Gamma_1})=\cL_K(X_{\Gamma_0})\right\}.\]
    
    Then $\mathcal{W}_{\Gamma_0} \subseteq \Sub(\Gamma)$ is an open neighborhood of $\Gamma_0$,
    and $\Ker(Y^{[\Delta]}_{\Gamma_0}) \le \Ker(X^{[\Delta]}_{\Gamma_0})$ 
    for any $\Gamma_1 \in \mathcal{W}_{\Gamma_0}$ and any  finitely generated subgroup $\Delta$ such that $\Delta_0 \le \Delta$ and $K \subseteq \Delta$.

    If $h(\overline{X}_{\Gamma_0}) < h(\overline{Y}_{\Gamma_0})$ then there exists a finite $\Gamma_0$-separated set $K \Subset \Gamma$ such that $\frac{1}{|K|}\log \left| \cL_{K}(X_{\Gamma_0}) \right|< h(\overline{Y}_{\Gamma_0})$.

    Since $Y = Y^{[\Delta_0]}$ it follows that $h(\overline{Y}_{\Gamma_1})=h(\overline{Y}_{\Gamma_0})$ whenever $\Gamma_1 \cap \Delta_0 = \Gamma_0 \cap \Delta_0$.
Note that For any $\Gamma_1 \le \Gamma_0$ such that 
$[\Gamma:\Gamma_1]= +\infty$  we have that $\cL(X_{\Gamma_0}) \subseteq \cL_K(X_{\Gamma_1})$. 
By a compactness argument (using that $\cL_K(X_{\Gamma_0})$ is a finite set) 
we can find thus find a finitely generated subgroup $\Delta_1 \le \Gamma$ such that for any $\Gamma_1 \le \Gamma$ 
with $\Gamma_1 \cap \Delta_1 = \Gamma_0 \cap \Delta_0$ we have $\cL(X_{\Gamma_0}) \subseteq \cL_{K}(X_{\Gamma_1})$. 
Let $\mathcal{U}_{\Gamma_0}$ denote the collection of subgroups $\Gamma_1 \in \mathcal{W}_{\Gamma_0}$  so that $\Gamma_1 \cap \Delta_0 = \Gamma_0 \cap \Delta_0$ and $\Gamma_1 \cap \Delta_1 = \Gamma_0 \cap \Delta_0$.

Then $\mathcal{U}_{\Gamma_0} \subseteq \Sub(\Gamma)$ is an open neighborhood of $\Gamma_0$ and for any finitely generated subgroup $\Delta \le \Gamma$ that contains $\Delta_0,\Delta_1$ and $K-K$ we have  
 $E(X^{[\Delta]},Y^{[\Delta]},\hat \rho,\Gamma_1)$
    for any $\Gamma_1 \in \mathcal{U}_{\Gamma_0}$ and  any subgroup $\Delta <\Gamma$ such that $\Delta_0 \le \Delta$ that satisfies $(K-K) \le \Delta$.

    It remains to 
    consider the case that $h(X_{\Gamma_0})$ is not strictly smaller than $h(Y_{\Gamma_0})$. In this case, since $E(X,Y,\hat \rho,\Gamma_0)$ holds, we have that $X_{\Gamma_0} \cong_{\hat \rho} Y_{\Gamma_0}$. 
    In other words, 
 there exists an isomorphism $\rho_{\Gamma_0}:X_{\Gamma_0} \to Y_{\Gamma_0}$ that extends $\hat \rho$. Let $W_1 \Subset \Gamma$ be a coding window for $\rho_{\Gamma_0}$, and let $W_2 \Subset \Gamma$ be an injectivity window  for $\rho_{\Gamma_0}$. Let $W = W_1 +W_2$. It follows that for any subgroup $\Gamma_1$ such that $\Gamma_1 \cap (W -W) = \Gamma_0 \cap (W -W)$ and any finitely generated subgroup $\tilde \Delta \le \Gamma$ such that $(W-W) \subseteq \Delta$,
    the same sliding block code defines an embedding 
    $X^{[\tilde \Delta]}_{\Gamma_1} \hookrightarrow_{\hat \rho} Y^{[\tilde \Delta\}]}_{\Gamma_1}$. The set $\mathcal{U}_{\Gamma_0}$ of all such subgroups is an open neighborhood of $\Gamma_0$, and $E(X^{[\Delta]},Y^{[\Delta]},\hat \rho,\Gamma_1)$ holds for any $\Gamma_1 \in \mathcal{U}_{\Gamma_0}$ and any finite generated subgroup $\Delta \le \Gamma$ with $(W-W) \subseteq \Delta$.
   
\end{proof}

To complete the proof of \Cref{thm:relative_embedding_thm} in case that $\Gamma$ is a countable abelian group,
not necessarily finitely generated, we use \Cref{lem:reduction_to_fg} as follows: Let $\Gamma$ be a countable ablain group, let  $X,Y$ be $\Gamma$-subshifts. Assume that $Y$ has the map extension property, that $Z \subseteq X$ 
is a closed $\Gamma$-invariant set which is either empty or relatively of finite type, 
and let $\hat \rho:Z \to Y$ be an injective map.
Suppose that $E(X,Y,\hat \rho,\Gamma_0)$ holds for any $\Gamma_0 \le \Gamma$.
Then by \Cref{lem:reduction_to_fg} we have  $Y=Y^{[\Delta]}$ and  $E(X^{[\Delta]},Y^{[\Delta]},\hat \rho,\Gamma_0)$ holds for every subgroup  $\Gamma_0 \le \Gamma$ and every  sufficiently large finitely generated subgroup $\Delta \le \Gamma$.
Also observe that for any sufficiently large finitely generated  subgroup $\Delta \le \Gamma$ we have that $E(X^{[\Delta]},Y^{[\Delta]},\hat \rho,\Gamma_0)$
holds for every $\Gamma_0 \le \Gamma$ if and only if $E(X^{(\Delta)},Y^{(\Delta)},\hat \rho,\Gamma_0)$
holds for every $\Gamma_0 \le \Delta$. Thus, by \Cref{prop:emb_inductive},
$X^{(\Delta)} \hookrightarrow_{\hat \rho} Y^{(\Delta)}$ for any sufficiently large finitely generated subgroup subgroup
$\Delta \le \Gamma$. This implies that $X^{[\Delta]} \hookrightarrow_{\hat \rho} Y^{[\Delta]}$ 
for any sufficiently large finitely generated subgroup subgroup $\Delta \le \Gamma$. 
Since $X \subseteq X^{[\Delta]}$ 
for any $\Delta \le \Gamma$ and $Y=Y^{[\Delta]}$ for all sufficiently large $\Delta$, it follows that $X \hookrightarrow_{\hat \rho} Y$.


\section{Lower entropy factors and the map extension property}\label{sec:lower_entropy_factors}

Recall that the condition $X\rightsquigarrow Y$ defined in \Cref{def:right_squigarrow} is necessary for the existence  of a map from $X$ into $Y$. Boyle's lower entropy factor theorem says it is a 
 necessary and sufficient condition for the existence of a factor map between irreducible $\Z$-subshifts of finite type \cite{MR753922}:

\begin{theorem}[Boyle's lower entropy factor theorem \cite{MR753922}]
    Let $X,Y$ be irreducible $\Z$-subshifts of finite type with $h(X)>h(Y)$. Then there exists a factor map from $X$ to $Y$ if and only if $X \rightsquigarrow Y$.
\end{theorem}

In contrast, there exist  topologically mixing $\Z^2$-subshifts with arbitrarily high entropy which do not factor onto any nontrivial full shift \cite{MR2645044}. However, various multidimensional analogs of Boyle's lower entropy factor theorem have been obtained for certain classes of $\Z^2$-subshifts  \cite{MR2645044,MR386685,MR2321105}.


In order to provide the context of previously known results about factors of subshifts of finite type we now recall more properties of subshifts.
The block gluing property was introduced in \cite{MR2645044}:
\begin{definition}
    A $\Z^d$-subshift $X$ is \emph{block gluing} if there exists a finite set $K \Subset \Z^d$ such that for any pair of $K$-disjoint hyperrecatangles $F_1,F_2 \Subset \Gamma$ and any $x^{(1)},x^{(2)} \in X$ there exists $x \in X$ such that $x_{F_1} = x^{(1)}_{F_1}$ and $x_{F_2} = x^{(2)}_{F_2}$.
\end{definition}

The finite extension property was introduced in \cite{MR386685}:
\begin{definition}\label{def:finite_ext_property_BMP}
    A $\Gamma$-subshift $Y \subseteq B^\Gamma$ has the \emph{finite extension property} if there exists  finite sets $W,F \Subset \Gamma$ so that for any $K \subseteq \Gamma$, if $x \in B^\Gamma$ satisfies that $\sigma_v(y)_W \in \cL_W(Y)$ for every $v \in K+F$ then there exists $y \in Y$ such that $y\mid_K= x\mid_K$.
\end{definition}
The above definition of the finite extension property  slightly differs from than the original formulation in  \cite{MR386685}.

The current state-of-the-art regarding lower entropy factors of $\Z^d$-subshifts of finite type is the following theorem:
\begin{theorem}[Briceno-Mcgoff-Pavlov \cite{MR386685}]\label{thm:BMP_factor}
    Let $X$ be a block gluing $\Z^d$-subshift and let $Y$ be a $\Z^d$-subshift of finite type  with a fixed point and the finite extension property such that $h(X) >h(Y)$. Then there exists a factor map from $X$ to $Y$.
\end{theorem}

We now prove the following variant:
\begin{theorem}\label{thm:lower_entropy_map_extension_property}
    Let $X$ be a block gluing $\Z^d$-subshift and let $Y$ be a $\Z^d$-subshift with the map extension property  such that $h(X) >h(Y)$. Then there exists a factor map from $X$ to $Y$ if and only if $X \rightsquigarrow Y$.
\end{theorem}

\begin{proof}
    As we have already explained, in general, if there exists a factor map from $X$ to $Y$, then  $X \rightsquigarrow Y$. 
     Let $X$ be a block gluing $\Z^d$-subshift  and let $Y \subseteq B^{\Z^d}$ be a $\Z^d$-subshift of finite type  with the map extension property  such that $h(X) >h(Y)$. 

     Now suppose that $X \rightsquigarrow Y$.
    Let $B_n= \{-n,\ldots,n\}^d \Subset \Z^d$. 
    Because $Y$ has the map extension property, there exists a finite set of subgroups $\mathcal{G} \Subset \Sub(\Z^d)$ such that $Y$ is an absolute retract for the class of $\mathcal{G}$-free subshifts. In particular, $Y$ is $\mathcal{G}$-free, and so there exists $M \in \mathbb{N}$ so that $B_M$ witnesses that $Y$ is $\mathcal{G}$-free.
    Because $X$ is block gluing, there exists $N \in \mathbb{N}$ so that for every $n \in \mathbb{N}$ and $f:(2n+2N)\Z^d \to \cL_{B_n}(X)$ there exists $x \in X$ such that 
    $\sigma_v(x)_{B_n} = f(v)$ for all $v \in (2n+2N)\Z^d$. 
    Given $w \in \cL_{B_{3N}}(X)$ and $n > 3N$, let $W_{w,n}$ denote the set of patterns $\tilde w \in \cL_{B_n}(X)$ such that $w_{B_{3N}}=w$ and so that the only overlap of $w$ with $\tilde w$ in $B_{n+N}$ occurs at $0$. Similarly to the proof \Cref{lem:marker_pattern_basic} one can show using  $h(X) > h(Y)$ and the fact that $X$ is block gluing, that for sufficiently large $n$ there exists $w \in \cL_{B_{3N}}$ such that $|W_{w,n}| > |\cL_{B_{n+N}}(Y)|$. Choose  $n \in \mathbb{N}$ and $w \in \cL_{B_{3N}}$ satisfying the above, with $n \ge 3M$.
     Let $\Phi:W_{w,n} \to \cL_{B_{n+N}}(Y)$ be a surjective function and 
    let
    \[
    \tilde X =\left\{ x\in X~:~ \exists v\in \Z^d \mbox{ s.t. }\forall u \in (2n+2N)\Z^d \sigma_u(x)_{B_n} \in W_{w,n} \right\}.
    \]
    It follows that for any $x \in \tilde X$, there exists a unique coset of $(2n+2N)\Z^d$ where the pattern $w$ occurs.
    Define a map $\tilde \rho:\tilde X \to B^{\Z^d}$ by declaring $\sigma_v(\tilde \rho(x))_{B_{n+N}} = \Phi(\sigma_v(x)_{B_n})$ whenever $\sigma_v(x)_{B_N}=w$. Since the occurrences of $w$ in any $x\in \tilde X$ coincide with a coset of $(2n+2N)\Z^d$, the map $\sigma_v$ is well-defined. Let $\hat Y= \tilde \rho(\tilde X)$. 
    Because $B_M$  witnesses that $Y$ is $\mathcal{G}$-free and $M < 3N$, it follows that $\hat Y$ is $\mathcal{G}$-free.
    We claim that $Y \subseteq \hat Y$. Indeed, given any $y \in Y$, choose $f:(2n+2N)\Z^d \to W_{w,n}$ so that
    $\Phi(f(v))= \sigma_v(y)_{B_{n+N}}$ for every $v \in (2n+2N)\Z^d$.  Now find $x \in X$ such that $\sigma_v(x)_{B_n}= f(v)$ for every $v \in (2n+2N)\Z^d$.
    Then $x \in \tilde X$ and $\hat \rho(x) =y$. Because $Y$ is an absolute retract for the class of $\mathcal{G}$-free subshifts there exists a retraction map  $r:\hat Y \to Y$. Let $\tilde \pi:\tilde X \to Y$ be defined by $\tilde \pi = r \circ \hat \rho$. Since $r$ is a retraction and $Y \subseteq \hat \rho(\tilde X)$, it follows that $\tilde \pi(\tilde X)=Y$.
    By the map extension property of $Y$, since  $X \rightsquigarrow Y$, the map $\tilde \pi$ extends to a map $\pi:X \to Y$. Since $\tilde \pi(\tilde X)=Y$ it follows that $\pi(X)=Y$, so $Y$ is a factor map.
\end{proof}

Poirier and  Salo showed that 
the map extension property  implies the finite extension property \cite{poirier2024contractible}. 
This fact is a strengthening of \Cref{prop:abs_retract_prop}, because the finite extension property implies strong irreducibility and finite type.
However, the finite extension property does not imply the map extension property, so \Cref{thm:lower_entropy_map_extension_property} does not completely recover \Cref{thm:BMP_factor}.
Furthermore, for $\Z^2$-subsfhits  the map extension property does not follow from a stronger property called \emph{topological strong spatial mixing}, introduced in \cite{MR3819997}:

\begin{definition}
A subshift $X \Subset A^\Gamma$ 
is  \emph{topologically strong spacial mixing}
\cite{MR3819997} 
(abbreviated by \emph{TSSM})  
if there exists a finite set $K \Subset \Gamma$ so that for any subsets $F,F_1,F_2 \Subset \Gamma$ 
such that $F_1,F_2 \Subset \Gamma$
are $K$-disjoint and any $x^{(1)},x^{(2)} \in X$ 
such that $x^{(1)}_F = x^{(2)}_F$ there exists $x \in X$ such that
$x_F= x^{(1)}_F = x^{(2)}_F$, 
$x_{F_1} = x^{(1)}_{F_1}$ and $x_{F_2} = x^{(2)}_{F_2}$.
\end{definition}

It was shown in \cite{MR386685} that topologically strong spacial mixing subshifts have the finite extension property.

In personal communication Raimundo Briceno and Alvaro Bustos have constructed a simple example of a topologically  strong spacial mixing $\Z^2$-subshift $X$ such that $\overline{X}_{\Gamma_0}$ 
is not strongly irreducible 
where $\Gamma_0 = \Z \times \{0\} < \Z^2$.
In particular, this example demonstrates that $X$
 the finite extension property does not imply the map extension property.
Here is a version of the construction with $\Z^2$ replaced by $\Gamma= \Z \times (\Z/2\Z)$: Let $\Gamma_0 = \{0\} \times (\Z /2\Z)< \Gamma$. Let $X \subseteq \{0,1\}^\Gamma$ be the subshift of finite type defined by the following rule: For every $v \in \Gamma$ if $x_v=x_{v+(0,1)}$ and $x_{v+(1,0)}=x_{v+(1,1)}$ then $x_{v}=x_{v+(1,0)}$.
Then it not difficult to check that $X$ is topologically  strong spacial mixing, but $X_{\Gamma_0}$ consists of two fixed points. 

In view of the above discussion, \Cref{thm:lower_entropy_map_extension_property} does not completely recover \Cref{thm:BMP_factor} of Briceno, Mcgoff and Pavlov. However, \Cref{thm:lower_entropy_map_extension_property} does address some new cases where the target subshift $Y$ does not have a fixed point. For instance, \Cref{thm:lower_entropy_map_extension_property} characterizes the (greater entropy) block-gluing subshifts that factor onto the $5$-colorings of $\Z^2$.




\bibliographystyle{amsplain}
\bibliography{refs}

\providecommand{\bysame}{\leavevmode\hbox to3em{\hrulefill}\thinspace}
\providecommand{\MR}{\relax\ifhmode\unskip\space\fi MR }
\providecommand{\MRhref}[2]{%
  \href{http://www.ams.org/mathscinet-getitem?mr=#1}{#2}
}
\providecommand{\href}[2]{#2}
\begin{thebibliography}{10}

\bibitem{MR4247630}
Noga Alon, Raimundo Brice\~{n}o, Nishant Chandgotia, Alexander Magazinov, and
  Yinon Spinka, \emph{Mixing properties of colourings of the {$\Bbb{Z}^d$}
  lattice}, Combin. Probab. Comput. \textbf{30} (2021), no.~3, 360--373.
  \MR{4247630}

\bibitem{MR1951240}
Paul Balister, B\'{e}la Bollob\'{a}s, and Anthony Quas, \emph{Entropy along
  convex shapes, random tilings and shifts of finite type}, Illinois J. Math.
  \textbf{46} (2002), no.~3, 781--795. \MR{1951240}

\bibitem{MR4607628}
Sebasti\'{a}n Barbieri, \emph{Aperiodic subshifts of finite type on groups
  which are not finitely generated}, Proc. Amer. Math. Soc. \textbf{151}
  (2023), no.~9, 3839--3843. \MR{4607628}

\bibitem{bland2022embedding}
Robert Bland, \emph{An embedding theorem for subshifts over amenable groups
  with the comparison property}, Ergodic Theory Dynam. Systems \textbf{44}
  (2024), no.~11, 3155--3185. \MR{4803661}

\bibitem{Borsuk1931}
Karol Borsuk, \emph{Sur les rétractes}, Fundamenta Mathematicae \textbf{17}
  (1931), no.~1, 152--170 (fre).

\bibitem{MR753922}
Mike Boyle, \emph{Lower entropy factors of sofic systems}, Ergodic Theory
  Dynam. Systems \textbf{3} (1983), no.~4, 541--557. \MR{753922}

\bibitem{MR2645044}
Mike Boyle, Ronnie Pavlov, and Michael Schraudner, \emph{Multidimensional sofic
  shifts without separation and their factors}, Trans. Amer. Math. Soc.
  \textbf{362} (2010), no.~9, 4617--4653. \MR{2645044}

\bibitem{MR3819997}
Raimundo Brice\~{n}o, \emph{The topological strong spatial mixing property and
  new conditions for pressure approximation}, Ergodic Theory Dynam. Systems
  \textbf{38} (2018), no.~5, 1658--1696. \MR{3819997}

\bibitem{MR386685}
Raimundo Brice\~{n}o, Kevin McGoff, and Ronnie Pavlov, \emph{Factoring onto
  {$\Bbb{Z}^d$} subshifts with the finite extension property}, Proc. Amer.
  Math. Soc. \textbf{146} (2018), no.~12, 5129--5140. \MR{3866852}

\bibitem{MR4311117}
Nishant Chandgotia and Tom Meyerovitch, \emph{Borel subsystems and ergodic
  universality for compact {$\Bbb Z^d$}-systems via specification and beyond},
  Proc. Lond. Math. Soc. (3) \textbf{123} (2021), no.~3, 231--312. \MR{4311117}

\bibitem{MR2321105}
Angela Desai, \emph{Subsystem entropy for {$\Bbb Z^d$} sofic shifts}, Indag.
  Math. (N.S.) \textbf{17} (2006), no.~3, 353--359. \MR{2321105}

\bibitem{downarowicz2016shearer}
Tomasz Downarowicz, Bartosz Frej, and Pierre-Paul Romagnoli, \emph{Shearer's
  inequality and infimum rule for {S}hannon entropy and topological entropy},
  Dynamics and numbers, Contemp. Math., vol. 669, Amer. Math. Soc., Providence,
  RI, 2016, pp.~63--75. \MR{3546663}

\bibitem{MR4539364}
Tomasz Downarowicz and Guohua Zhang, \emph{Symbolic extensions of amenable
  group actions and the comparison property}, Mem. Amer. Math. Soc.
  \textbf{281} (2023), no.~1390, vi+95. \MR{4539364}

\bibitem{MR3359054}
Su~Gao and Steve Jackson, \emph{Countable abelian group actions and hyperfinite
  equivalence relations}, Invent. Math. \textbf{201} (2015), no.~1, 309--383.
  \MR{3359054}

\bibitem{MR3540453}
Yonatan Gutman, Elon Lindenstrauss, and Masaki Tsukamoto, \emph{Mean dimension
  of {$\Bbb{Z}^k$}-actions}, Geom. Funct. Anal. \textbf{26} (2016), no.~3,
  778--817. \MR{3540453}

\bibitem{MR2643712}
Michael Hochman, \emph{On the automorphism groups of multidimensional shifts of
  finite type}, Ergodic Theory Dynam. Systems \textbf{30} (2010), no.~3,
  809--840. \MR{2643712}

\bibitem{Korkine1873}
Korkine, \emph{Sur les formes quadratiques (zus. mit zolotareff)},
  Mathematische Annalen \textbf{6} (1873), 366--389 (fre).

\bibitem{MR693975}
Wolfgang Krieger, \emph{On the subsystems of topological {M}arkov chains},
  Ergodic Theory Dynam. Systems \textbf{2} (1982), no.~2, 195--202 (1983).
  \MR{693975}

\bibitem{MR1972240}
Samuel~J. Lightwood, \emph{Morphisms from non-periodic {$\Bbb Z^2$}-subshifts.
  {I}. {C}onstructing embeddings from homomorphisms}, Ergodic Theory Dynam.
  Systems \textbf{23} (2003), no.~2, 587--609. \MR{1972240}

\bibitem{MR2085910}
\bysame, \emph{Morphisms from non-periodic {$\Bbb Z^2$} subshifts. {II}.
  {C}onstructing homomorphisms to square-filling mixing shifts of finite type},
  Ergodic Theory Dynam. Systems \textbf{24} (2004), no.~4, 1227--1260.
  \MR{2085910}

\bibitem{MR1674915}
Sibe Marde\v{s}i\'{c}, \emph{Absolute neighborhood retracts and shape theory},
  History of topology, North-Holland, Amsterdam, 1999, pp.~241--269.
  \MR{1674915}

\bibitem{MR4311113}
Kevin McGoff and Ronnie Pavlov, \emph{Ubiquity of entropies of intermediate
  factors}, J. Lond. Math. Soc. (2) \textbf{104} (2021), no.~2, 865--885.
  \MR{4311113}

\bibitem{poirier2024contractible}
Leo Poirier and Ville Salo, \emph{Contractible subshifts}, arXiv preprint
  arXiv:2401.16774 (2024).

\bibitem{MR1764932}
Anthony~N. Quas and Paul~B. Trow, \emph{Subshifts of multi-dimensional shifts
  of finite type}, Ergodic Theory Dynam. Systems \textbf{20} (2000), no.~3,
  859--874. \MR{1764932}

\bibitem{MR1307966}
Thomas Ward, \emph{Automorphisms of {$\bold Z^d$}-subshifts of finite type},
  Indag. Math. (N.S.) \textbf{5} (1994), no.~4, 495--504. \MR{1307966}

\end{thebibliography}
\end{document}